\documentclass[graybox]{svmult}
\usepackage[utf8]{inputenc}
\usepackage[english]{babel}
\usepackage{amsmath}
\usepackage{amssymb}
\usepackage[all,cmtip]{xy}
\usepackage{graphicx}
\usepackage{bookmark}
\usepackage{tikz-cd}
\usepackage{longtable}
\usepackage{enumerate}
\usepackage{multicol}
\usepackage{pdflscape}
\usepackage{multirow}
\usepackage{adjustbox}
\usepackage{array}
\usepackage[shortlabels]{enumitem}

\usepackage{type1cm} 
\usepackage[bottom]{footmisc}
\usepackage{newtxtext} 
\usepackage{newtxmath}

\newcommand{\QQ}{\mathbb{Q}}

\newcommand{\CC}{\mathbb{C}}
\newcommand{\PP}{\mathbb{P}}

\newcommand{\GL}{\mathop{\rm GL}\nolimits}

\newcommand{\OO}{\mathcal{O}}
\newcommand{\Spec}{\mathop{\rm Spec}}

\newcommand{\Aut}{\mathop{\rm Aut}\nolimits}

\newcommand{\Char}{\mathop{\rm char}\nolimits}

\newcommand{\M}{\mathcal M}
\newcommand{\X}{\mathcal X}
\newcommand{\Y}{\mathcal Y}

\newcommand{\Kb}{\overline{K}}
\newcommand{\Yb}{\overline{Y}}
\newcommand{\Xb}{\overline{X}}
\newcommand{\Db}{\overline{D}}

\newcommand{\fb}{\overline{f}}

\newcommand{\Hb}{\overline{\mathcal H}}
\newcommand{\Mb}{\overline{\mathcal M}}

\newcommand{\Iiii}{I_3}
\newcommand{\Ivi}{I_6}

\usepackage{tikz}
\usetikzlibrary{hobby,calc}
\usepackage{rotating}
\usepackage[caption=false]{subfig}

\newcommand{\one}{$1$}
\newcommand{\two}{$2$}
\newcommand{\three}{$3$}

\def\SageMath{SageMath~\cite{SageMath}}

\begin{document}

\title*{Reduction types  of genus-3 curves in a special stratum of their
	moduli  space}
\author{Irene Bouw, Nirvana Coppola, P{\i}nar K{\i}l{\i}\c{c}er, Sabrina Kunzweiler, Elisa Lorenzo Garc\'ia and Anna Somoza}

\authorrunning{I. Bouw, N. Coppola, P. K{\i}l{\i}\c{c}er, S. Kunzweiler, E. Lorenzo Garc\'ia and A. Somoza}

\institute{Irene Bouw \at Universität Ulm, Institut für Reine Mathematik, D-89081 Ulm, Germany, \email{irene.bouw@uni-ulm.de} \and Nirvana Coppola \at University of Bristol, School of Mathematics, BS8 1UG Bristol, United Kingdom, \email{nc17051@bristol.ac.uk} \and P{\i}nar K{\i}l{\i}\c{c}er \at Bernoulli Institute for Mathematics, Computer 
  Science and Artificial Intelligence, 9747 AG Groningen, Netherlands, \email{p.kilicer@rug.nl}
\and  Sabrina Kunzweiler \at {Universität Ulm, Institut für Reine Mathematik, 
  D-89081 Ulm, Germany,} \email{sabrina.kunzweiler@uni-ulm.de}
\and  Elisa Lorenzo Garc\'ia \at {Univ Rennes, CNRS, IRMAR - UMR 6625, F-35000
  Rennes, France.\\ Institut de Math\'ematiques, Universit\'e de Neuch\^atel, Rue Emile-Argand 11, 2000, Neuch\^atel, Switzerland,} \email{elisa.lorenzogarcia@univ-rennes1.fr} \and  Anna Somoza \at {Univ Rennes, CNRS, IRMAR - UMR 6625, F-35000 Rennes, France,} \email{anna.somoza@univ-rennes1.fr}
}

\maketitle

\abstract{
We study a $3$-dimensional stratum $\M_{3,V}$ of the moduli space
$\M_3$ of curves of genus $3$ parameterizing curves $Y$ that admit a
certain action of $V = C_2\times C_2$. We determine the possible
types of the stable reduction of these curves to characteristic
different from $2$. We define invariants for $\M_{3,V}$ and
characterize the occurrence of each of the reduction types in terms of 
them. We also calculate the $j$-invariant (respectively~the Igusa invariants)
of the irreducible components of positive genus of the stable
reduction~$Y$ in terms of the invariants.
}


\section{Introduction}

Let $(K, \nu)$ be a discrete valuation field, $\OO$ its ring of
integers, and $k$ its residue field. Let~$Y$ be a smooth projective
and absolutely irreducible curve over $K$ of genus $g(Y)\geq 1$.  A
theorem of Deligne--Mumford \cite{DM} states that after replacing
$K$ by a finite extension there exists a semistable model
$\mathcal{Y}$ over $\Spec(\OO)$ (see Definition~\ref{def:goodred}). If
$g(Y)\geq 2$ there exists a unique minimal semistable model, which we
call the \emph{stable model}.  Its special fiber $\Yb$ is called
the \emph{stable reduction} of $Y$.  For fixed genus $g(Y)\geq 2$ there
are only finitely many possibilities for the \emph{reduction type},
i.e.,~the graph of irreducible components of $\Yb$ together with the
genus of the normalization of each irreducible component. If $\Yb$ is smooth we say that $Y$ has \emph{potentially good reduction}.

In the case that $g(Y)=1$ a minimal semistable model does not need to
be unique, but the reduction type does not depend on the choice of a
minimal semistable model. It is determined by the $j$-invariant $j(Y)$
of $Y$. Namely, the special fiber $\Yb$ of any minimal semistable
model of $Y$ is smooth if and only if the valuation $\nu(j(Y))$ is
non-negative and $\Yb$ is a projective line that intersects itself in
one point (multiplicative reduction) otherwise. (See for
example \cite[Chapter VII, Prop.~5.5]{Silverman}).

In \cite[Th\'eor\`eme 1]{Liu} Liu has generalized this result to
curves of genus $2$. He determines the reduction types in terms of the
Igusa invariants. Moreover, Liu gives an expression for the
$j$-invariant of the irreducible components of positive genus in terms
of the Igusa invariants in the case that $\Yb$ is not smooth. (If
$\Yb$ is singular, then the irreducible components of positive genus are
necessarily elliptic curves.)

Smooth projective curves of genus $3$ are either hyperelliptic or
plane quartics, where the latter form a dense open subset of the
moduli space $\M_3$ of curves of genus~$3$.  The Dixmier--Ohno
invariants for smooth quartics form a set of parameters on the moduli space, similar to the Igusa
invariants for curves of genus $2$ and the $j$-invariant for elliptic
curves.  The Dixmier--Ohno invariants do not extend to the locus of
smooth hyperelliptic curves, and one should consider the locus in
$\M_3$ of smooth hyperelliptic curves as a part of the boundary of the
locus of smooth quartics in the Deligne--Mumford compactification
$\overline{\M_3}$.  There  exists a different set of invariants, called the
Shioda invariants, for smooth hyperelliptic curves of genus $3$.

It is natural to ask whether one can characterize the reduction types
of smooth plane quartics in terms of the Dixmier--Ohno invariants.
The first result in this direction is proved in \cite{LLLR}. The
authors of loc.~cit.~characterize when the stable reduction $\Yb$ of a smooth plane quartic
is a smooth hyperelliptic curve in terms of the Dixmier--Ohno
invariants. The authors of \cite{LLLR} also
compute the Shioda invariants of the smooth hyperelliptic curve $\Yb$.

An alternative direction to generalize the result of Liu is to
consider families of superelliptic curves, which are cyclic covers
$f:Y\to \PP^1_K$ of the projective line. In the case that either
$\Char(k)\nmid \deg(f)$ or $\Char(k)=\deg(f)=p$ there exist algorithms
to compute the stable reduction $\Yb$ (see for example \cite{BW} for
the general case or \cite{DDMM}  for hyperelliptic curves). However,
this approach does not yield a natural interpretation in terms of the
Dixmier--Ohno invariants. An interesting special case is the case of
Picard curves, which are both plane quartics and superelliptic
curves. In this case the possible reduction types can be found
in \cite{BBW} and \cite{BKSW}. The case of Picard curves with
potentially good reduction in terms of the Dixmier--Ohno invariants is
described in \cite[Section 4.2]{LLLR}.

The motivating question for the current paper is whether it is
possible to characterize the reduction type of a smooth plane quartic in
terms of the Dixmier--Ohno invariants also in the case that the stable
reduction $\Yb$ is not smooth. Due to the many different reduction
types we restrict to a specific $3$-dimensional stratum in
$\M_3$. In this paper we determine all possible reduction types
for this stratum and prove a result analogous to Liu's result for
curves of genus $2$. 

\bigskip
We now describe our results in more detail. In the rest of the
introduction we assume that both $K$ and $k$ have characteristic
different from $2$.  We consider the stratum $\M_{3,V}\subset \M_3$
consisting of smooth curves $Y/K$ of genus $3$ such that
$\Aut_{\Kb}(Y)$ contains a subgroup isomorphic to $V = C_2\times C_2$ with
$g(Y/V)=0$ and such that all degree-$2$ subcovers of $Y\to X:=Y/V$
have genus $1$.

The study of these curves goes back to Ciani in 1899
\cite{ciani}. They are sometimes called Ciani surfaces in his
honor. More recent references are \cite{LR}
and \cite{HLP}. In \cite{HLP} the authors find curves from this
family with many rational points and small conductor (see also
Example~\ref{exa:HLP}).

A dense open set of $\M_{3,V}$ parametrizes smooth plane quartics with an
action of the Klein $4$-group $V$.  Over a sufficiently large
field, these curves may be described by an explicit
quartic equation in standard form (see Lemma~\ref{lem:quartic_eq}),
and one can express the Dixmier--Ohno invariants in terms of the
coefficients of this equation. In Section~\ref{sec:invariants_quartic}
we replace the full set of
Dixmier--Ohno invariants by a smaller set $I_3, I_3', I_3'',
I_6$ of invariants, which are easier to handle in our
set-up. They may be expressed in terms of the Dixmier--Ohno
invariants (see \cite{Coppola}). Our main results are formulated in terms of these invariants in Section~\ref{sec:main_results}.

All curves $Y$ parametrized by $\M_{3,V}$ admit a $V$-Galois cover
$f:Y\to \PP^1_K$. This allows us to calculate the possible reduction
types by extending the method of \cite{BW} to our situation.  The key
ideas of this method in our set-up are explained in
Sections~\ref{sec:stable}---\ref{sec:compute}.

Combining these two approaches, we show in Theorem~\ref{FromGraphtoRedType} that there are exactly $13$
reduction types for curves parametrized by $\M_{3,V}$. The different types are
illustrated in Appendix~B. Proposition~\ref{prop:potgoodredquartics} characterizes potentially good quartic reduction. The analogous result in the case of potentially good hyperelliptic reduction can be found in  Proposition~\ref{prop:goodhyp}.
In Theorems~\ref{THM:Main_NonDegConic}
and~\ref{THM:Main_DegConic} we give explicit conditions in
terms of the invariants $I_3, I_3', I_3'', I_6$ characterizing the
different reduction types.  Moreover, we determine the Igusa
invariants (respectively~the $j$-invariant) of the irreducible components of
positive genus in the case that $\Yb$ is singular
(see the proofs in Section~\ref{sec:main_proofs}). In Section~\ref{Sec:Hyper} we discuss the case
where the curve $Y$ is hyperelliptic. Theorem~\ref{THM:hyperelliptic} is the result analogous to Theorems~\ref{THM:Main_NonDegConic}
and~\ref{THM:Main_DegConic} in the hyperelliptic case, and it is
phrased in terms of a modified set of invariants described in Section~\ref{SSec:invariants}. In \cite{Coppola}, an implementation in \SageMath~of the results in Theorems~\ref{THM:Main_NonDegConic}
and~\ref{THM:Main_DegConic} is given. It takes as input the coefficients of the plane quartic and outputs the reduction type.

\begin{acknowledgement}
	This project began at the Women in Numbers Europe 3 workshop in Rennes, August 2019. We are grateful to the organizers for bringing us together and providing us with an excellent working environment to get this project underway. We thank Christophe Ritzenthaler for his ideas for the proofs of Propositions \ref{Prop:Invariants_Generate}, \ref{prop:potgoodredquartics}, \ref{prop:invhyp} and \ref{prop:goodhyp}. We thank Raymond van Bommel and Andreas Pieper for pointing out a mistake in Table \ref{tab:correspondence} of the published version of this paper.
\end{acknowledgement}

\subsection{Notation} \(\,\)
\begin{longtable}{r@{\hspace*{1em}}p{.8\textwidth}}
	\hline
	\textbf{Symbol} & \textbf{Meaning and place of definition} \\
	\hline
	\endfirsthead
	\multicolumn{2}{c}%
	{\textit{Continued from previous page}} \\
	\hline
	\textbf{Symbol} & \textbf{Meaning and place of definition} \\
	\hline
	\endhead
	\hline \multicolumn{2}{c}{\textit{Continued on next page}} \\
	\endfoot
	\hline
	\endlastfoot
	\centering
	$K$            & a complete discrete valuation field of characteristic $\neq 2$                        \\
	$\nu$           & a valuation of $K$   \\
	$\Kb$           & an algebraic closure of $K$                                                         \\
	$\OO$           & the ring of integers of $K$ with uniformizer $\pi$  
	\\
	$k$       & the residue field of $\OO$ of characteristic $\neq 2$                               \\
	$V$            & Klein 4-group $C_2 \times C_2$                                                      \\
	$Y$            & a smooth curve of genus $3$ defined over $K$                                        \\
	$X$            & a conic defined over $K$                                                            \\
	$\Delta(C)$        & the discriminant of a plane curve $C$                                               \\
	$\mathcal{Y}$  & a model of $Y$           \\
	$\mathcal{X}$  & a model of $X$           \\
	$\Yb$           & the special fiber of $\mathcal{Y}$                                        \\
	$\Xb$           & the special fiber of $\mathcal{X}$                                        \\
	$g$            & the genus of a smooth projective curve                                              \\
	$\M_{g}$         & the moduli space of smooth projective curves of genus $g$                           \\
	$\Mb_{g}$         & the Deligne--Mumford compactification of $\M_g$                                      \\
	$\M_{3,V}$        & the stratum of genus $3$ curves in $\M_3$ such that $\Aut_{\Kb}(Y)$ contains a 
	subgroup isomorphic to $V$ with $g(Y/V)=0$ and such that all degree-$2$ 
	subcovers have genus $1$ (\S~\ref{Ssec:setup})                                            			   \\
	$\Mb_{3, V}$       & the closure of $\M_{3,V}$ in $\Mb_{3}$  (\S~\ref{Ssec:setup})                                                                                                                                       \\
	$\M_{3,V}^{\text{hyp}}$  & the hyperelliptic stratum in  $\M_{3,V}$ (\S~\ref{Ssec:setup})                                                                                         \\
	$\M_{3,V}^{\text{quar}}$ & the plane quartic locus in  $\M_{3,V}$ (\S~\ref{Ssec:setup})                                                                                         \\ 
\end{longtable}
\addtocounter{table}{-1}

\section{Stable reduction and admissible covers}

\subsection{The set-up}\label{Ssec:setup} Let $K$ be a discrete valuation field of characteristic $p\geq0$  different from $2$. Let $\Kb$ be a fixed algebraic closure of $K$.  We denote by $\M_g$ the moduli space of smooth projective
curves of genus $g$ defined over~$K$.  In this paper we consider the
locus $\M_{3, V}$ of curves of genus $3$ such that $\Aut_{\Kb}(Y)$
contains a subgroup isomorphic to the Klein $4$-group $V =
C_2\times C_2$ such that $X:=Y/V$ has genus~$0$ and each of the
intermediate covers of degree $2$ has genus $1$. Since we assume
$\Char(K)\neq 2$, the cover $f:Y\to X$ is tamely ramified.  The
following lemma states some elementary facts on such covers.

\begin{lemma}\label{lem:Galois} Let $Y/\Kb$ be a smooth projective
	curve of genus $3$, and $f:Y\to X\simeq \PP^1_{\Kb}$ a Galois cover
	with Galois group isomorphic to $V$ such that all intermediate covers
	of degree $2$ have genus~$1$.
	\begin{enumerate}[(1)]
		\item  Every element of order~$2$ in $V$ generates the inertia group of exactly $2$ branch points of~$f$. 
		\item If $Y$ is hyperelliptic, then $\Aut_{\Kb}(Y)$ contains a subgroup $A\simeq (C_2)^3$. 
		The hyperelliptic involution $\iota$ acts on $X$
		by interchanging the two branch points of $f$ with the same inertia
		generator. The cover $Y\to Y/A$ of degree $8$ is branched at $5$
		points. Two of these have inertia generator $\iota$, the other three
		one of the elements of order~$2$ contained in a unique subgroup
		$V\subset A$.
		\item The stratum  $\M_{3,V}$ is irreducible and of dimension $3$. The
		hyperelliptic curves form an irreducible substratum
		$\M_{3,V}^{\emph{hyp}}$ of dimension $2$. 
	\end{enumerate}
\end{lemma}

\begin{proof}
	Statement (1) follows from the Riemann--Hurwitz formula applied to $f$
	and each of its subcovers of degree $2$. Statement (2) follows
	similarly by considering the cover $Y\to Y/A$ and its subcovers.
	Statement (3) is well-known, see e.g.~\cite[Thm.~3.1]{LRRS}. 
\hfill $\qed$
\end{proof}

We denote by $\M_{3,V}^{\text{quar}}=\M_{3,V}\setminus
\M_{3,V}^{\text{hyp}}$ the locus of smooth plane quartics with $V\subset \Aut_{\Kb}(Y)$.  We write $\Mb_3$ for the Deligne--Mumford
compactification of $\M_3$ and $\Mb_{3, V}$ for the closure of
$\M_{3,V}$ in $\Mb_3$. A $\Kb$-point of the boundary of $\Mb_{3, V}$
consists of a stable curve~$\Yb$ of genus $3$ on which $V$ acts
faithfully with a quotient of genus zero.  The goal of this paper is
to study the ``types'' of stable curves occurring in the boundary of~$\Mb_{3, V}$. Rather than working out necessary and sufficient
conditions for a $V$-action on a stable curve $\Yb$ of genus $3$ to
correspond to a point of this boundary, we determine instead the
compactification $\Hb_{3,V}$ of the Hurwitz space parametrizing
$V$-Galois covers $Y\to X\simeq \PP^1_{\Kb}$ as in
Lemma~\ref{lem:Galois}. We refer to \cite{RW} for precise definitions
and properties of Hurwitz spaces. There are different variants, but
for our purposes it is not necessary to specify which one we use.

Lemma~\ref{lem:Galois}.(3) implies that to determine the types of
stable curves occurring in the boundary of $\Mb_{3, V}$ it suffices to
consider the possible degenerations of the non-hyperelliptic
curves of genus~$3$, i.e.,~smooth plane quartics. 

The non-hyperelliptic curves of genus $3$ with non-trivial automorphism
group (over $\mathbb{C}$) have been classified by Vermeulen
\cite{Vermeulen} (see also \cite{Hen}). In \cite{LRRS} it is shown
that this classification also holds in positive characteristic with a
few exceptional cases in small characteristic.  The following result
is a special case of this classification. Actually, one may
additionally assume that the parameters $A, B, C$ are equal to $1$. 

\begin{lemma}\label{lem:quartic_eq} Let $Y/\Kb$ be a smooth plane quartic such 
	that there exists a subgroup $V\subseteq \Aut_{\Kb}(Y)$.
	\begin{enumerate}[(1)]
		\item  Then $Y$ may be defined by an equation
		\begin{equation}\label{eq: F}
		Y:\; F:=Ax^4+By^4+Cz^4+ay^2z^2+bx^2z^2+cx^2y^2 =0
		\end{equation}
		and the non-trivial elements of $V$ act as $ (x:y:z)\mapsto (\pm x:\pm y:z)$.
		\item  The ramification points of $f:Y\to Y/V=:X$ are the points with $xyz=0$.
	\end{enumerate}
\end{lemma}

\begin{proof}
	Statement (1) follows from the classification of non-hyperelliptic
	curves of genus~$3$ with non-trivial automorphism group
	\cite{LRRS}. Note that for a plane quartic $Y$ the existence of a
	subgroup of $\Aut_{\Kb}(Y)$ isomorphic to $V$ already implies that
	$g(Y/V)=0$ and that the degree-$2$ subcovers have genus $1$. 
	Statement (2) follows by direct verification.
   \hfill $\qed$ 
   \end{proof}

In the rest of the paper, if $Y$ is a smooth plane quartic, then we denote the non-trivial elements of $V$ as 
\begin{equation*}
\begin{aligned}
\sigma_a((x:y:z))=&\,(-x:y:z),\\
\sigma_b((x:y:z))=&\,(x:-y:z), \\\sigma_c((x:y:z))=&\,(x:y:-z).
\end{aligned}
\end{equation*}

We find that $X:=Y/V$ admits an equation of the form
\begin{equation}\label{eq: G}
X:\; Au^2+Bv^2+Cw^2+avw+buw + cuv=0,
\end{equation}
where $u=x^2, v=y^2$, and $w=z^2$ and the map $f:Y\to X$ is given by
$(x:y:z)\mapsto (u:v:w)$.

For $i\in \{a,b,c\}$ we define $E_i=Y/\langle\sigma_i\rangle$. Then
$E_i$ is an elliptic curve for all $i$, as explained in the proof of Lemma~\ref{lem:quartic_eq}. We obtain the following diagram.
\begin{equation*}
\xymatrix{
	& Y\ar@{-}[ld]\ar@{-}[d]\ar@{-}[rd] & \\
	Y/\langle\sigma_a\rangle=E_a\ar@{-}[rd]&Y/\langle\sigma_b\rangle=E_b\ar@{-}[d] & Y/\langle\sigma_c\rangle=E_c\ar@{-}[ld]\\
	&Y/V = X. &  
}
\end{equation*}

Set 
\begin{equation}\label{eq:pa}
p_a(T)=T^2-2aT+4BC,\quad p_b(T)=T^2-2bT+4AC, \quad p_c(T)=T^2-2cT+4AB.
\end{equation}
For $i\in \{a,b,c\}$ we let $\Delta_i$ be $1/4$ times the discriminant of $p_i$, i.e.,
\begin{equation}\label{eq:Delta_a}
\Delta_a= a^2-4BC, \quad \Delta_b=b^2-4AC, \quad \Delta_c=c^2-4AB.
\end{equation}
One computes
\[
\begin{split}
\Delta(X)=&-4 A B C + A a^{2} + B b^{2} + C c^{2} -  a b c,\\
\Delta(Y)=&-2^{-20}A  B  C  \Delta_a^{2}  \Delta_b^{2}  \Delta_c^{2}   \Delta(X)^{4},\\
j(E_a) = & 256\,\frac{\left((2Aa-bc)^2- 3A\Delta(X) \right)^3}{A^2\Delta_b\Delta_c\Delta(X)^2}.
\end{split}
\]

Here $\Delta(X)$ and $\Delta(Y)$ denote the discriminants of $X$ and $Y$ as hypersurfaces (see for example \cite{Demazure}),
respectively, and $j(E_a)$ is the $j$-invariant of the elliptic curve $E_a$. The analogous formulas for $E_b$ and $E_c$ are obtained by permuting $(A,B,C,a,b,c)$ in the obvious way.

\begin{remark} The Jacobian of $Y$ is isogenous to $E_a\times E_b\times E_c$. This may for example be deduced from \cite[Proposition~15]{HLP}. Loc. cit. also gives an explicit finite extension of $K$ over which the isomorphism between the Jacobian of $Y$ and $E_a\times E_b\times E_c$ may be defined.
	The authors of that
	paper use this to find an example of a curve in this family with small conductor \cite[Example 16]{HLP}.
\end{remark}

Let $\alpha$ be a root of $p_a$, $\beta$ a root of $p_b$, and $\gamma$
a root of $p_c$.  Then the branch points of the cover $f:Y \rightarrow
X$ are
\begin{equation}\label{eq:branchV4}
\begin{aligned}
\sigma_a :&\; P_a = (0:\alpha: -2B), &P_a' = (0:-2C:\alpha),\\
\sigma_b :&\; P_b = (-2C:0:\beta), &P_b' = (\beta:0:-2A),\\
\sigma_c :&\; P_c = (\gamma:-2A:0), &P_c' = (-2B : \gamma: 0).
\end{aligned}
\end{equation}
The group element $\sigma_i$ is  the inertia
generator of the  points $P_i$ and $P_i'$ for $i\in \{a,b,c\}$.

\subsection{Stable reduction of covers}\label{sec:stable}
In this section $(K, \nu)$ is a complete discrete valuation field. We
assume that the characteristic of $K$ is different from $2$.
We denote by $\mathcal{O}$ the ring
of integers of $K$, by $\pi$ a uniformizing element, and by $k$ its residue
field.  We assume that $k$ is algebraically closed and of
characteristic $\neq 2$. We allow $K$ to be replaced by a finite
extension. We assume the valuation to be normalized by $\nu(p)=1$,
where $p>0$ is the residue characteristic of $\nu$. In the case of
equal characteristic~$0$, one adapts this choice of normalization
suitably.

Let $Y$ be a smooth projective absolutely irreducible curve over $K$.
A \emph{model} of $Y$ is a flat proper normal $\mathcal{O}$-scheme
$\mathcal{Y}$ such that $\mathcal{Y}\otimes_{\mathcal{O}} K\simeq
Y$. If the model is clear from the context, we call its special fiber
the \emph{reduction of $Y$} and denote it by $\overline{Y}$.

\begin{definition}\label{def:goodred}
	\begin{enumerate}[(1)]
		\item A curve $Y$ over $K$ has \emph{good reduction} if there exists a  model of~$Y$ over~$\OO$ such that~$\Yb$ is smooth.
		\item A curve $Y$ over $K$ has \emph{potentially 
			good reduction} if it has good reduction after replacing~$K$ by a
		finite extension. 
		\item A curve $Y$ over $K$
		has \emph{geometric bad reduction} if it does not have
		potentially good reduction.  
		\item A curve $Y$
		over $K$ has \emph{semistable reduction} if there
		exists a model $\Y$ of $Y$ whose special fiber
		$\overline{Y}$ is semistable, i.e.,~is reduced and has
		at most ordinary double points as
		singularities.  \end{enumerate}
\end{definition}

The Stable Reduction Theorem of Deligne--Mumford \cite{DM} states
that every curve~$Y$ of genus $g(Y)\geq 2$ admits a semistable model
after replacing $K$ by a finite extension, if necessary. Moreover,
there exists a unique minimal semistable model $\mathcal{Y}^{{\rm st}}$,
which is called the \emph{stable model} of~$Y$. Its special fiber
$\Yb$ is characterized by the property that each irreducible component
of genus~$0$ intersects the rest of $\Yb$ in at least three
points. Every semistable model $\mathcal{Y}$ admits a surjective map
$\mathcal{Y}\to \mathcal{Y}^{{\rm st}}$, which contracts the
``superfluous'' irreducible components of genus $0$ of the special
fiber and is an isomorphism on the generic fiber.

After replacing $K$ by a finite extension, we may assume that the
branch points of $f$ are $K$-rational.  Let $D\subset X$ be the branch
locus  of $f$. We consider it as a marking on $X$. A \emph{semistable
	$\mathcal{O}$-model} $(\mathcal{X}, \mathcal{D})$ of the marked curve
$(X,D)$ is a semistable model $\mathcal{X}$ of $X$ over $\mathcal{O}$,
together with a relative divisor $\mathcal{D}\to \Spec(\OO)$ with
$D = \mathcal{D}\otimes_{\OO} K$. Here $\mathcal{D}$ is the union of
disjoint sections $s_1, \ldots, s_r:\Spec(\OO)\to \mathcal{X}^{\rm
	sm}$ into the smooth locus of $\mathcal{X}$.

The following result is a version of the Stable Reduction Theorem for
covers. It follows immediately from \cite[Proposition~3.2]{BW}. Note that
the conditions of the following result are satisfied in our
situation.

\begin{proposition}\label{prop:stablymarked1}
	Let $G$ be a group. Let $f:Y\to X\simeq \PP^1_K$ be a $G$-Galois cover over $K$ such that
	the branch points of $f$ are $K$-rational, the number of branch
	points is greater than or equal to $3$, and the residue
	characteristic of $K$ does not divide the order of $G$.
	\begin{enumerate}[(1)]
		\item There exists a unique minimal semistable model  $(\mathcal{X}, \mathcal{D})$ of the marked curve $(X,D)$.
		\item The special fiber $\overline{X}$ of $\mathcal{X}$ is a tree of projective lines. Every irreducible component $\Xb_i$ of $\Xb$ contains at least three points which are either  singular points of $\Xb$ or belong to the support of $\overline{D}$.
	\end{enumerate}
\end{proposition}

We call the model $(\mathcal{X}, \mathcal{D})$ from
Proposition~\ref{prop:stablymarked1} the \emph{stably marked model}  of $(X,
D)$. The normalization~$\mathcal{Y}$ of~$\mathcal{X}$ in the function
field $K(Y)$ of $Y$ is a model of $Y$. After replacing~$K$ by a
further finite extension, we may assume that the special fiber $\Yb$
of $\mathcal{Y}$ is reduced.

\begin{definition}\label{def:adm}
	An \emph{admissible cover} is a $G$-Galois cover $\overline{f}: \Yb\to
	\Xb$ between projective semistable irreducible curves over an
	algebraically closed field such that:
	\begin{enumerate}[(1)]
		\item  The singular points of $\Yb$ map to the singular points of $\Xb$.
		\item  For every singular point $\tau\in \Yb$ the inertia group of $\tau$ acts on the two branches of $\Yb$ via characters that are inverse to each other.
	\end{enumerate}
\end{definition}

We refer to \cite[Section 5]{RW} or \cite{Wewers} for more details on
admissible covers and their deformation. In our situation all elements
of the Galois group $V$ have order~$1$ or $2$, and condition (2) of
Definition~\ref{def:adm} reduces to the condition 
that the inertia
group of the restriction of $\overline{f}$ to each of the branches of
$\Yb$ at the singular point $\tau$ is the same subgroup of $V$.

\begin{proposition}\label{prop:stablymarked2}
	\begin{enumerate}[(1)]
		\item  The model $\mathcal{Y}$ of $Y$ is semistable.
		\item The map $f$ extends to a finite flat map 
		$f:\mathcal{Y}\to\mathcal{X}$ over $\OO$.
		\item The special fiber $\overline{f}:\Yb\to \Xb$ of $f:\mathcal{Y}\rightarrow\mathcal{X}$ is an admissible cover.
		\item Every admissible $G$-Galois cover $\overline{f}:\Yb\to \Xb$ over $k$ occurs as the
		reduction of a finite map of semistable models of curves.
	\end{enumerate}
\end{proposition}

\begin{proof}  Statements (1) and (2) follow from \cite[Theorem 3.4]{BW}. Statements (3) and (4) follow from \cite{Wewers}. Here we use that
	the residue characteristic of $K$ does not divide the cardinality of the
	Galois group.
   \hfill $\qed$ 
   \end{proof}

The special fiber $\overline{f}:\Yb\to \Xb$ of the map
$f:\mathcal{Y}\to \mathcal{X}$ from
Proposition~\ref{prop:stablymarked2}.(2) is called the \emph{stable
	reduction} of $f$.

\subsection{Combinatorial description of the stable reduction} \label{sec:combi}
The boundary of the Hurwitz space $\Hb_{3,V}$ parame\-trizes admissible
$V$-Galois covers (see Definition~\ref{def:adm}). The natural map
\[
\Hb_{3,V}\to \Mb_{3, V},\qquad [Y\to X]\mapsto Y^{\text{st}},
\]
is finite and surjective.  Here $Y^{\text{st}}$ is the stable curve of
genus $3$ obtained by contracting all irreducible components of genus
$0$ that intersect the rest of $\Yb$ in at most two points. To find
all possibilities for the stable curves in the boundary of $\Mb_{3,
	V}$ it therefore suffices to find all possibilities for the
corresponding admissible covers by
Proposition~\ref{prop:stablymarked2}.(4). In this section we give a
combinatorial description for the admissible $V$-Galois covers arising
as the reduction of the Galois covers described in
Lemma~\ref{lem:Galois}.  For simplicity, we call these covers 
simply \emph{admissible $V$-Galois covers}.

\begin{definition}\label{def:decorated_graph}
	A \emph{decorated graph} over $k$ is a datum $(\Xb, \overline{D})$, where
	\begin{enumerate}[(1)]
		\item $\Xb$ is a semistable curve of genus $0$ over $k$, i.e.,~a tree of 
		projective lines,
		\item  $\overline{D}\subset \Xb^{{\rm sm}}$ is a set of smooth $k$-rational
		points of cardinality $6$,
		\item  every irreducible component of $\Xb$ contains at
		least three points that are either in~$\overline{D}$ or singular
		points of $\Xb$,
		\item  the points of $\overline{D}$ are labelled by one of $\{1,2,3\}$, 
		where each label occurs exactly twice.
	\end{enumerate}
\end{definition}

Definition~\ref{def:decorated_graph} may easily be adapted to a more
general set-up. In this paper the notion always refers to our
particular set-up. 

The following lemma follows from the fact that $\Xb$ is a tree,
together with the description of the fundamental group of an affine
curve of genus $0$. A similar statement of the lemma can be found in \cite[Section 3, Prop. 3.6]{flon}.

\begin{lemma}\label{lem:graph}
	Let $(\Xb, \overline{D})$ be a decorated graph. Then there is a unique
	way to label the singular points of $\Xb$ by one of $\{0, 1,2,3\}$
	such that the product of the inertia generators $\sigma_i\in V$ of the
	marked and singular points on each irreducible component of $\Xb$ is
	the trivial element of $V$.
\end{lemma}

Let $\overline{f}:\Yb\to \Xb$ be an admissible $V$-Galois cover.
Write $\overline{D}\subset \Xb^{{\rm sm}}$ for the branch points of~$\fb$ contained in the smooth locus $\Xb^{{\rm sm}}$ of $\Xb$.
Proposition~\ref{prop:stablymarked1} and Lemma~\ref{lem:Galois} imply
that every choice of a numbering of the elements of order~$2$ of $V$
gives rise to a decorated graph. In fact one checks that the labeling
of a singular point $\tau$ of $\Xb$ given by Lemma~\ref{lem:graph} is
$0$ if and only if $\fb$ is unbranched above $\tau$. If the label of
$\tau$ is non-zero, it corresponds to the inertia generator of $\tau$. This follows from the fact that the restriction of
$\overline{f}$ to an irreducible component  of $\Yb$ 
is a tamely ramified cover of
a component of $\Xb$ isomorphic to $\PP^1_k$.

The following lemma is straightforward, but a proof sketch is provided in Appendix A. A similar statement in the case of cyclic covers can be found in \cite[Section 4]{BW} and also in \cite[Section 3]{flon}. An elementary introduction can be found in \cite{An}.

\begin{lemma}\label{lem:adm}
	Let $(\Xb, \overline{D})$ be a decorated graph.
	Choose a numbering $\{\sigma_1, \sigma_2, \sigma_3\}$ of the elements
	of order~$2$ of $V$.  There exists a unique admissible $V$-Galois cover
	$\overline{f}:\Yb\to \Xb$ such that the inertia generator of
	$x\in \overline{D}$ is given by its label and $\fb$ is unbranched
	outside $\overline{D}\cup \Xb^{{\rm sing}}$. The decorated graphs are listed in Appendix~B. 
\end{lemma}

In the following theorem we determine the different possibilities for
the stable reduction of $Y$ in our situation.  By \emph{type} of a
stable curve we mean the intersection graph of its irreducible
components together with the genus of each of the irreducible components.

\begin{theorem}\label{FromGraphtoRedType}
	Let $Y/K$ be a smooth projective curve of genus $3$ such that there
	exists a subgroup $V \subseteq \Aut_{\Kb}(Y)$ with $g(Y/V)=0$
	and such that all subcovers of $Y\to Y/V$ of degree $2$ have
	genus~$1$.  Then there are $13$ different possibilities for the
	type of stable reduction of $Y$.
\end{theorem}

\begin{proof}
	Let
	$X=Y/V$ and $D$ be the branch locus of $f:Y\to X$, considered as a
	divisor on $X$. After replacing $K$ by a finite extension, we
	may assume that the divisor $D$ splits over~$K$. It follows by Proposition~\ref{prop:stablymarked1} that $(X, D)$ admits a stable
	model $(\mathcal{X}, \mathcal{D})$. Its special fiber
	$(\Xb, \overline{D})$ gives rise to a decorated graph
	(see Definition~\ref{def:decorated_graph}).

	Since $\overline{D}$ has cardinality $6$ and
	$(\Xb,\overline{D})$ must fulfill condition (3) from
	Definition~\ref{def:decorated_graph}, the curve $\Xb$ cannot
	have more than $4$ components. There are the following
	possibilities for $\Xb$:
	\begin{enumerate}[(I)] 
		\item $\Xb$ is
		irreducible.  
		\item $\Xb$ is a chain of two projective lines. Either there are three
		marked points on each component or   there are four marked points on one
		component and two marked points on the other.
		\item $\Xb$ is a chain
		of three projective lines. Either there are exactly two marked
		points on each component or there is only one marked point on
		the middle component and two, respectively three, marked
		points on the remaining components.  
		\item $\Xb$ is a chain of
		four projective lines and there is one marked point on each of the
		two components in the middle and two marked points on the
		other two components.  
		\item[(IV$^*$)] $\Xb$ consists of four
		projective lines such that the first component intersects all other
		components and there are no further singularities. There is no marked point on the first component
		and two marked points on each of the remaining
		components.  \end{enumerate}

	Next, we assign a label in $\{0,1,2,3\}$ to the points in
	$\overline{D}$ and the set $S$ of singularities of~$\Xb$. Let
	$\sigma_1, \sigma_2, \sigma_3$ denote the non-trivial elements
	of $V$. Lemma~\ref{lem:adm} states that the label on the
	singularities is uniquely determined by that of the marked
	points. Moreover, none of the marked points is assigned the label $0$.

	Up to the choice of the numbering of the elements of order~$2$
	in $V$, there are 20 possibilities for the decorated
	graph. These are depicted in Figures~(\ref{figs:dec-graphs-I})--(\ref{figs:dec-graphs-IV}) 
	in Appendix~B. The numbering of the cases in the rest
	of the proof refers to the numbers of the cases there. 
	
	To finish the proof it remains to describe the admissible cover
	$\Yb\to \Xb$ and the image of $\Yb$ in $\Mb_{3,V}$ for each of the
	possibilities for the decorated graph $(\Xb, \overline{D})$. For more details, we refer to the extended version of this proof in Appendix A. Recall
	from the beginning of Section~\ref{sec:combi} that the image in
	$\Mb_{3,V}$ of the semistable curve $\Yb$ is obtained by contracting
	all its irreducible components of genus $0$ which intersect the rest
	of $\Yb$ in at most two points. We discuss one of the cases. The other
	cases are similar. The different types are depicted in Figure~\ref{figs:admissible-covers} in Appendix~B.

	Let  $(\Xb, \overline{D})$ be a decorated graph of type III.5,
	Figure~\ref{fig:TypeIII.5}.  
	
	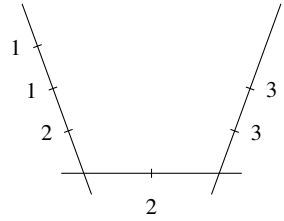
\begin{figure}[h!] \sidecaption
		\begin{tikzpicture}[scale=.6]
		\draw (-0.5,0) -- (3.5,0);
		\draw (0,0) -- +(110:4cm);
		\draw (0,0) -- +(-70:0.5cm);
		\draw (3,0) -- +(70:4cm);
		\draw (3,0) -- +(-110:0.5cm);
		\foreach \x/\y in {1/\two,2/\one,3/\one}
		\draw (110:\x)++(20:0.1) -- ++(-160:0.2) node[left=0.5mm] {\y };
		\draw (1.5,-0.1) -- (1.5,0.1) node[below=3mm] {\two };
		\foreach \x/\y in {1/\three,2/\three}
		\draw (3,0)++(70:\x)++(160:0.1) -- ++(-20:0.2) node[right=0.5mm] {\y };
		\end{tikzpicture}
		\caption{$\Xb$ of Type III.5. \label{fig:TypeIII.5}}
	\end{figure}
	
	We number the irreducible components of $\Xb$ as
	$\Xb_1,\Xb_2,\Xb_3$ (from left to right). Let $\fb: \Yb\to\Xb$
	be the (unique) admissible $V$-Galois cover from
	Lemma~\ref{lem:adm}.  By Lemma~\ref{lem:graph} the label of
	the intersection point of $\Xb_1$ and $\Xb_2$ is $2$; the
	label of the intersection point of $\Xb_2$ and $\Xb_3$ is
	$0$. 
	
	The restriction of $\fb$ to $\Xb_1$ contains branch points with two
	different inertia generators. Therefore there is a unique
	irreducible component $\Yb_1$ of $\Yb$ above $\Xb_1$. The
	Riemann--Hurwitz formula implies that $\Yb_1$ has genus $1$.
	All branch points of the restriction of $\fb$ to~$\Xb_2$
	(respectively~$\Xb_3$) have the same label. Therefore there are two
	irreducible components of $\Yb$ above $\Xb_2$ and two above~$\Xb_3$. The Riemann--Hurwitz formula implies that all four
	components have genus zero. Since $\fb$ is unbranched at the
	intersection point of $\Xb_2$ and $\Xb_3$, these four
	components intersect in four points as depicted in
	Figure~\ref{fig:TypeIII.5_Y}.
	
	\begin{figure}[h!]\sidecaption  \begin{tikzpicture}[scale=.6]
		\draw (-0.5,0) -- (4.5,0);
		\draw (-1,1) -- (5,1);
		\draw (0,0) -- +(110:4cm);
		\draw (0,0) -- +(-70:0.5cm);
		\draw (3,0) -- +(70:4cm);
		\draw (3,0) -- +(-110:0.5cm);
		\draw (4,0) -- +(70:4cm);
		\draw (4,0) -- +(-110:0.5cm);
		\draw (-1.5,4.5) node[] {$\Yb_1$};
		\draw (-2,1.5) node[] {$\Yb_2$};
		\draw (-1.5,0) node[] {$\Yb'_2$};
		\draw (4.5,4.5) node[] {$\Yb_3$};
		\draw (6,4.5) node[] {$\Yb'_3$};
		\end{tikzpicture}
		\caption{$\Yb$ corresponding to Type III.5. \label{fig:TypeIII.5_Y}}
	\end{figure}
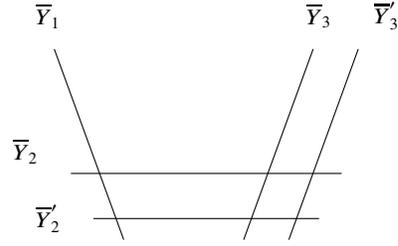

	We conclude that there are exactly two irreducible components of $\Yb$
	of genus zero that intersect the rest of $\Yb$ in at most two points,
	namely the two components above~$\Xb_3$. Contracting these yields the
	image of $\Yb$ in $\Mb_{3,V}$ as depicted in Figure~\ref{fig:TypeIII.5_stable}.
	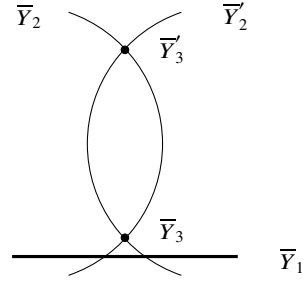
\begin{figure}[h!] \sidecaption
		\begin{tikzpicture}[scale=.5]
		\draw (-0.5,-0.5) to[quick curve through={(2,3)}]
		(-0.5,6.5);
		\draw (2.5,-0.5) to[quick curve through={(0,3)}]
		(2.5,6.5);
		\draw[very thick] (-2,0) to (4,0);
		\draw (6,0.5)  node[below=3mm,left=0mm] {$\Yb_1$};
		\draw (-1,7)  node[below=3mm,left=0mm] {$\Yb_2$};
		\draw (4.5,7)  node[below=3mm,left=0mm] {$\Yb'_2$};
		\draw [fill] (1,0.5) circle [radius=0.1];
		\draw [fill] (1,5.5) circle [radius=0.1];
		\draw (2.8,1.3)  node[below=3mm,left=0mm] {$\Yb_3$};
		\draw (2.8, 6.1)  node[below=3mm,left=0mm] {$\Yb'_3$};
		\end{tikzpicture} 
		\caption{Special fiber of the stable model. \label{fig:TypeIII.5_stable}}
	\end{figure}	
   \hfill $\qed$ 
   \end{proof}

\begin{remark}
	We note that none of the singular reduction types occurring in 
	Theorem~\ref{FromGraphtoRedType} is of compact type. Phrased differently, 
	all the singular curves in Appendix~B have loops. This is not
	surprising, as the inertia group of a singular point of $\Yb$ is
	cyclic since the residue characteristic is different from $2$. One may
	deduce from the fact that $\Yb$ is connected, that we always have loops.
\end{remark}

\subsection{Computing the stable reduction}\label{sec:compute}

In this section we outline the method we use to compute the decorated
graph associated with a curve $Y/K$ endowed with an action of~$V$ such
that all degree-$2$ subcovers have genus $1$. We assume that $K$ is
sufficiently large to have all branch points of $f:Y\to
X:=Y/V$ $K$-rational. Recall that we may assume that $Y$ is given
by Equation \eqref{eq: F}. As in the proof of
Theorem~\ref{FromGraphtoRedType}, it suffices to determine the
decorated graph associated with the special fiber $(\Xb,
\overline{D})$ of the stably marked model $(\mathcal{X},
\mathcal{D})$ of $(X, D)$, where $D$ is the branch locus
of the natural map $f:Y\to X$ marked by the inertia generators.

To compute $(\Xb, \overline{D})$ we follow the strategy from
\cite[Section 4.2]{BW}. Loc.~cit.~treats the superelliptic case,
i.e.,~the case of cyclic covers of the projective line. With our
preparations, the adaptation to the current set-up is straightforward. 

The key idea used in \cite[Section 4.2]{BW} is to describe the
irreducible components of~$\Xb$ in terms of \emph{coordinates}
$\xi:X\stackrel{\sim}{\to}\PP^1_K$. Every coordinate $\xi$ defines a
model $\PP^1_{\OO}$ of~$X$, and hence by reduction, a projective line
$\Xb_\xi$ over $k$ \cite[Proposition~4.2.(1)]{BW}. Let $T$ 
be the set of triples of pairwise distinct points of $D$. Every $t=(P,Q,R)\in T$
defines a unique coordinate $\xi_t$ with 
\[
\xi_t(P) =0, \; \xi_t(Q) =1,\; \xi_t(R) = \infty. 
\]
Two coordinates $\xi_1$ and $\xi_2$ are \emph{equivalent} if
$\xi_1\circ \xi_2^{-1}:\PP^1_K\stackrel{\sim}{\to} \PP^1_K$ extends to
an isomorphism over $\OO$. This equivalence relation defines an
equivalence relation on $T$ denoted by $\sim$.
\cite[Proposition~4.2.(3)]{BW} states that there is a bijection between
$T/_\sim$ and the set of irreducible components of $\Xb$. We write
$\xi_t$ (respectively~$\Xb_t$) for the coordinate (respectively~the irreducible
component of $\Xb$) corresponding to $t\in T$. In \cite[Remark
4.3]{BW} it is explained how to reconstruct the intersection points
between the different irreducible components from the values
$\xi_t(P)$, where $t\in T/_\sim$ and $P\in D$ runs over the branch
points.

\begin{example}\label{exa:blowup}
	Let $Y/K$ be a plane quartic curve defined by Equation
	\eqref{eq: F}.
	Assume that $\nu(\Delta(X)) = 0$, $\nu(A),
	\nu(B) > 0$, $\nu(C)=0$, $\nu(a)>0$, and $\nu(b)= \nu(c) = 0$. We
	determine the stable reduction of $Y$ under these conditions. This
	is a special case of case (f.vi) of Theorem~\ref{THM:Main_NonDegConic}: the
	result there is formulated more symmetrically in terms of the invariants we
	introduce in Section~\ref{sec:invariants_quartic}.

	We let $\X_0$ be the model of a conic $X$ defined by Equation \eqref{eq: G} and use the
	notation of Equation \eqref{eq:branchV4} for the branch points of $f:Y\to
	X$. Write $\beta$ (respectively~$\gamma$) for the root of $p_b$ (respectively~$p_c)$
	of valuation zero, with \(p_b,\,p_c\) defined as in Equation \eqref{eq:pa}. 
	Both roots, $\alpha$ and ${\alpha}'$, of $p_a$ have positive valuation, $\nu(\alpha)+\nu({\alpha}')=\nu(B)$ and we can assume $\nu(\alpha)\leq\nu({\alpha}')$. In particular, $0=\nu(C)<\nu(\alpha)<\nu(B)$, and the branch points $P_a, \; P_a'$ and $P_c'$ reduce to
	the point $(0:1:0)$, the branch points $P_b'$ and~$P_c$ reduce to the
	point $(1:0:0)$, and $P_b=(-2C:0:\beta)$ reduces to a different point from the previous ones on the special fiber
	$\Xb_0$ of $\X_0$. In particular, $\X_0$ is not stably marked. The
	previous discussion implies that $\Xb_0$ is one of the irreducible
	components of the special fiber $\Xb$ of the stably marked model $\X$.
	
	\begin{figure}[!htbp]\sidecaption
		{
			\begin{tikzpicture}[scale=.5]
			\draw (-2,0) -- (5,0);
			\foreach \x/\y in {0/{\(\overline{P_a}, \overline{P_c'}, \overline{P_a'}\)},1/\(\overline{P_b}\),2/{\(\overline{P_b'},\overline{P_c}\)}}
			\draw (2*\x,-0.1) -- (2*\x,0.1) node[below=3mm,left=0mm,rotate=45] {\y};
			\draw (7,0.6)  node[below=3mm,left=0mm] {$\overline{X}_0$};
			\end{tikzpicture}
		}
		\caption{Configuration of the branch points on the component $\Xb_0$.}
		\label{fig:X0}
	\end{figure}
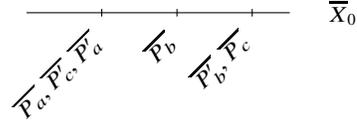

	Inspection of the cases in Appendix~B yields that the
	decorated graph $(\Xb, \Db)$ is of type IV.2, IV.3 or III.6. In
	order to determine the reduction type of $\Xb$, it suffices to
	distinguish between these three cases. We can do this by considering 
	the  coordinate $\xi=\xi_t$ of $X$ corresponding to $t = (P_a', P_a, P_b')$. 
	This coordinate is given by 
	\[
	\xi = \frac{\beta u + \alpha v + 2Cw}{2(\alpha-a) \cdot v}.
	\]
	Here $u,v,w$ are the coordinates of $X$ as in Equation \eqref{eq: G}.
	
	Using the assumptions on the valuations of the parameters
	$A,B,C,a,b,c$ one computes that $\overline{\xi(P_b')} =
	\overline{\xi(P_c)} = \infty$. Hence the type of $\Xb$ only depends on
	$\overline{\xi(P_c')}$. One may check that
	$\overline{\xi(P_b)}=\infty$ as well, by finding a different
	expression for $\xi$. However, this is not needed to distinguish
	between the possibilities for $\Xb$.
	
	Namely, $P_a, P_a'$, and $P_c'$ specialize to pairwise distinct points
	of the component $\Xb_\xi$ different from the intersection point of
	$\Xb_\xi$ with the rest of $\Xb$ if and only if
	$\overline{\xi(P_c')}\not\in \{0, 1, \infty\}.$ Otherwise, we need an
	additional coordinate $\xi'$ to separate $P_c'$ from the point $Q\in
	\{P_a, P_a', P_b\}$ with $\xi(P_c')\equiv \xi(Q)\pmod{\pi}$. However,
	to decide what decorated graph occurs it suffices to know for which
	point $Q$ this holds. It is not necessary to calculate the coordinate
	$\xi'$ explicitly. The possible configurations of the components $\Xb_0$ and $\Xb_{\xi}$ are depicted in Figure \ref{fig:example_specializations}.
	
	\begin{figure}[!htbp]
		\centering
		\captionsetup[subfigure]{labelformat=empty}
		\subfloat[\(\xi(P_c') = 1 \mod{\pi}\)]{
			\begin{tikzpicture}[scale=.4]
			\draw (-0.5,0) -- (6,0);
			\draw (0,0) -- +(110:5.5cm);
			\draw (0,0) -- +(-70:0.5cm);
			\foreach \x/\y in {3/{\(\overline{P_a}, \overline{P_c'}\)},1.5/\(\overline{P_a'}\)}
			\draw (110:\x)++(20:0.1) -- ++(-160:0.2) node[left=0.5mm] {\y };
			\foreach \x/\y in {1/\(\overline{P_b}\),2/{\(\overline{P_b'},\overline{P_c}\)}}
			\draw (2*\x,-0.1) -- (2*\x,0.1) node[below=3mm,left=0mm,rotate=45] {\y};
			\end{tikzpicture}
		}
		\subfloat[\(\xi(P_c') = \infty \mod{\pi}\)]{
			\begin{tikzpicture}[scale=.4]
			\draw (-0.5,0) -- (6,0);
			\draw (0,0) -- +(110:5.5cm);
			\draw (0,0) -- +(-70:0.5cm);
			\foreach \x/\y in {3/\(\overline{P_a}\),1.5/\(\overline{P_a'}\)}
			\draw (110:\x)++(20:0.1) -- ++(-160:0.2) node[left=0.5mm] {\y };
			\foreach \x/\y in {1/\(\overline{P_b}\),2/{\(\overline{P_b'},\overline{P_c}\)}}
			\draw (2*\x,-0.1) -- (2*\x,0.1) node[below=3mm,left=0mm,rotate=45] {\y};
			\draw (0,0)++(55:0.1) -- ++(-125:0.2) node[below = 3mm, left=0mm] {\(\overline{P_c'}\)};
			\end{tikzpicture}
		}
		\subfloat[\(\xi(P_c') \neq 0,1,\infty \mod{\pi}\)]{
			\begin{tikzpicture}[scale=.4]
			\draw (-0.5,0) -- (6,0);
			\draw (0,0) -- +(110:5.5cm);
			\draw (0,0) -- +(-70:0.5cm);
			\foreach \x/\y in {4.5/{\(\overline{P_c'}\)}, 3/{\(\overline{P_a}\)},1.5/\(\overline{P_a'}\)}
			\draw (110:\x)++(20:0.1) -- ++(-160:0.2) node[left=0.5mm] {\y };
			\foreach \x/\y in {1/\(\overline{P_b}\),2/{\(\overline{P_b'},\overline{P_c}\)}}
			\draw (2*\x,-0.1) -- (2*\x,0.1) node[below=3mm,left=0mm,rotate=45] {\y};
			\end{tikzpicture}
		}
		\caption{Possible configurations of $\Xb_{\xi}$ and $\Xb_0$.}
		\label{fig:example_specializations}
	\end{figure}
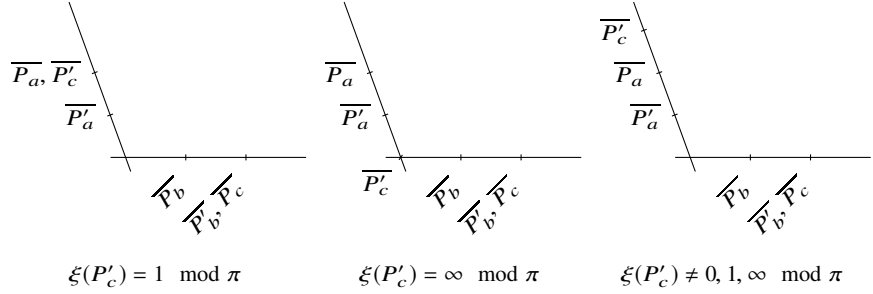
	
		Recall that $\nu(\alpha)\leq\nu({\alpha}')$, so $\nu(\alpha)\leq
	\nu(a)$ and $2\nu(\alpha) \leq \nu(BC)$. We have
	\[
	\xi(P_c') = \frac{-2B \beta + \alpha \gamma }{2(\alpha -a) \gamma}=\frac{\alpha \gamma -2B \beta }{\alpha \gamma- {\alpha}' \gamma}.
	\]	
	
	By our assumptions, we have that $\nu(\alpha-a)\geq \nu(\alpha)$. 
	Hence we get $\overline{\xi(P_c')} \neq 0$. There remain three cases:
	\[
	\overline{\xi(P_c')}=\begin{cases}
	\overline{\xi(P_a)}=1& \text{ iff } \nu({\alpha}')>\nu(\alpha) \text{ iff } \nu(BC)>2\nu(a),\; \text{(type IV.3)},\\
	\overline{\xi(P_b)}= \infty& \text{ iff } \nu({\alpha}-a)>\nu(\alpha)\text{ iff }\nu(\Delta_a)>\nu(BC)=2\nu(a),\; \text{(type IV.2)},\\
	\neq 0,1,\infty&  \text{ otherwise},\; \text{(type III.6)}.
	\end{cases}
	\]
	One may check that the last case occurs if and only if
	$\nu(\Delta_a)=\nu(BC)<2\nu(a)$.
	
	As in the proof of Theorem~\ref{FromGraphtoRedType} one finds that
	the reduction type of $\Yb$ is Winky Cat if $\Xb$ is of type III.6,
	Cat if $\Xb$ is of type IV.3, and Garden if $\Xb$ is of type IV.2.

	If $\Xb$ is of type III.6, then the curve $\Yb$ has one component $\Yb_1$ of
	positive genus. The curve $\Yb_1$ is an elliptic curve and the
	restriction of $\fb$ to $\Yb_1$ factors as $\Yb_1\to \Xb_{\xi'}\to
	\Xb_\xi$, where $\Xb_{\xi'}:=\Yb_1/\langle\sigma_c\rangle$ is a
	projective line. Lemma~\ref{lem:graph} yields a description of the inertia
	generators. We explain how to compute
	the $j$-invariant of $\Yb_1$.
	
	The map $\Xb_{\xi'}\to \Xb_\xi$ has degree $2$ and is
	exactly branched at $\overline{\xi(P_a')}=0$ and
	$\overline{\xi(P_a)}=1$. Normalizing the unique point of $\Xb_{\xi'}$
	above $P_a'$ (respectively~$P_a$) to $0$ (respectively~$1$) and one of the points
	above the intersection point of $\Xb_{\xi}$ with the rest of $\Xb$ to
	$\infty$, we obtain Figure~\ref{fig:ellipitc_curve}.
	
	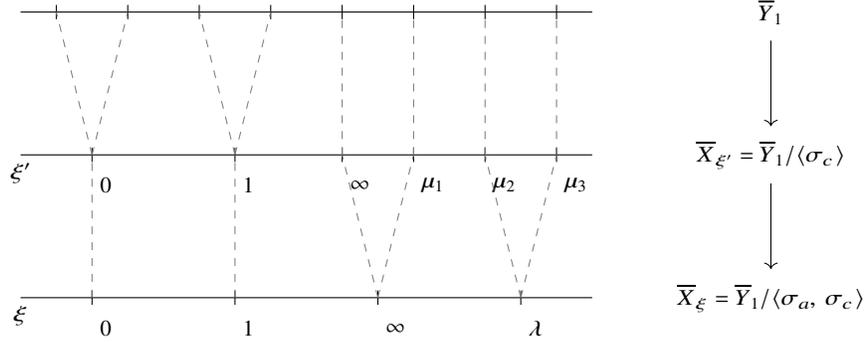
\begin{figure}[!htbp]
		\centering
		\begin{tikzpicture}[scale = 0.95]
		\draw (0.5,4)  --  (8.5,4);
		\draw  (0.5,2) node[left, below]{$\xi'$} -- (8.5,2) ;
		\draw (0.5,0) node[left, below] {$\xi$} -- (8.5,0);
		\foreach \x/\y in {1, 2, 3, 4, 5, 6, 7, 8}
		\draw (\x,3.9) -- (\x,4.1) ;
		\foreach \x/\y in {1.5/$0$,3.5/$1$,5/$\infty$,6/$\mu_1$, 7/ $\mu_2$, 8 / $\mu_3$}
		\draw (\x,1.9) -- (\x,2.1) node[below=5mm, right] {\y };
		\foreach \x/\y in {1.5/ $0$,3.5/ $1$,5.5/ $\infty$,7.5/  $\lambda$}
		\draw (\x, -0.1) -- (\x, 0.1) node[below=5mm, right] {\y };
		\foreach \x/\y in {1.5,3.5}
		\draw[draw = gray, dashed] (\x,0) -- (\x,2);
		\foreach \x/\y in {5.5,7.5}{
			\draw[draw = gray, dashed] (\x,0) -- (\x-.5,2); 
			\draw[draw = gray, dashed] (\x,0) -- (\x+.5,2); 
		}
		\foreach \x/\y in {5,6,7,8}
		\draw[draw = gray, dashed] (\x,2) -- (\x,4);
		\foreach \x/\y in {1.5,3.5}{
			\draw[draw = gray, dashed] (\x,2) -- (\x-.5,4); 
			\draw[draw = gray, dashed] (\x,2) -- (\x+.5,4); 
		}
		
		\draw[->] (11, 3.6)  -- (11,2.4);
		\draw[->](11,1.6) -- (11,0.4);
		\draw (11, 4) node{$\Yb_1$};
		\draw (11,2) node{$\Xb_{\xi'}=\Yb_1/\langle\sigma_c\rangle$};
		\draw (11,0) node {$\Xb_{\xi}=\Yb_1/\langle\sigma_a,\,\sigma_c\rangle$};
		\end{tikzpicture}
		\caption{The factorization of $\bar{f}$ restricted to $\Yb_1$.}
		\label{fig:ellipitc_curve}
	\end{figure}
	
	The coordinate $\xi$ can be then be written as a quotient of polynomials of 	degree smaller or equal to $2$ on $\xi'$, and from the conditions above we obtain
	\[
	\Xb_{\xi'}\to \Xb_{\xi}:\, \xi'\mapsto \xi=\frac{(\xi')^2}{2\xi'-1}.
	\]
	
	The degree-$2$ map $\Yb_1\to \Xb_{\xi'}$ is branched at the inverse
	image of $\xi(P_c')=:\lambda$ and $\xi=\infty$, i.e.,~at $\xi'=\infty,
	\mu_1=1/2$ and the two roots $\mu_2$ and $\mu_3$ of $t^2-2\lambda t+\lambda$. Using the assumptions
	on the parameters we find that $\lambda=\alpha/2(\alpha-a)$.
	Taking the cross ratio of these $4$ points, we find that
	\[
	j(\Yb_1)\equiv \frac{2^6(a^2+12BC)^3}{\Delta_a^2\cdot 4BC}\pmod{\pi}.
	\]
\end{example}

\begin{remark}\label{rem:blowup}
	\begin{itemize}
		\item[(a)] In Proposition \ref{Prop:norm1} we formulate certain minimality
		conditions on the parameters $A,B,C,a,b,c$. Assuming these
		conditions, Equation \eqref{eq: G} defines a model of $X$, which
		we denote by $\mathcal{X}_0$. In the situation of Example
		\ref{exa:blowup} this model is semistable. However, this is not true
		in general. It may happen that the special fiber $\Xb_0$ of
		$\mathcal{X}_0$ is not reduced. In Proposition
		\ref{prop:degeneratedconicproperties}.(ii) this case is
		characterized. The method to compute the stable model of $X$ still
		works, but one needs to go to an extension of $K$ to
		find a model of $X$ whose special fiber is reduced. More details can
		be found in the proof of Lemma
		\ref{lem:degeneratehyp}. 
		\item[(b)] In Example \ref{exa:blowup} we sketched a systematic method
		for computing the invariants of the components of positive genus of
		the stable reduction of a curve $Y$. This method has the advantage
		that it always works. Once one knows the type of the stable
		reduction of a curve $Y$, it is sometimes faster to explicitly write
		down a model $\mathcal{Y}$ of $Y$, that is not necessarily semistable, to
		calculate the invariants of the components of positive genus of the
		stable reduction of a curve $Y$.
		
		Namely, let $\mathcal{Y}$ be a not necessarily semistable model of $Y$
		and assume that the normalization of its reduced special fiber
		contains an irreducible component $Z$ of positive genus. Then the
		uniqueness of the stable model implies that~$Z$ is also an irreducible
		component of the normalization of the stable reduction
		$\Yb^{\text{stab}}$ of $Y$. Therefore an equation for $Z$ may be used
		to compute the invariants for the corresponding irreducible component
		of the stable reduction. This method is used, for example, in the proof of Lemma
		\ref{lemma:Case0a-c}. We refer to this proof for more details. 
	\end{itemize}
\end{remark}

\section{The smooth plane quartic case} \label{sec:setup}

Let $Y/K$ be a smooth projective plane quartic over a complete
discrete valuation field of characteristic $0$ and residue
characteristic $p \geq 0$ different from $2$, such that $\Aut_{\Kb}(Y)$ contains
a subgroup isomorphic to $V = C_2\times C_2$. Recall that, by
Lemma~\ref{lem:quartic_eq}, $Y$ admits an equation of the
form:
\[
Y :\, Ax^4+By^4+Cz^4+ay^2z^2+bz^2x^2 + cx^2y^2 = 0
\]
for some $A,B,C,a,b,c \in K$, possibly after replacing $K$ by a finite
extension, as in Equation~\eqref{eq: F}.

In Section \ref{sec:main_results} we identify the reduction type of a
given curve $Y/K$, where all possible types are listed in
Appendix~B. Before stating the results we discuss the
problem that an equation of the form \eqref{eq: F}, and hence the
coefficients $A,B,C,a,b,c$, for $Y/K$ are only unique up to
$K$-isomorphisms. Proposition~\ref{Prop:norm1} states a normalization
condition for the valuation of the coefficients of Equation \eqref{eq: F}. This
allows us in Proposition~\ref{prop:invariants} to exhibit a set of
invariants $\Iiii,\Iiii',\Iiii'',\Ivi$ for the special locus $\mathcal{M}_{3,V}^{\text{quar}}$.  Proposition~\ref{Prop:norm1} allows us to
assume that $\Iiii,\Iiii',\Iiii'',\Ivi$ have non-negative valuation
and at least one has valuation zero. The classification of the
reduction types of $Y$ in terms of the invariants is stated in
Theorems~\ref{THM:Main_NonDegConic} and~\ref{THM:Main_DegConic}. 

\subsection{Invariants} \label{sec:invariants_quartic}

Let $Y/K$ be given as in Equation \eqref{eq: F}; such equation can be normalized, as given by the following result:

\begin{proposition}\label{Prop:norm1} 
	After a suitable change of variables in Equation \eqref{eq: F} we can always assume that the valuation of at least one of the elements in each set $\{A,B,c\}$, $\{A,b,C\}$, $\{a,B,C\}$, $\{A,b,c\}$, $\{a,B,c\}$, $\{a,b,C\}$ is zero while all the others are non-negative.
\end{proposition}

\begin{proof} Let \(o_r, o_s, o_t\in\OO\) be elements with valuation \(r,s, t\in \QQ_{\geq 0}\) respectively. We assume the valuation to be normalized by $\nu(p)=1$,
	where $p>0$ is the residue characteristic of $\nu$. We allow $K$ to be replaced by a finite
	extension (if necessary) such that $K$ contains an element of this valuation.
	
	Observe that the change of variables
	\[(x,y,z) \mapsto (o_rx, o_sy, o_tz)\]
	changes the valuation of the coefficients as follows
	\[
	\begin{aligned}
	\nu(A) &\mapsto \nu(A) + 4r, & \nu(B) &\mapsto \nu(B) + 4s,& \nu(C) &\mapsto \nu(C) + 4t,\\
	\nu(a) &\mapsto \nu(a)+ 2s+2t, & \nu(b) &\mapsto \nu(b)+2r+2t,& \nu(c) &\mapsto \nu(c)+2r+2s.
	\end{aligned}
	\]
	
	Let \(U\) be one of the sets listed in the statement, and assume that all the valuations of the parameters in \(U\) are positive. We will see how to obtain an isomorphic model for \(Y\) that is normalized with respect to all the sets for which the original model was normalized, and also with respect to \(U\).
	
	By symmetry it is enough to consider the cases where \(U\) is $\{A,b,c\}$ or $\{A,B,c\}$.
	\begin{itemize}
		\item For the first case consider the change
		\[(x,y,z) \mapsto (x/\pi_1, y, z)\]
		with $\pi_1\in\mathcal{O}$ an element of valuation $\nu(\pi_1) = \min(\nu(A)/4, \nu(b)/2, \nu(c)/2)$. With this change, the valuation of at least one among \(A,b,c\) becomes zero, and the valuations of \(a,B\), and \(C\) remain the same.
		\item Assume now that the model is normalized with respect to every set of the form \(\{L,m,n\}\), and assume that it is not with respect to \(\{A,B,c\}\). It follows that \(\nu(a) = \nu(b) = 0\), since otherwise we could normalize further with respect to either \(\{a,B,c\}\) or \(\{A,b,c\}\). Consider now the change
		\[(x,y,z) \mapsto (x/\pi_2, y/\pi_2, \pi_2z)\]
		with $\pi_2\in\mathcal{O}$ an element of valuation $\nu(\pi_2) = \min(\nu(A)/4, \nu(B)/4, \nu(c)/4)$. With this change, the valuation of at least one among \(A,B,c\) becomes zero, the valuations of \(a\) and \(b\) remain the same, and the valuation of \(C\) increases by $4\nu(\pi_2)$. The increase of the valuation of \(C\) does not affect the normalization with respect to any other set, since any set containing \(C\) also contains \(a\) or \(b\), which have valuation zero, as explained above. 
	\end{itemize}
   \hfill $\qed$ 
   \end{proof}

Dixmier--Ohno invariants~\cite{Dixmier,Ohno} 
classify isomorphism classes of plane quartics. 
Moreover, in \cite{LRS16} a reconstruction method is 
presented to compute the equation of a curve 
corresponding to a given tuple of Dixmier--Ohno 
invariants. In \cite[{Function \texttt{IsInstrataD4}}]{Jeroen} the expressions for Dixmier--Ohno invariants for quartics in the locus $\mathcal{M}_{3}^{\text{quar}}$ are given. However, it is more convenient to work with a smaller set of invariants, specifically for the locus $\mathcal{M}_{3,V}^{\text{quar}}$ instead of the general Dixmier--Ohno invariants; indeed explicit computation with these is unnecessarily complicated. 
We therefore consider the four invariants as in the following statement.

\begin{proposition}\label{prop:invariants}
	The elements
	\begin{eqnarray*}
		\Iiii=ABC,\qquad\Iiii'=A\Delta_a + B \Delta_b + C \Delta_c,\\
		\Iiii''=-4ABC + Aa^2 + Bb^2 + Cc^2 - abc,\qquad
		\Ivi=\Delta_a\Delta_b\Delta_c
	\end{eqnarray*}
	are invariants for the locus $\mathcal{M}_{3,V}^{\emph{quar}}$.
\end{proposition}

\begin{proof} 
	By Lemma \ref{lem:quartic_eq}.(1), any isomorphism between plane quartics in $\mathcal{M}_{3,V}$ has to preserve the automorphism group
	$$
	V\simeq \Bigg\langle\begin{pmatrix}
	-1 & 0 & 0\\0 & -1 & 0\\0 &0&1
	\end{pmatrix},\begin{pmatrix}
	-1 & 0 & 0\\0 & 1&0\\0 & 0 &-1
	\end{pmatrix}\Bigg\rangle \subset \Aut_K(Y).
	$$
	Since plane quartics are given by their canonical
	models, isomorphisms between them are linear. Moreover, those isomorphisms leave invariant by conjugation 
	the previous group. This implies that isomorphisms of plane quartics
	in $\mathcal{M}_{3,V}^{\text{quar}}$ are given by products of
	permutation matrices and diagonal matrices. Hence the elements
	from the statement of the lemma considered as element of
	$K[A,B,C,a,b,c]$ are invariants for the locus
	$\mathcal{M}_{3,V}^{\text{quar}}$.
   \hfill $\qed$ 
   \end{proof}

\begin{remark}\label{rem:DeltaX}
	Notice that $\Iiii''=\Delta(X)$.
\end{remark}

\begin{proposition}\label{Prop:Invariants_Generate} The invariants $\Iiii,\Iiii',\Iiii''$ and $\Ivi$ are generators for the invariants algebra of the locus $\mathcal{M}_{3,V}^{\emph{quar}}$. 
\end{proposition}

\begin{proof} 	
	In general, for any characteristic different from $2$, we can proceed as follows: first, we normalize to obtain $A=B=C=1$. With this normalization the group of linear transformations acts on $K[a,b,c]$ via the finite group 
	\[G:=\Bigg\langle \begin{pmatrix}0 & 0 & 1\\ 1 & 0 & 0\\ 0 & 1 & 0\end{pmatrix},\begin{pmatrix}0 & 1 & 0\\ 1 & 0 & 0\\0 & 0 & 1\end{pmatrix},\begin{pmatrix}i & 0 & 0\\ 0 & i & 0\\ 0 & 0 & -1\end{pmatrix}\Bigg\rangle.\]
	
	Then, we compute the invariants $K[a,b,c]^G$ via Derksen's algorithm \cite[Algorithm 4.1.9]{Kemper} with the \texttt{Magma}~\cite{Magma} function \texttt{FundamentalInvariants}, and we obtain generators: $1, a^2 + b^2 + c^2,\,abc,\,a^2b^2 + a^2c^2 + b^2c^2$. After de-normalizing we get the weight $3,3,3,6$ invariants $ABC,\,Aa^2+Bb^2+Cc^2,\,abc,\,ABa^2b^2+BCb^2c^2+CAc^2a^2$. 
	
	In order to do the computations in \texttt{Magma} we needed to fix the field, so we fixed the fields $\mathbb{F}_3$ and $\mathbb{Q}$ because the order of $G$ is a product of a power of $2$ and a power of $3$. For any other characteristic $p>3$ and because of Molien's Formula \cite[Theorem 3.2.2]{Kemper} we always find the same Hilbert series $H(K[a,b,c]^G,t)=H(\mathbb{Q}[a,b,c]^G,t)=\frac{1}{(1-t^2)(1-t^3)(1-t^4)}$, so $K[a,b,c]^G$ is generated by $3$ invariants of weights $2,\,3$ and $4$. Since the $3$ expressions $a^2 + b^2 + c^2,\,abc,\,a^2b^2 + a^2c^2 + b^2c^2$ are invariants for any characteristic and the one of weight $4$ is not a multiple of the square of the one of weight $2$, they are generators of $K[a,b,c]^G$ for all characteristics different from $2$.
   \hfill $\qed$ 
   \end{proof}

\begin{remark} In characteristic $0$ we know that the Dixmier--Ohno invariants generate the invariant ring of smooth plane quartics \cite{Dixmier,Ohno}. Work in progress by  R.~Lercier, E.~Lorenzo García, and C.~Ritzenthaler aims to show that this also holds in characteristic $p>7$. The Dixmier--Ohno invariants of plane quartics can be written in terms of the invariants $\Iiii,\Iiii',\Iiii''$ and $\Ivi$ (see \cite[\texttt{InvariantsGenerateDO}]{Coppola} for the details), which gives another proof of Proposition \ref{Prop:Invariants_Generate} in the characteristic $0$ case. 
\end{remark}

The invariants $\Iiii,\Iiii',\Iiii'',\Ivi$ are homogeneous of weight
$3,3,3,6$, respectively.  Moreover, considered as functions on the
weighted projective space $\PP^3_{3,3,3,6}$, they are algebraically independent. 
To state the classification theorems for the
reduction types of plane quartics it is also convenient to define
\begin{align*}
I = AB\Delta_a\Delta_b + AC\Delta_a\Delta_c + BC \Delta_b\Delta_c.
\end{align*}

The invariant $I$ is in the algebra generated by $\Iiii,\Iiii',\Iiii'',\Ivi$.
Concretely, we have
\begin{equation}\label{eq:I}
4I + \Ivi-\Iiii'^2+16\Iiii\Iiii''+2\Iiii'\Iiii''-\Iiii''^2 =0.
\end{equation}

We now restate Proposition \ref{Prop:norm1} in terms of these invariants:
\begin{corollary}\label{COR:Normalization_Invariants}
	After a change of variables as in Proposition~\ref{Prop:norm1} we can always work with an integer model as in Equation \eqref{eq: F} such that all the valuations of $\Iiii,\Iiii',\Iiii''$ and $\Ivi$ are non-negative and at least one is equal to zero.
\end{corollary}

\begin{proof}
	Suppose that all four invariants have positive valuation. Then we have $\nu(\Iiii) = \nu(ABC)>0$, and without loss of generality we can assume that \(\nu(A)\) is positive. Then, since $\nu(\Iiii')$, $\nu(\Iiii'')$ and \(\nu(\Ivi)\) are positive, we obtain $\nu(Bb^2+Cc^2), \nu(abc), \nu(BCb^2c^2)>0$, respectively. So we are in one of the following scenarios:
	\begin{itemize}
		\item $\nu(B), \nu(C), \nu(abc)>0$, or
		\item $\nu(B), \nu(c)>0$ or, symmetrically, $\nu(C), \nu(b)>0$, or
		\item $\nu(b), \nu(c)>0$,
	\end{itemize}
	but all contradict the normalization conditions in Proposition~\ref{Prop:norm1}. Hence the corollary follows.
   \hfill $\qed$ 
   \end{proof}

\subsection{Main results} \label{sec:main_results}

In this section we characterize the possible reduction types of a
plane quartic curve~$Y$ in terms of the valuations of the four
invariants $\Iiii,\Iiii',\Iiii'',\Ivi$.  We assume that $Y$ is given
by an equation of the form $Ax^4 +By^4 + Cz^4 + ay^2z^2 + bx^2z^2 +
cx^2y^2 = 0$, normalized as in Proposition~\ref{Prop:norm1}; in particular by
Corollary \ref{COR:Normalization_Invariants} all the invariants have
non-negative valuation, and  at least one of them has valuation zero. Additionally, we use the invariant $I$ determined by Equation \eqref{eq:I}.

In terms of the
invariants, we have
\begin{equation}\label{eq:Delta}
\Delta(Y)=-2^{-20} \Iiii\Iiii''^4\Ivi^2.
\end{equation}

\begin{proposition}\label{prop:potgoodredquartics}
	Let $Y$ be a plane quartic defined by 
	$$Ax^4 +By^4 + Cz^4
	+ ay^2z^2 + bx^2z^2 + cx^2y^2 = 0$$ normalized as in
	Proposition~\ref{Prop:norm1}. Let $\Delta(Y)$ be the discriminant of $Y$. 
	The following statements are equivalent:
	\begin{enumerate}[(i)]
		\item $Y$ has potentially good reduction,
		\item $\nu(\Delta(Y))=0$,
		\item $\nu(\Iiii)=\nu(\Iiii'')=\nu(\Ivi)=0$.
	\end{enumerate}
\end{proposition}

\begin{proof} If $\nu(\Delta(Y))=0$, then the curve has good reduction. If $Y$ has potentially good reduction then over a finite extension of the base field, it admits a plane quartic integral model with good reduction, and because of Theorem $3.15$ in \cite{LLLR}, this model can be taken in the form $$Ax^4 +By^4 + Cz^4
	+ ay^2z^2 + bx^2z^2 + cx^2y^2 = 0$$ with $A,B,C,a,b,c\in\mathcal{O}$, $\nu(\Delta(Y))=0$. In particular, with $\nu(\Iiii),\nu(\Iiii'),\nu(\Iiii''),\nu(\Ivi)\geq0$ and hence with  $\nu(\Iiii)=\nu(\Iiii'')=\nu(\Ivi)=0$. Finally, the fact that $(iii)$ implies $(ii)$ is immediate from the expression for $\Delta(Y)$ in Equation \eqref{eq:Delta}.
   \hfill $\qed$ 
   \end{proof}

In what follows we assume that
$Y$ has geometric bad reduction. Our result concerning the characterization of the possible reduction types is divided into two
statements, depending on whether $\nu(\Iiii'')$ is zero
(Theorem \ref{THM:Main_NonDegConic}) or positive
(Theorem \ref{THM:Main_DegConic}). Recall from Remark \ref{rem:DeltaX}
that $I_3''$ is the discriminant of the conic $X$ defined
by Equation \eqref{eq: G}. Hence the two cases correspond to the reduction
of this conic being non-degenerate or degenerate (see also
Remark \ref{rem:blowup}.(b)).

Theorem \ref{THM:Main_NonDegConic} is proved in
Section \ref{sec:main_proof_nondeg} and
Theorem \ref{THM:Main_DegConic} in
Section \ref{sec:main_proofs_deg}. In these sections we  also
give the Igusa invariants (respectively~the $j$-invariant) of the
irreducible components of the stable reduction of $Y$ with positive
genus in each of the cases. 

\begin{theorem}\label{THM:Main_NonDegConic}
	Let $Y$ be a plane quartic curve defined by 
	\[Ax^4 +By^4 + Cz^4 + ay^2z^2 + bx^2z^2 + cx^2y^2 = 0\]
	normalized as in
	Proposition~\ref{Prop:norm1}. Let $\Delta(Y)$ be the discriminant of
	the quartic $Y$, and let $\Delta(X)$ be the discriminant of
	the conic $X$ defined by Equation \eqref{eq: G}, which we assume to have valuation $0$, that
	is, \(\nu(\Iiii'') = 0\).
	
	Then if the valuation of \(\Delta(Y)\) is positive, $Y$ has geometric
	bad reduction and one of the cases in Table~\ref{tab:non-deg} occurs.
\end{theorem}

\renewcommand{\arraystretch}{1.25}
\begin{table}[!ht]
	\caption{Cases of Theorem~\ref{THM:Main_NonDegConic}.}
	\label{tab:non-deg}
	
	\begin{adjustbox}{max width = \textwidth}
		\begin{tabular}{cccccc>{\centering\arraybackslash}m{5.5cm}ccc}
			\hline\hline
			&       \(\nu(I_3)\)       &      \(\nu(I_3')\)       &     \(\nu(I_3'')\)      &      \(\nu(I_6)\)       &       \(\nu(I)\)        & Other conditions                                                                & Decorated graph & Stable curve &                 Lemma                 \\ \hline\hline
			(a)   &         \(= 0\)          &                          &         \(=0\)          &         \(>0\)          &         \(=0\)          &                                                                                 &       II.3       &     Loop     & \multirow{3}{*}{\ref{lemma:Case0a-c}} \\ \cline{1-9}
			(b)   &         \(= 0\)          &         \(= 0\)          &         \(=0\)          &         \(>0\)          &         \(>0\)          &                                                                                 &      III.1       &     DNA      &                                       \\ \cline{1-9}
			(c)   &         \(= 0\)          &         \(> 0\)          &         \(=0\)          &         \(>0\)          &         \(>0\)          &                                                                                 &      IV*.1       &    Braid     &                                       \\ \hline
			(d)   &         \(> 0\)          &                &         \(=0\)          &         \(=0\)          &         \(=0\)          &                                                                                 &       II.4       &     Lop      &          \ref{lemma:CaseNDd}          \\ \hline
			(e)   &         \(> 0\)          &         \(= 0\)          &         \(=0\)          &         \(>0\)          &         \(=0\)          &                                                                                 &      III.2       &    Looop     &          \ref{lemma:CaseNDe}          \\ \hline
			(f.i)  & \multirow{8}{*}{\(> 0\)} & \multirow{8}{*}{\(= 0\)} & \multirow{8}{*}{\(=0\)} & \multirow{8}{*}{\(>0\)} & \multirow{8}{*}{\(>0\)} & \(2\nu(I) > \nu(I_3)+\nu(I_6) > 2\nu(I_3)\) or \(\nu(I_3) < \nu(I) < \nu(I_6)\) &       IV.1       &   Grl Pwr    & \multirow{8}{*}{\ref{lemma:CaseNDf}}  \\ \cline{1-1}\cline{7-9}
			(f.ii)  &                          &                          &                         &                         &                         & \(2\nu(I) > \nu(I_3)+\nu(I_6) > 2\nu(I_6)\) or \(\nu(I_3) > \nu(I) > \nu(I_6)\) &       IV.3       &     Cat      &                                       \\ \cline{1-1}\cline{7-9}
			(f.iii) &                          &                          &                         &                         &                         & \(2\nu(I) > \nu(I_3)+\nu(I_6) = 2\nu(I_3)\) or \(\nu(I_3) = \nu(I) = \nu(I_6)\) &       II.1       &    Candy     &                                       \\ \cline{1-1}\cline{7-9}
			(f.iv)  &                          &                          &                         &                         &                         & \(\nu(I) < \nu(I_3), \, \nu(I) < \nu(I_6)\)                                     &       IV.2       &    Garden    &                                       \\ \cline{1-1}\cline{7-9}
			(f.v)  &                          &                          &                         &                         &                         & \(\nu(I) = \nu(I_3) < \nu(I_6)\)                                                &      III.5       &     Tree     &                                       \\ \cline{1-1}\cline{7-9}
			(f.vi)  &                          &                          &                         &                         &                         & \(\nu(I) = \nu(I_6) < \nu(I_3)\)                                                &      III.6       &  Winky Cat   &                                       \\ \hline
			(g)   &         \(> 0\)          &         \(= 0\)          &         \(=0\)          &         \(=0\)          &         \(>0\)          &                                                                                 &      III.3       &     Loop     &          \ref{lemma:CaseNDg}          \\ \hline
			(h)   &         \(> 0\)          &         \(> 0\)          &         \(=0\)          &         \(=0\)          &         \(>0\)          &                                                                                 &       IV*.3       &    Looop     &          \ref{lemma:CaseNDh}          \\ \hline\hline\\
		\end{tabular}
	\end{adjustbox}
\end{table}

\begin{theorem}\label{THM:Main_DegConic}
	Let $Y$ be a plane quartic curve defined by 
	\[Ax^4 +By^4 + Cz^4 + ay^2z^2 + bx^2z^2 + cx^2y^2 = 0\]
	normalized as in Proposition~\ref{Prop:norm1}. Let $\Delta(Y)$ be the
	discriminant of the quartic $Y$, and let $\Delta(X)$ be
	the discriminant of the conic $X$ defined
	by Equation \eqref{eq: G}, which we assume to have positive valuation, that is, \(\nu(\Iiii'') > 0\).
	
	Then the valuation of \(\Delta(Y)\) is positive, $Y$ has geometric
	bad reduction and one of the cases in Table~\ref{tab:deg} occurs. 
\end{theorem}

\begin{table}[!ht]
	\caption{Cases of Theorem~\ref{THM:Main_DegConic}.\label{tab:deg}}
	\begin{adjustbox}{max width = \textwidth}
		\begin{tabular}{cccccc>{\centering\arraybackslash}m{5.5cm}ccc}
			\hline\hline
			&       \(\nu(I_3)\)        &         \(\nu(I_3')\)          &      \(\nu(I_3'')\)      &       \(\nu(I_6)\)       &        \(\nu(I)\)        & Other conditions                                                                                                     & Decorated graph & Stable curve &                   Lemma                   \\ \hline\hline
			(a)   &          \(= 0\)          &                                &          \(>0\)          &          \(=0\)          &                          &                                                                                                                      &       II.2       &     DNA      &            \ref{lemma:CaseDa}             \\ \hline
			(b.i)  & \multicolumn{2}{c}{\multirow{3}{*}{}} & \multirow{3}{*}{\(>0\)}  & \multirow{3}{*}{\(>0\)}  & \multirow{3}{*}{\(=0\)}  & \(0 < \nu(I_3'') < \nu(I_6)\)                                                                                        &      IV*.2       &     Cat      &     \multirow{3}{*}{\ref{lem:caseb}}      \\ \cline{1-1}\cline{7-9}
			(b.ii)  &                           &                                &                          &                          &                          & \(\nu(I_3'') > \nu(I_6)>0\)                                                                                          &       IV.5       &    Grl Pwr     &                                           \\ \cline{1-1}\cline{7-9}
			(b.iii) &                           &                                &                          &                          &                          & \(\nu(I_3'') = \nu(I_6)>0\)                                                                                          &      III.4       &    Candy     &                                           \\ \hline
			(c.i)  & \multirow{12}{*}{\(= 0\)} &   \multirow{12}{*}{\(> 0\)}    & \multirow{12}{*}{\(>0\)} & \multirow{12}{*}{\(>0\)} & \multirow{12}{*}{\(>0\)} & \(2\nu(I_6) = 3\nu(I_3'') \leq 6\nu(I_3')\)                                                  &        I         &  Good (hyp)  & \multirow{12}{*}{\ref{lem:degeneratehyp}} \\ \cline{1-1}\cline{7-9}
			(c.ii)  &                           &                                &                          &                          &                          & \(2\nu(I_6) > 3\nu(I_3'')\) and \(2\nu(I_3')\geq \nu(I_3'') = \nu(I_3'^2 - 16I_3I_3'')\) &       II.3       &  Loop (hyp)  &                                           \\ \cline{1-1}\cline{7-9}
			(c.iii) &                           &                                &                          &                          &                          &  \(2\nu(I_6) > 3\nu(I_3'')\) and \(2\nu(I_3') =  \nu(I_3'') < \nu(I_3'^2 - 16I_3I_3'')\)   &      III.1       &  DNA (hyp)   &                                           \\ \cline{1-1}\cline{7-9}
			(c.iv)  &                           &                                &                          &                          &                          & \(2\nu(I_6) < 3\nu(I_3'')\) and \(3\nu(I_3')\geq \nu(I_6)\)                                                 &       II.2       &  DNA (hyp)   &                                           \\ \cline{1-1}\cline{7-9}
			(c.v)  &                           &                                &                          &                          &                          & \(3\nu(I_3') < \nu(I_3'I_3'') < \nu(I_6)\)                                                                      &      IV*.2       &  Cat (hyp)   &                                           \\ \cline{1-1}\cline{7-9}
			(c.vi)  &                           &                                &                          &                          &                          & \(3\nu(I_3')< \nu(I_6) < \nu(I_3'I_3'') \)                                                                      &       IV.5       & Grl Pwr (hyp)  &                                           \\ \cline{1-1}\cline{7-9}
			(c.vii) &                           &                                &                          &                          &                          & \(3\nu(I_3') < \nu(I_3'I_3'') = \nu(I_6)\)                                                                      &      III.4       & Candy (hyp)  &                                           \\ \hline
			(d)   &          \(> 0\)          &                                &          \(>0\)          &          \(=0\)          &          \(=0\)          &                                                                                                                      &      III.7       &     Cave     &            \ref{lemma:CaseDd}             \\ \hline
			(e)   &          \(> 0\)          &            \(= 0\)             &          \(>0\)          &                          &          \(>0\)          &                                                                                                                      &       IV.4       &    Braid     &            \ref{lemma:CaseDe}             \\ \hline\hline\\
		\end{tabular}
	\end{adjustbox}
\end{table}

A decisional tree reading the conditions of Tables \ref{tab:non-deg} and \ref{tab:deg} is given in \cite{Coppola}. Also a \SageMath~implementation is given there. It takes as input the coefficients $(A,B,C,a,b,c)$ of a smooth plane quartic in $\mathcal{M}_{3,V}$, not necessarily normalized, and outputs the reduction type computed with Theorems \ref{THM:Main_NonDegConic} and \ref{THM:Main_DegConic}.

In the following section we give a detailed proof of the two theorems, using the strategy explained in Section \ref{sec:compute}. In particular, for each case, we determine the special fiber $(\Xb,\overline{D})$ of the stably marked model of $(X,D)$, thus we obtain one of the twenty decorated graphs depicted in Appendix B and, as in the proof of Theorem \ref{FromGraphtoRedType}, we deduce the corresponding stable curve, which is the special fiber of the stable model of  $Y$ (Section~\ref{sec:stable}).

\begin{remark}
	The stable curve of type Candy corresponds to the decorated graphs II.1 (see Theorem \ref{THM:Main_NonDegConic}.(f.iii)) and III.4 (see Theorem \ref{THM:Main_DegConic}.(b.iii) and Theorem \ref{THM:Main_DegConic}.(c.vii)). Here, this can really be considered as two different reduction types, since the $j$-invariants of the elliptic curves $\Yb_1$ and $\Yb_2$  of the stable curve behave differently. 
	
	In the first case, it is shown in Lemma \ref{lemma:CaseNDf} that $j(\Yb_1)$ and $j(\Yb_2)$ depend on the value of the invariants and are in general not the same  in $k$.
	
	In the second case, $\Yb_1$ and $\Yb_2$ are isomorphic and we have that $j(\Yb_1) = j(\Yb_2) = 1728$. This is proved in Lemma \ref{lem:caseb} and Corollary \ref{cor:hyp_jinvariants}. 
	
	The difference between the two cases may be explained by considering the action of $V$ on the stable reduction $\Yb$ of $Y$. We refer to the proofs of the results for more details. 
\end{remark}

\begin{example}\label{exa:HLP}
	As an example we consider the curve
	\[
	Y:\; 2x^4 + 2y^4 + 15z^4 - 11y^2 z^2 - 11x^2z^2 + 3x^2y^2   = 0.
	\]
	Its automorphism group $\Aut_{\CC}(Y)$ is isomorphic to
	$D_4$. Additionally to the action of $V$, there is an automorphism
	$(x:y:z)\mapsto (y:-x, z)$ of order $4$. All its automorphisms are defined over $\QQ$. We have that
	\[
	\Delta(X)=I_3''=2^4, \quad \Delta(Y)=2^2\cdot 3\cdot 5\cdot 7^2.
	\]
	
	This curve has been studied by Howe--Lepr\'evost--Poonen
	\cite[Corollary~16]{HLP}. They show that the conductor of this curve is
	$N=2940=\Delta(Y)$, which is the smallest  value for a curve of genus $3$
	that we know of. The curve can also be found in Sutherland's database of non-hyperelliptic genus-$3$
	curves over $\QQ$ with small discriminant \cite{Sutherland}. In fact, it is the curve
	with smallest discriminant in this database.

	When we apply the results of this section to this curve for the primes
	$p=3,5,7$, we find that the reduction type of $Y$ is Lop for $p=3,5$ (case (d) of
	Theorem \ref{THM:Main_NonDegConic}) and Loop for $p=7$ (case (a) of
	Theorem \ref{THM:Main_NonDegConic}).
\end{example}

\section{Proofs of main results}\label{sec:main_proofs}

\subsection{Main result with non-degenerate conic}

\label{sec:main_proof_nondeg}
In this section we prove Theorem \ref{THM:Main_NonDegConic}, i.e.,~the
case that the reduction of the conic $X$ is non-degenerate. In
particular in this section we assume that $I_3''=\Delta(X)$ has
valuation zero. Equation \eqref{eq: G} defines a smooth model $\mathcal{X}_0$ of $X$ over
$\mathcal{O}$. Its special fiber, which we denote by $\Xb_0$, is an
irreducible component of the special fiber $\Xb$ of $\mathcal{X}$,
where $(\mathcal{X}, \mathcal{D})$ is the stably marked model of $(X,
D)$ from Proposition \ref{prop:stablymarked1}.(1). Recall that $D$
denotes the branch divisor of $f:Y\to X$. Moreover, the special fiber
$(\Xb, \Db)$ of the stably marked model of $(X, D)$ is a decorated
graph defined in Definition \ref{def:decorated_graph}. In the proof of
Theorem \ref{FromGraphtoRedType} we have seen that there are 20
possibilities for the decorated graph, which determine the 13
possibilities for the stable reduction of $Y$. The possibilities for
the decorated graph and the stable reduction of $Y$ are listed in
Appendix B. The strategy of the proofs is explained in
Example \ref{exa:blowup}.

We assume that $D$ splits
over $K$ and use the notation $P_a,\, P_a',\, P_b,\,  P_b',\, P_c,\, P_c'$ as
in Equation \eqref{eq:branchV4} for the $6$ branch points of $f$. 
The following lemma is useful in determining the cases for the decorated graph.

\begin{lemma}\label{lemma:specialization_of_points}
	If \(\nu(\Delta_a)\) is positive, then the points \(P_a,
	P_a'\) specialize to the same point
	in~\(\overline{X}_0\). 
	
	If \(\nu(A)\) is positive, then one
	point among \(P_b, P_b'\) and one point among \(P_c, P_c'\)
	specialize to the same point in \(\overline{X}_0\).
\end{lemma}

\begin{proof} Let $\pi$ be a uniformizing element of $K$.
	If \(\nu(\Delta_a)\) is positive, then the two roots
	of \(p_a(T)=T^2 - 2aT + 4BC=(T-a)^2-\Delta_a\) are congruent modulo \(\pi\),
	thus \(P_a, P_a'\), as
	in Equation \eqref{eq:branchV4}, specialize to the same point
	in \(\overline{X}_0\) because $\alpha^2\equiv a^2\equiv 4BC \pmod{\pi}$.
	
	Similarly, if \(\nu(A)\) is positive, one root of \(p_b(T)=T^2
	- 2bT + 4AC\) and one root of \(p_c(T)=T^2 - 2cT + 4AB\) have
	positive valuation, thus one point among \(P_b, P_b'\) and one
	point among \(P_c, P_c'\) specialize to \((1:0:0)\)
	in \(\overline{X}_0\).
   \hfill $\qed$ 
   \end{proof}

\begin{lemma}[Theorem~\ref{THM:Main_NonDegConic}, cases (a)--(c)]\label{lemma:Case0a-c} 
	Let $Y$ be as in Theorem~\ref{THM:Main_NonDegConic}, in
	particular $\nu(I_3'')=0$. Assume that \(\nu(I_3)=0\)
	and \(\nu(I_6) > 0\). Then one of the following occurs: 
	\begin{enumerate}[(a)] \item If \(\nu(I) =
		0\), then the decorated graph has type II.3 and the reduction
		type of the curve is Loop, with \(j\)-invariant \(j = 16
		(16I_3I_3'' + I)^3/(I_3I_3''I^2)\); \item if \(\nu(I_3') = 0\)
		and \(\nu(I) > 0\), then the decorated graph has type III.1
		and the reduction type of the curve is DNA; and \item
		if \(\nu(I_3') > 0\) and \(\nu(I) > 0\), then the decorated
		graph has type IV*.1 and the reduction type of the curve is
		Braid.  \end{enumerate}
\end{lemma}

\begin{proof}
	From the conditions on the invariants, it follows that the
	valuations \(\nu(A),\nu(B),$ $\nu(C)\) are zero, and in case (a)
	(respectively (b), (c)) we get that exactly one (respectively two, three) of
	the valuations \(\nu(\Delta_a),\nu(\Delta_b),\nu(\Delta_c)\) is
	positive.
	
	Lemma~\ref{lemma:specialization_of_points} implies that $P_i$ and
	$P_i'$ specialize to the same point of $\Xb_0$ if and only if the
	valuation of $\Delta_i$ is positive for $i\in \{a,b,c\}$. Moreover,
	since we have \(\nu(ABC)=0\),
	Lemma~\ref{lemma:specialization_of_points} implies that no two points
	from different pairs specialize to the same point of $\Xb_0$. Since
	the $6$ branch points specialize to at least $3$ pairwise distinct
	points of $\Xb_0$ it follows that $\Xb_0$ is an irreducible component
	of the special fiber \(\Xb\) of the stably marked model of \((X, D)\).
	From this it follows that we are in one of the cases II.3, III.1 or
	IV*.1.  We obtain that $\Xb$ is of type II.3 (respectively III.1 or
	IV*.1) if exactly one (respectively two or three) of the valuations
	$\nu(\Delta_a)$, $\nu(\Delta_b)$, $\nu(\Delta_c)$ are positive. As
	in the proof of Theorem \ref{FromGraphtoRedType} it follows that the
	stable reduction $\Yb$ of $Y$ is Loop (respectively DNA or Braid).

	To compute the $j$-invariant in case (a) assume that $\Delta_a$ is the
	discriminant with positive valuation. 
	Chose square roots $\sqrt{B}$
	and $\sqrt{C}$ with $\nu(a-2\sqrt{BC})>0$ in $\bar{K}$. Extend $K$ to contain them if needed and rename it as well as the corresponding notions $\mathcal{O}$, $\nu$ and $\pi$ for the ring of integers, valuation and uniformizer. We
	write $\mathcal{X}_0$ as
	\[
	A u^2+(b w+cv)u+G_2(v,w)=0,
	\]
	where $G_2(v,w)\equiv (\sqrt{B}v+\sqrt{C}w)^2\pmod{\pi}$. Here we use that
	$\nu(\Delta_a)$ is positive. 
	
	Write $\mathcal{Y}_0$ for the normalization of $\mathcal{X}_0$ in the
	function field of $Y$. Its special fiber $\Yb_0$ is  birationally given by
	\[
	\Yb_0:\; A x^4+(b z^2+cy^2)x^2+G_2(y^2,z^2)=0.
	\]
	There exists a change of coordinates $S\in \GL_3(K)$ such that the
	equation for $\mathcal{Y}_0$ with respect to the new variables
	$(x_1=x, y_1, z_1)$ still has integral coefficients and
	$G_2(y,z)\equiv y_1^2z_1^2\pmod{\pi}$.  Here we use the
	assumption  $\nu(ABC)=0$.
	Hence $\Yb_0$ may birationally be given by 
	\[
	\Yb_0:\; Ax_1^4+G_3(y_1,z_1)x_1^2+y_1^2z_1^2\equiv 0 \pmod{\pi}
	\]
	for some polynomial $G_3(y_1,
	z_1)=a_0z_1^2+a_1z_1y_1+a_2y_1^2\in \OO[y_1,z_1]$.  We set $z_1=1$,
	multiply the equation by $(1+a_2x^2)$, and define
	$y_2=((1+a_2x_1^2)y_1+a_1x_1^2/2)/x_1$. A short calculation shows that $\Yb_0$ is birationally given by
	\[
	\Yb_0:\; y_2^2\equiv -(1+a_2x^2)(a_0+Ax^2)+a_1^2/4x^2 \pmod{\pi}.
	\]
	This is an elliptic curve with  $j$-invariant
	\begin{equation}\label{eq:jinv}
	j \equiv \dfrac{16 (16I_3I_3'' + I)^3}{I_3I_3''I^2} \pmod{\pi}. 
	\end{equation}
	This expression is also valid if $\nu(\Delta_b)$ (resp.~$\nu(\Delta_c)$)
	is positive instead of $\nu(\Delta_a)$.
	
	We have already seen that the stable reduction $\Yb$ of $Y$ is Loop.
	It follows that the normalization of $\Yb$ is the normalization of
	$\Yb_0$. The statement on the $j$-invariant in the lemma follows.
   \hfill $\qed$ 
   \end{proof}

\begin{remark}\label{rmk:difficult} Notice that while checking the validity of Equation (\ref{eq:jinv}) is straightforward, the computation of the right-hand side from the left-hand side is not. The former was computed after some manipulations of the expression of the $j$-invariant computed with~\SageMath~taking into account the congruencies modulo $\pi$.
\end{remark}

\begin{lemma}[Theorem~\ref{THM:Main_NonDegConic}, case (d)]\label{lemma:CaseNDd} 
	Let $Y$ be as in Theorem~\ref{THM:Main_NonDegConic}, in
	particular $\nu(I_3'')=0$.
	Assume \(\nu(\Iiii)>0,
	\nu(\Ivi)=0\)
	and \(\nu(I)=0\). Then the decorated graph has type II.4 and
	the reduction type of the curve is Lop, and the Igusa
	invariants of the genus-2 curve are \begin{align*} J_2
	=& \Iiii'\Iiii'' - \Iiii''^2 + 2\Ivi + 24I ,\\
	J_4 =& \Iiii''^2\Ivi + 64\Iiii'\Iiii''I - 64\Iiii''^2I + 128\Ivi
	I + 768I^2 ,\\ 
	J_6 =& \Iiii''^2\Ivi I - 32\Iiii'\Iiii''I^2 +
	32\Iiii''^2I^2 - 64\Ivi I^2 - 256I^3 ,\\ 
	J_8
	=& \Iiii''^4\Ivi^2 + 256\Iiii'\Iiii''^3\Ivi I -
	256\Iiii''^4\Ivi I + 512\Iiii''^2\Ivi^2I + 4608\Iiii''^2\Ivi
	I^2\\ & - 32768\Iiii'\Iiii''I^3 + 32768\Iiii''^2I^3 -
	65536\Ivi I^3 - 196608I^4 ,\\ 
	J_{10}
	=& \Iiii''^4\Ivi^2I.  \end{align*}
\end{lemma}

\begin{proof}
	From the conditions on the invariants, it follows that exactly one valuation among \(\nu(A),\nu(B),\nu(C)\) is positive and that all the valuations  \(\nu(\Delta_a),\nu(\Delta_b),\nu(\Delta_c)\) are zero.
	Assume  \(\nu(A) > 0\). Then, by Lemma~\ref{lemma:specialization_of_points}, a point with inertia generator \(\sigma_b\) and a point with inertia generator \(\sigma_c\)  both specialize to \((1:0:0)\), hence the decorated graph has type II.4 and the reduction of the curve is Lop.

	We determine an equation for the normalization of the stable reduction
	$\Yb$ of $Y$.  Since we are in case Lop this is a curve of genus $2$. Arguing as in the proof of Lemma \ref{lemma:Case0a-c}, we find that 
	$\Yb$ is birationally given by
	$$
	\Yb:\; x^2\equiv -\frac{By^4+ay^2z^2+Cz^4}{cy^2+bz^2}.
	$$
	Hence $t^2=(x(cy^2+bz^2))^2=-(cy^2+bz^2)(By^4+ay^2z^2+Cz^4)$ is the
	genus-$2$ curve we are looking for. We computed its Igusa invariants with \SageMath. As mentioned in Remark \ref{rmk:difficult}, the equalities in the statement of the lemma are straightforward to check, while stating them was not and required smart manipulations.
   \hfill $\qed$ 
   \end{proof}

\begin{lemma}[Theorem~\ref{THM:Main_NonDegConic}, case (e)]\label{lemma:CaseNDe} 
	Let $Y$ be as in Theorem~\ref{THM:Main_NonDegConic}, in
	particular $\nu(I_3'')=0$.
	Assume \(\nu(\Iiii)>0,\nu(\Iiii')=0, \nu(\Ivi)>0\)
	and \(\nu(I)=0\). Then the decorated graph has type III.2 and
	the reduction type of the curve is Looop.
\end{lemma}

\begin{proof} From the conditions on the invariants, it follows that exactly one valuation among \(\nu(A),\nu(B),\nu(C)\) and exactly one among \(\nu(\Delta_a),\nu(\Delta_b),\nu(\Delta_c)\) are positive while only one among \(\nu(A\Delta_a), \nu(B\Delta_b), \nu(C\Delta_c)\) is also positive.
	
	Assume \(\nu(A\Delta_a)>0\). Then, by Lemma~\ref{lemma:specialization_of_points}, the points \(P_a\) and \(P_a'\) both specialize to \((0:1:-1)\), and a point with inertia generator \(\sigma_b\) and a point with inertia generator \(\sigma_c\)  both specialize to \((1:0:0)\), hence the decorated graph has type III.2 and the reduction of the curve is Looop. 
   \hfill $\qed$ 
   \end{proof}

\begin{lemma}[Theorem~\ref{THM:Main_NonDegConic}, case (f)] \label{lemma:CaseNDf}
	Let $Y$ be as in Theorem~\ref{THM:Main_NonDegConic}, in
	particular
	$\nu(I_3'')=0$. Assume \(\nu(\Iiii)>0, \nu(\Iiii')=0, 
	\nu(\Ivi)>0\)
	and \(\nu(I)>0\). Then one of the following occurs:
	\begin{enumerate}[(i)] \item If \(2\nu(I)
		> \nu(I_3)+\nu(I_6) > 2\nu(I_3)\) or \(\nu(I_3)
		< \nu(I) < \nu(I_6)\), then the decorated graph has
		type IV.1 and the reduction type of the curve is Grl
		Pwr.
		\item If \(2\nu(I)
		> \nu(I_3)+\nu(I_6) > 2\nu(I_6)\) or \(\nu(I_3)
		> \nu(I) > \nu(I_6)\), then the decorated graph has
		type IV.3 and the reduction type of the curve is
		Cat. 
		\item If \(2\nu(I)
		> \nu(I_3)+\nu(I_6) = 2\nu(I_3)\) or \(\nu(I_3)
		= \nu(I) = \nu(I_6)\), then the decorated graph has
		type II.1 and the reduction type of the curve is
		Candy, and the $j$-invariants of the two genus-$1$
		components of the special fiber are the roots of the
		polynomial 
		\begin{align*}
		I_6^2I_3I_3't^2 -
		& 2^4\left(I_6^2I+3\cdot 2^4I_3I_3'I_6^2+ 3\cdot
		2^8I_3I_3'I_6I-
		2^{13}I_3^2I_3'^2I_6+2^{12}I_3I_3'I^2\right)t \\
		& + 2^8\left(I_6 + 2^4I + 2^8I_3I_3'\right)^3
		\end{align*}
		\item If \(\nu(I) < \nu(I_3)$ and $\nu(I) < \nu(I_6)\), then the decorated graph has type IV.2 and the reduction type of the curve is Garden.
		\item If \(\nu(I) = \nu(I_3) < \nu(I_6)\), then the decorated graph has type III.5, the reduction type of the curve is Tree, and the $j$-invariant of the genus-1 component of the special fiber is $j = 2^4(I+2^4I_3I_3')^3/(I^2I_3I_3')$.  
		\item If \(\nu(I) = \nu(I_6) < \nu(I_3)\), then the decorated graph has type III.6, the reduction type of the curve is Winky Cat, and the $j$-invariant of the genus-1 component of the special fiber is $j = 2^4(I_6+2^4I)^3/(I_6^2I)$.
	\end{enumerate}
\end{lemma}

\begin{proof}
	The conditions on the invariants imply that we may assume 
	\[\nu(C \Delta_c) = 0,\quad  \nu(A)>0, \textrm{ and }  \nu(B\Delta_b) > 0\]
	after permuting the variables, if necessary.
	In order to determine the stable reduction of $Y$ in the different subcases, we use two different coordinates for $X$. 
	The coordinates 
	\[
	\xi_1 = \dfrac{\beta u + \alpha v + 2Cw}{2(\alpha-a) v} \qquad \xi_2 = \dfrac{\beta u + \alpha v + 2Cw}{2(\beta - b)u},
	\]
	correspond to $t_1=( P_a', P_a, P_b')$ and $t_2=(P_b, P_b', P_a)$ in
	the notation of Section \ref{sec:compute}. The coordinates $\xi_1$ and
	$\xi_2$ define models $\mathcal{X}_1$ and $\mathcal{X}_2$ of $X$,
	which may or may not be isomorphic over $\mathcal{O}$.  We write
	$\Xb_1$ and $\Xb_2$ for the special fibers of the corresponding
	models. We use the same notation for further coordinates we introduce in the course of the proof. The coordinate $\xi_1$ is the same we considered in
	Example \ref{exa:blowup}, which corresponds to case (vi) of the
	current lemma.
	
	In this proof, we choose $\alpha$ to be a root of $p_a(T)=T^2-2aT+4BC$ of minimal valuation. Similarly, we choose  $\beta$ and $\gamma$ to be  a root of $p_b(T)=T^2-2bT+4AC$  and $p_c(T)=T^2-2cT+4AB$ of minimal valuation.
	\begin{enumerate}[(i)]
		\item If \(2\nu(I) > \nu(I_3)+\nu(I_6) > 2\nu(I_3)\) or 
		\(\nu(I_3) < \nu(I) < \nu(I_6)\), then $\nu(\Delta_a) > \nu(B)$ and 
		$\nu(\Delta_b)>\nu(A)$. Moreover, $\nu(A\Delta_a)=\nu(B\Delta_b)$.
		
		We may assume, without 
		loss of generality, that
		$\nu(B)\leq \nu(A)$.  It follows that
		$2\nu(\beta)=\nu(A)\geq \nu(B)=2\nu(\alpha)=2\nu(a)$ and
		$2\nu(\alpha-a)=\nu(\Delta_a)$.
		
		We define a new coordinate
		\[
		\xi_3= \dfrac{\beta \tilde{\pi}_3u + \alpha \pi_3 v + 2C\pi_3w}{2(\alpha-a) v},
		\]
		where $\pi_3$ and $\tilde{\pi}_3$ are chosen such that $\nu(-\beta
		B\tilde{\pi}_3+\alpha\gamma\pi_3)=\nu(\alpha -a).$ (This is obviously
		possible.)  One computes that the points \(P_a\) and \(P_a'\) both
		specialize to $\xi_3=0$ on $\Xb_3$, the points \(P_b\), \(P_b'\),
		and \(P_c\) specialize to $\xi_3=\infty$ on $\Xb_3$, and $\xi_3(P_c')$
		specializes to a point with $\xi_3\neq 0,\infty$ on $\Xb_3$.
		
		Similarly, we find a coordinate $\xi_4$ such that the
		points \(P_a\), \(P_a'\), and $P_c'$ specialize to $\xi_4=0$ on
		$\Xb_4$, the points \(P_b\) and \(P_b'\) specialize to $\xi_4=\infty$ on
		$\Xb_4$, and $\xi_3(P_c)$ specializes to a point with $\xi_4\neq
		0,\infty$ on $\Xb_4$. We conclude that the decorated graph is of type
		IV.1, and the irreducible components of $\Xb$ are
		$\Xb_1, \Xb_3, \Xb_4, \Xb_2$ from left to right. The reduction type of the curve is Grl Pwr.
		
		\item If \(2\nu(I) > \nu(I_3)+\nu(I_6) > 2\nu(I_6)\) or \(\nu(I_3) > \nu(I) > \nu(I_6)\), then $\nu(\Delta_a) < \nu(B)$ and $\nu(\Delta_b)<\nu(A)$. 
		We may assume that $\nu(\Delta_b)\leq\nu(\Delta_a)$. 
		
		One calculates that the points $P_a$ and $P_c'$ both specialize to the
		point $\xi_1=1$, the points $P_b, P_b'$, and $P_c$ specialize to the
		point $\xi_1=\infty$, and $P_a'$ specializes to $\xi_1=0$ on
		$\Xb_1$. Similarly, one computes that the points $P_a, P_a', P_c'$ specialize to the point $\xi_2=\infty$, the points $P_b'$ and $P_c$ to the point $\xi_2=1$, and $P_b$ to $\xi_2=0$ on $\Xb_2$.
		We conclude that the  decorated graph has type IV.3: the irreducible components $\Xb_1$ and $\Xb_2$ are the two middle components. The reduction type of the curve is Cat.
		
		\item If \(2\nu(I) \geq \nu(I_3)+\nu(I_6) =
		2\nu(I_3)\), then $\nu(\Delta_a) = \nu(B)$ and
		$\nu(\Delta_b) = \nu(A)$. We may assume, without 
		loss of generality, that $\nu(B) \leq \nu(A)$.

		One computes that $P_b, P_b'$, and $P_c$ all specialize to the point
		$\xi_1=\infty$ on $\Xb_1$. Moreover, the point $P_c'$ specializes to a
		point with $\xi_1\neq 0,1,\infty$ on $\Xb_1$. In particular, the
		points $P_a, P_a', P_c'$ and $P_b'$ specialize to pairwise distinct
		points of $\Xb_1$. Similarly, one checks that the points $P_b, P_b',
		P_c$, and $P_a$ specialize to pairwise distinct points on
		$\Xb_2$.  Hence $\Xb$ has type II.1 and the reduction type of the curve 
		is Candy.

		The stable reduction $\Yb$ of $Y$ consists of two genus-$1$ curves
		intersecting in two points. To calculate their $j$-invariants we
		proceed as in Example \ref{exa:blowup}. Let $\Yb_i$ be the
		irreducible component of $\Yb$ above $\Xb_i$ for $i=1,2$. The
		coordinate $\xi_1$ of $\Xb_1$ is identical to the coordinate from
		Example \ref{exa:blowup}, hence we find the same expression
		\[
		j(\Yb_1)\equiv \frac{2^6(a^2+12BC)^3}{\Delta_a^2\cdot
			4BC}\pmod{\pi}.  
		\] 
		A similar calculation yields 
		\[
		j(\Yb_2)\equiv \frac{2^6(b^2+12AC)^3}{\Delta_b^2\cdot
			4AC}\pmod{\pi}.  
		\] 
		
		One checks that  the
		$j$-invariants  $j(\Yb_1)$ and $j(\Yb_2)$  are the roots of the
		polynomial given in the statement.

		\item Assume \(\nu(I) < \nu(I_3)\) and \( \nu(I)
		< \nu(I_6)\). We may assume, without 
		loss of generality, that
		$\nu(\Delta_b)<\nu(A)$. It follows that
		$\nu(B)< \nu(\Delta_a)$.
		
		One computes that $P_a, P_a'$, and $P_c'$ specialize to the point
		$\xi_2=\infty$ and $P_b'$ and $P_c$ specialize to the point $\xi_2=1$
		on $\Xb_2$. By definition of $\xi_2$ the point $P_b$ specializes to
		$\xi_2=0$.

		Define
		\[
		\xi_3=\frac{\gamma(1-2B)u+4B^2v+2\alpha B w}{-2Bv}.
		\]
		Then the points $P_a$ and $P_a'$ specialize to $\xi_3=0$, the points
		$P_b, P_b'$, and $P_c$ specialize to $\xi_3=\infty$, and $P_c'$
		specializes to $\xi_3=1$ on $\Xb_3$.  We conclude that the decorated
		graph has type IV.2. The components are $\Xb_1$, $\Xb_3$, $\Xb_2$, and
		a fourth one to which we did not give a name. The reduction type of the
		curve is Garden.
		
		\item Assume \(\nu(I) = \nu(I_3) < \nu(I_6)\). We may assume, without 
		loss of generality, that $\nu(\Delta_b)
		= \nu(A)$. It follows that $\nu(B)< \nu(\Delta_a)$.
		
		\textbf{Case 1}: We first consider the case that $\nu(A)\leq \nu(B)$.
		
		Lemma \ref{lemma:specialization_of_points} implies that $P_a, P_a'$, and $P_c'$ specialize to the same point $(0:1:0)$ of $\Xb_0$. The points $P_b$ and $P_b'$ specialize to the same point of $\Xb_0$, as well, and that $P_a, P_c$, and $P_b$ specialize to pairwise distinct points of $\Xb_0$. 
		As in the previous cases, we may check that $\xi_1(P_c')\not\equiv  0, 1, \infty \pmod{\pi}$.   We conclude that the decorated graph has type III.5. The irreducible components are $\Xb_1$, $\Xb_0$, and $\Xb_2$ from left to right. 
		
		The reduction type of $\Yb$ is  Tree. The genus-$1$ component $\Yb_1$ of the stable reduction is the normalization of the component corresponding to coordinate $\xi_1$. Again, we are in the situation of Example \ref{exa:blowup} and get
		\[
		j(\Yb_1)\equiv \frac{2^6(a^2+12BC)^3}{\Delta_a^2\cdot 4BC}\pmod{\pi}.
		\]
		In terms of invariants this can be expressed as 
		\[
		j(\Yb_1)\equiv 2^4(I+2^4I_3I_3')^3/(I^2I_3I_3').
		\]
		
		\textbf{Case 2}: 
		If we are not in case 1 then  $\nu(A)>\nu(B)$. In this case $P_b, P_b'$, and $P_c$ (resp.~$P_a,P_a'$) specialize to the same point of $\Xb_0$, and $P_a, P_b, P_c'$ specialize to pairwise distinct points on $\Xb_0$. Moreover, $\xi_2(P_c')\not\equiv 0,1,\infty\pmod{\pi}$. 
		
		As in the previous case, the decorated graph has type III.5 and the reduction type of $\Yb$ is tree.   The component of genus $1$ is the unique irreducible component $\Yb_2$ above the component $\Xb_2$ corresponding to the coordinate $\xi_2$.  We get 
		\[
		j(\Yb_2) \equiv \frac{2^6(b^2+12AC)^3}{\Delta_b^2\cdot 4AC}\pmod{\pi}.
		\]
		In terms of invariants we get the same expression as in the above case, 
		\[
		j(\Yb_2)\equiv 2^4(I+2^4I_3I_3')^3/(I^2I_3I_3').
		\]

		\item Assume that \(\nu(I) = \nu(I_6) < \nu(I_3)\). We may assume, without 
		loss of generality, that 
		$\nu(\Delta_b) < \nu(A)$ and $\nu(B) = \nu(\Delta_a)$.
		
		One computes that the points $P_a, P_a', P_b'$, and $P_c'$ specialize
		to pairwise distinct points on $\Xb_1$. Moreover, the
		points \(P_b\), \(P_b'\) and \(P_c\) specialize to the same point of
		$\Xb_1$. 
		
		The points $P_a, P_a'$, and $P_c'$ specialize to the same point
		$\xi_1=\infty$, the points $P_b'$ and $P_c$ specialize to the point
		$\xi_2=1$, and $P_b$ specializes to the point $\xi_2=0$ on
		$\Xb_2$. This shows that $\Xb$ has type III.6: the irreducible components are $\Xb_1$, $\Xb_2$, and a further component from left to right. The reduction type of the curve $Y$ is Winky Cat. 
		
		As in Example \ref{exa:blowup} one computes that the $j$-invariant of the irreducible component $\Yb_1$ above $\Xb_1$ is 
		\[
		j(\Yb_1) \equiv \frac{2^6(a^2+12BC)^3}{\Delta_a^2\cdot 4BC}\pmod{\pi}.
		\]
		In terms of invariants, this is 
		\[
		j(\Yb_1)\equiv 2^4(I_6+2^4I)^3/(I_6^2I).
		\]
	\end{enumerate}
   \hfill $\qed$ 
   \end{proof}

\begin{lemma}[Theorem~\ref{THM:Main_NonDegConic}, case (g)]\label{lemma:CaseNDg} 
	Let $Y$ be as in Theorem~\ref{THM:Main_NonDegConic}, in
	particular $\nu(I_3'')=0$.  Assume
	that \(\nu(\Iiii)>0,\nu(\Iiii')=0,\nu(\Ivi)=0\)
	and \(\nu(I)>0\). Then the decorated graph has type III.3 and
	the reduction type of the curve is Loop.  The \(j\)-invariant
	of the genus-1 component is 
	\[j = 16\dfrac{(I_3''^2 -
		16I_3'I_3'' + 16I_3'^2)^3}{I_3'I_3''^4(I_3' - I_3'')}.\]
\end{lemma}

\begin{proof}
	From the conditions on the invariants, it follows that
	$\nu(\Delta_a) = \nu(\Delta_b) = \nu(\Delta_c) = 0$ and
	exactly two among $\nu(A), \nu(B), \nu(C)$ are positive. Now
	Lemma \ref{lemma:specialization_of_points} implies that branch
	points with the same inertia generator do not specialize to
	the same point on $\Xb_0$. Without loss of generality, we may
	assume that $A$ and $B$ have positive valuation and that
	$\nu(C)=0$. We conclude that the points $P_a, P_a', P_b, P_b'$
	specialize to pairwise distinct points of $\Xb_0$.
	
	Up to possibly interchanging $P_c$ and $P_c'$, it follows from
	Lemma \ref{lemma:specialization_of_points}  that $P_c$
	specializes to the same point as one of $\{P_b,P_b'\}$ on
	$\Xb_0$. The same argument using that $\nu(B)>0$ implies that
	$P_c'$ specializes to the same point as one of $\{P_a,P_a'\}$
	on $\Xb_0$.
	
	Hence the decorated graph has
	type III.3: the irreducible component $\Xb_0$ is the middle component of $\Xb$. The reduction type of the curve is Loop.
	
	Let $\Yb_0$ be the irreducible component above $\Xb_0$ of the stable
	reduction $\Yb$ of $Y$. To compute $j(\Yb_0)$ we argue as in the proof
	of Lemma \ref{lemma:Case0a-c}. After applying a suitable coordinate
	change in $\GL_3(K)$  on $Y$, we find a birational equation for $\Yb_0$:
	\[
	\Yb_0:\; (cx^2+az^2)y^2+bx^2z^2+Cz^4=0.
	\] 
	Setting $x=1$ we recognize $\Yb_0$ as an elliptic curve and find 
	$$j(\Yb_0) = \dfrac{16(a^2b^2 + 14abcC + c^2C^2)^3}{abcC(ab - cC)^4} \equiv \dfrac{16(I_3''^2 - 16I_3'I_3'' + 16I_3'^2)^3}{I_3'I_3''^4(I_3' - I_3'')} \pmod{\pi}.$$
   \hfill $\qed$ 
   \end{proof}

\begin{lemma}[Theorem~\ref{THM:Main_NonDegConic}, case (h)]\label{lemma:CaseNDh} 
	Let $Y$ be as in Theorem~\ref{THM:Main_NonDegConic}, in
	particular $\nu(I_3'')=0$.
	Assume \(\nu(\Iiii)>0,\nu(\Iiii')>0, \nu(I) >0\),
	and \(\nu(\Ivi)=0\). Then the decorated graph has type IV*.3
	and the reduction type of the curve is Looop.
\end{lemma}

\begin{proof}
	From the conditions on the invariants, it follows that
	$\nu(\Delta_a) = \nu(\Delta_b) = \nu(\Delta_c) = 0$ and
	$\nu(A), \nu(B), \nu(C)$ are
	positive. Lemma \ref{lemma:specialization_of_points} implies
	that branch points with the same inertia generator do not
	specialize to the same point on $\Xb_0$. Moreover, for every
	pair~$i\neq j\in \{a,b,c\}$ one of the branch points with
	inertia generator $\sigma_i$ and one of the branch points with
	inertia generator $\sigma_j$ specialize to the same point of
	$\Xb_0$. We conclude that  the
	decorated graph has type IV*.3.
   \hfill $\qed$ 
   \end{proof}

\subsection{Main result with degenerate conic} \label{sec:main_proofs_deg}

As in Section \ref{sec:main_proof_nondeg} we write $\mathcal{X}_0$
for the model of $X$ defined by Equation \eqref{eq: G} and $\Xb_0$ for its
special fiber. Since we assume that the left-hand side of Equation \eqref{eq: G}  for $X$ is normalized as in Proposition \ref{Prop:norm1}, $\mathcal{X}_0$ is indeed a model, and $\Xb_0$ is a conic over the residue field $k$ of $K$. 
In this section we prove
Theorem \ref{THM:Main_DegConic}, which treats the case that $\Xb_0$ is
degenerate. Recall that this implies that $I_3''=\Delta(X)$ has
positive valuation. The
classification of degenerate conics in characteristic different from $2$ implies therefore that $\Xb_0$ is either reducible or non-reduced. In the first case, $\Xb_0$ consist of two irreducible components. In the second case the underlying reduced scheme 
$\Xb_0^{\emph{red}}$ is irreducible.

\begin{proposition}\label{prop:degeneratedconicproperties}
	Assume that \(\nu(I_3'')\) is positive. 
	\begin{enumerate}[(i)] \item The curve $\Xb_0$ is reducible and 
		the points \(P_a, P_a'\) specialize to different irreducible components of $\Xb_0$  if and only if \(\nu(\Delta_a)\) is
		zero.  
		\item The curve $\Xb_0$ is non-reduced and $\Xb_0^{\emph{red}}$ is irreducible  if and
		only if  \(\nu(\Delta_a)\), \(\nu(\Delta_b)\)
		and \(\nu(\Delta_c)\) are all positive.  
		\item Assume that $\Xb_0$ is reduced and $\nu(C)=0$. Let \(\alpha\)
		(respectively \(\beta\)) be a root of \(p_a(T)=T^2 - 2aT +
		4BC\) (respectively \(p_b(T)=T^2 - 2bT + 4AC\)). Then the two
		points \(P_a'=(0:-2C:\alpha)$ and $ P_b=(-2C:0:\beta)\)
		specialize to the same irreducible component
		of \(\overline{X}_0\) if and only if the valuation
		of \(\alpha(2b - \beta) + (2a - \alpha)\beta - 4Cc\) is
		positive.  \end{enumerate}
\end{proposition}

\begin{proof}
	\begin{enumerate}[(i)]
		\item If \(\nu(\Delta_a)\) is positive, then the two roots of \(T^2 - 2aT + 4BC\) are congruent modulo \(\pi\), thus  \(P_a, P_a'\) specialize to the same point on \(\overline{X}_0\).

		Assume that $\nu(\Delta_a)=0$ and that
		\(P_a, P_a'\)
		specialize to  the same irreducible component of
		$\Xb_0$.  This also includes the case that $\Xb_0^{\text{red}}$ is reduced. We denote the irreducible component of $\Xb_0 ^{\text{red}}$ to which $P_a,P_a'$ specialize by $\Xb_1$. It follows
		from Equation \eqref{eq: G} that $u=0$ is an equation for $\Xb_1$.  But \(u\) is a factor of the left-hand side of Equation \eqref{eq: G} $\pmod{\pi}$ if and only if the valuations of \(B, C,\)
		and \(a\) are positive, which contradicts
		the assumption $\nu(\Delta_a)=0$.  Statement (i) follows.
		
		\item Assume that $\Xb_0$ is non-reduced. Then
		$\Xb_0^{\text{red}}$ is irreducible and the
		left-hand side of Equation~\eqref{eq: G} modulo $\pi$ is a
		square. We conclude that we may choose square roots of $A, B, C\pmod{\pi}$ such that 
		\[
		2\sqrt{AB}\equiv c, \quad 2\sqrt{AC}\equiv b, \qquad 2\sqrt{BC}\equiv a.
		\]  
		This implies that that valuations of \(\Delta_a\), \(\Delta_b\)
		and \(\Delta_c\) are positive. The converse is
		similar. Statement (ii) follows.
		
		\item Assume that $\Xb_0$ is reduced and
		$\nu(C)=0$.   We write $\Xb_1$ and $\Xb_2$ for the
		irreducible components of $\Xb_0$.
		
		Note that if $P_a'$ and $P_b$  specialize to the same irreducible component $\Xb_1$ of $\Xb_0$, then $\Xb_1$  is defined   by 
		\[
		\Xb_1:\; \beta u
		+ \alpha v
		+ 2C w = 0.
		\]

		We write \(\alpha'=2a-\alpha\) (resp.~\(\beta'=2b-\beta\)) for the
		second root of $p_a$ (resp.~$p_b$). Statement (i) implies that the
		points $P_a=(0:\alpha:-2B)=(0:-2C:\alpha')$ and
		$P_b'=(\beta:0:-2A)=(-2C:0:\beta')$ specialize to $\Xb_2$, which may
		be given by
		\[
		\Xb_2:\;  \beta' u
		+ \alpha' v + 2C w = 0.
		\]
		Computing the product of the equations for $\Xb_1$ and $\Xb_2$ we obtain 
		
		\begin{gather*} (\beta u
		+ \alpha v + 2C w)(\beta' u + \alpha' v + 2C w) =\\
		4CA u^2 + 4BCv^2 +
		4C^2w^2 + 4Cavw + 4Cbuw + (\alpha\beta' + \alpha'\beta)uv.\end{gather*} 
		The right-hand side of this equation is congruent to the left-hand side of Equation \eqref{eq: G}
		if and only if \[\nu(\alpha(2b
		- \beta) + (2a - \alpha)\beta - 4Cc) >0.
		\] \end{enumerate}
   \hfill $\qed$ 
   \end{proof}

\begin{lemma}[Theorem~\ref{THM:Main_DegConic}, case (a)] \label{lemma:CaseDa}
	Let $Y$ be as in Theorem \ref{THM:Main_DegConic}, in particular
	we have $\nu(I_3'')>0$.  Assume that \(\nu(\Iiii)=\nu(\Ivi)=0\). Then the
	decorated graph has type II.2 and the reduction of the curve is DNA.
\end{lemma}

\begin{proof}
	From the conditions on the invariants, it follows that the
	valuations of \(A,B,C,a,b,c\) are zero. Then, by Lemma~\ref{lemma:specialization_of_points},
	all $6$ branch points specialize to pairwise distinct points
	of $\Xb_0$. Proposition~\ref{prop:degeneratedconicproperties}
	implies that $\Xb_0$ is reducible, and the branch points with the same
	inertia generator specialize to different irreducible components of
	~$\Xb_0$. Therefore the decorated graph has type II.2, thus the
	reduction of the curve is~DNA.
   \hfill $\qed$ 
   \end{proof}

\begin{lemma}[Theorem~\ref{THM:Main_DegConic}, case (b)]\label{lem:caseb} 
	Let $Y$ be as in Theorem \ref{THM:Main_DegConic}, in particular $\nu(I_3'')>0$. Assume that 
	\(\nu(\Ivi)>0\) and $\nu(I)=0$. Then one of the following occurs: 
	\begin{enumerate}[(i)]
		\item if \(\nu(\Ivi) > \nu(\Iiii'')\), the decorated graph has type IV*.2 and the reduction of the curve is Cat;
		\item if \(\nu(\Ivi) < \nu(\Iiii'')\), the decorated graph has type IV.5 and the reduction of the curve is Grl Pwr; and
		\item otherwise, the decorated graph has type III.4, the reduction type of the curve is Candy, and the $j$-invariants of the genus-$1$ components of the special fiber are $j_1 = j_2 = 1728$.
	\end{enumerate}
\end{lemma}

\begin{proof} It follows from the conditions $\nu(\Ivi)>0$ and $\nu(I)=0$ that exactly one among $\nu(\Delta_a),\nu(\Delta_b),\nu(\Delta_c)$ is positive.  Without loss of generality we may assume that
	$\nu(\Delta_a)>0$. Then $\nu(I) =0$ also implies $\nu(BC) = 0$. It follows from Proposition \ref{prop:degeneratedconicproperties} that $\Xb_0$ is reduced, has two irreducible components and the points $P_a, P_a'$ both specialize to the intersection of these two components.

	Now, fix \(\beta\) and \(\gamma\) to be roots of \(p_b(T)=T^2 - 2bT +
	4AC\) and \(p_c(T)=T^2 - 2cT + 4AB\), respectively,  that have valuation $0$. 
	
	\textbf{Claim}: We  may choose $\beta$ and $\gamma$ so that they additionally satisfy \(\nu((2b - \beta)\gamma + \beta(2c - \gamma) - 4Aa)
	>0\). Moreover, there exist coordinates $\xi_1$ and $\xi_2$ such that $P_a$ and $P_a'$ specialize to the intersection point $\tau$ of the corresponding irreducible components $\Xb_1$ and $\Xb_2$ and  $P_b'$ and $P_c$ (respectively~$P_b$ and $P_c'$) specialize to pairwise distinct points of  $\Xb_1$ (respectively~$\Xb_2$) different from $\tau$.
	
	Assume first that $\nu(A) = 0$. Then all the roots of $p_b$ and $p_c$ have valuation $0$. In this case it follows from the proof of Proposition \ref{prop:degeneratedconicproperties}.(iii) that we may choose $\beta$ and $\gamma$ as in the claim.  That proof then also implies that $P_b'$ and $P_c$ (resp.~$P_b, P_c'$) specialize to the same component of $\Xb_0$. Moreover, Proposition \ref{prop:degeneratedconicproperties}.(i) implies that the points $P_b, P_b',P_c, P_c'$ specialize to pairwise distinct points on $\Xb_0$. This proves the claim in this case.
	
	Next assume that $\nu(A)$ is positive. Then  \(\nu((2b - \beta)\gamma + \beta(2c - \gamma) - 4Aa)
	>0\) for any choice of $\beta$ and $\gamma$, but there are unique roots $\beta$ of $p_b$  and $\gamma$ of $p_c$ with $\nu(\beta) = \nu(\gamma) = 0$. With this choice, we have that $P_b'$ and $P_c$ both specialize to the point \((1:0:0)\) in the smooth locus of $\Xb_0$ and in particular to the same irreducible component of $\Xb$.  Let $\xi_1$ and $\xi_2$ be coordinates corresponding to $t_1=(P_b', P_c, P_a)$ and $t_2=(P_b, P_c',P_a)$ as in Section \ref{sec:compute}. The corresponding components $\Xb_1$ and $\Xb_2$ satisfy the conditions in the claim.

	It remains to compute a component separating $P_a, P_a'$. Consider
	\[ 
	\xi_3 = \dfrac{\beta u + \alpha v + 2Cw}{2v(\alpha - a)}.
	\]
	This is the coordinate corresponding to $t_3=(P_a', P_a, P_b')$.
	One checks that \(\overline{\xi_3(P_c)} = \infty\). The decorated graph of \(X\) depends on the value of \(\overline{\xi_3(P_b)} = \overline{\xi_3(P_c')}\).
	Using Equation \eqref{eq: G} and that $\beta$ is a root of $p_b$ we find the equivalent expression for the coordinate
	\[\xi_3 = \dfrac{-2\beta(cu + Bv + aw) + \alpha(2Au + \beta w)}{2(\alpha-a)(2Au + \beta w)}. \]
	We obtain
	\[\xi_3(P_b) = \dfrac{4cC  - \alpha(2b - \beta)- (2a - \alpha)\beta}{4(\alpha - a)(\beta - b)}.\]
	
	Note that the discriminant factors as follows:
	\[\Delta(X) = (4cC  - \alpha(2b - \beta)- (2a - \alpha)\beta)(4cC - \alpha\beta - (2a - \alpha)(2b - \beta)).\]
	
	Let \(\delta = \nu(- 4Cc + \alpha(2b-\beta) + (2a - \alpha)\beta )\) and note
	\[
	\begin{aligned}
	\nu( - 4Cc + \alpha\beta + (2a - \alpha)(2b - \beta)) &= \nu(4(\alpha - a)(\beta - b)  - 4Cc + \alpha(2b - \beta) + (2a - \alpha)\beta) \\&\geq \min(\nu((\alpha - a)(\beta - b)), \delta),
	\end{aligned}
	\]
	where \(\nu(\alpha -a) = \frac{1}{2}\nu(\Delta_a)\). Then, we are in one of the following cases:
	\begin{itemize}
		\item If \(2\delta < \nu(\Delta_a)\), then one has \(\nu(\Delta(X)) = 2\delta < \nu(\Delta_a)\) and \(\overline{\xi_3(P_b)} = \infty\). In this case, \(\overline{X}\) is of type IV*.2. The component $\Xb_3$ is the left most vertical component. 
		\item If \(2\delta > \nu(\Delta_a)\), then one has \(\nu(\Delta(X)) = \delta + 1/2\nu(\Delta_a) > \nu(\Delta_a)\), and \(\overline{\xi(P_b)} = 0\). In this case, \(\overline{X}\) is of type IV.5. The component $\Xb_3$ is the one in the middle that intersects $\Xb_1$.  
		\item If \(2\delta = \nu(\Delta_a)\) and \(\nu(\alpha\beta + (2a - \alpha)(2b - \beta) - 4Cc)>\delta\), then one has \(\nu(\Delta(X)) > \nu(\Delta_a)\), and \(\overline{\xi(P_b)} = 1\). In this case, \(\overline{X}\) is again of type IV.5, but the component $\Xb_3$ is the one in the middle that intersects $\Xb_2$.  
		\item Lastly, if \(2\delta = \nu(\Delta_a)\) and \(\nu(\alpha\beta + (2a - \alpha)(2b - \beta) - 4Cc) = \delta\), then one has \(\nu(\Delta(X)) = \nu(\Delta_a)\), and \(\overline{\xi(P_b)} \neq 0, 1, \infty\). In this case, \(\overline{X}\) is of type III.4. The component $\Xb_3 $ is the central one.
	\end{itemize}
	This yields the case distinction from the statement of the lemma. 
	
	It remains to compute the $j$-invariant of the two components of the stable reduction $\Yb$ of $Y$ in case (iii) of the lemma. In this case $\Yb$ consists of two genus-$1$ curves $\Yb_1$ and $\Yb_2$ intersecting in two points. These two curves are permuted by the action of $\sigma_b$ and $\sigma_c$. Hence $\Yb_1 \simeq \Yb_2$. The map $f:Y\to X$ extends to a finite and flat map $\fb:\Yb\to \Xb_3$, where $\Xb_3$ is the special fiber of the model of $X$  defined by the coordinate $\xi_3$. Moreover, $\fb$ is Galois with Galois group $V$.

	The map $\fb:\Yb \to \Xb$ factors as $\Yb \to \Yb/\langle \sigma_i \rangle \to \Xb$ for $i \in \{a,b,c\}$. We write $\Yb^{\circ} = \Yb_1\sqcup \Yb_2$ for the normalization of $\Yb$. The induced map $\fb: \Yb^{\circ}\to \Yb/\langle \sigma_i \rangle$ for $i=b,c$ just identified the two irreducible components. It follows that the elliptic curve
	$E_i=Y/\langle \sigma_i\rangle$ has good reduction for $i=b,c$ and its reduction is $\overline{E_i} = \Yb/\langle\sigma_i\rangle$. This implies that $\overline{E_b}$ and $\overline{E_c}$ are elliptic curves and they are both isomorphic to $\Yb_1 \simeq \Yb_2$. Therefore
	\[
	j(\Yb_1) \equiv j(\Yb_2) \equiv j(E_b) \equiv j(E_c) \pmod{\pi}.
	\]
	To compute $j(\Yb_1)$ it remains therefore to  compute the reduction of  $j(E_b)\pmod{\pi}$. 
	
	We have $E_b:\,Ax^4+Bv^2+Cz^4+avz^2+bz^2x^2+cx^2v=0$. After a change of coordinates in $\GL_3(\OO)$,  we can assume $B=C=1$, $a=2+a_1\pi_1$ and $c=b+b_1\pi^2_1$ for a suitable element $\pi_1\in\OO$ of positive valuation. A direct computation of the $j$-invariant yields $j(E_b)\equiv 1728\text{ mod }\pi$.
   \hfill $\qed$ 
   \end{proof}

\begin{lemma}[Theorem~\ref{THM:Main_DegConic}, case (c)]\label{lem:degeneratehyp}
	Let $Y$ be as in Theorem \ref{THM:Main_DegConic} and assume that \(\nu(\Iiii)=0,\nu(\Iiii')>0, \nu(\Ivi)>0\) and $\nu(I)> 0$. Then $Y$ has (maybe bad) hyperelliptic reduction. More specifically, 
	\begin{enumerate}[(i)]
		\item if \(2\nu(I_6) = 3\nu(I_3'') \leq 6\nu(I_3')\), then \(Y\) has good hyperelliptic reduction,
		\item if \(2\nu(I_6) > 3\nu(I_3'')\) and \(2\nu(I_3')\geq \nu(I_3'') = \nu(I_3'^2 - 16I_3I_3'')\), then the decorated graph has type II.3 and the reduction of the curve is Loop,
		\item if \(2\nu(I_6) > 3\nu(I_3'')\) and \(2\nu(I_3') = \nu(I_3'') < \nu(I_3'^2 - 16I_3I_3'')\), then the decorated graph has type III.1 and the reduction of the curve is DNA,
		\item if \(2\nu(I_6) < 3\nu(I_3'')\) and \(3\nu(I_3') \geq \nu(I_6)\), then the decorated graph has type II.2 and the reduction of the curve is DNA,
		\item if \(3\nu(I_3') < \nu(I_3'I_3'') < \nu(I_6)\), then the decorated graph has type IV*.2 and the reduction of the curve is Cat,
		\item if \(3\nu(I_3')< \nu(I_6) < \nu(I_3'I_3'')\), then the decorated graph has type IV.5 and the reduction of the curve is Grl Pwr, and
		\item if \(3\nu(I_3') < \nu(I_3'I_3'') = \nu(I_6)\), then the decorated graph has type III.4 and the reduction of the curve is Candy.
	\end{enumerate}
\end{lemma}

\begin{proof}
	It follows from the conditions on the invariants that \(\nu(A), \nu(B)\) and \(\nu(C)\) are zero and that \(\nu(\Delta_a)\), \(\nu(\Delta_b)\) and \(\nu(\Delta_c)\) are all positive.
	
	Therefore, the special fiber \(\overline{X}_0\) of the model of $X$ defined by Equation \eqref{eq: G} is non-reduced, see Proposition~\ref{prop:degeneratedconicproperties}. We claim that  \(Y\) has (not necessarily good) hyperelliptic reduction.
	
	Note that there exists a choice of square roots such that the equation of $Y$ can be written as
	$$
	(\sqrt{A}x^2+\sqrt{B}y^2+\sqrt{C}z^2)^2+(a-2\sqrt{BC})y^2z^2+(b-2\sqrt{AC})x^2z^2+(c-\sqrt{AB})x^2y^2=0,
	$$
	in such a way that the coefficients of $x^2y^2$, $y^2z^2$ and $z^2x^2$ have all positive valuation. Let $\pi_1\in\mathcal{O}$ be an element with valuation $\operatorname{min}\{\nu(a-2\sqrt{BC}), \nu(b-2\sqrt{AC}), \nu(c-\sqrt{AB}) \}/2$. Assume that the minimum of these valuations is attained by $a-2\sqrt{BC}$, then we can rewrite the equation of $Y$ as follows:
	\begin{equation}
	\label{eq:redhyp}
	Y:\,\begin{cases}
	\pi_1^2t^2=-((a-2\sqrt{BC}) y^2z^2+(b-2\sqrt{AC}) x^2z^2+(c-\sqrt{AB}x^2y^2)^r)\\
	\pi_1t=\sqrt{A}x^2+\sqrt{B}y^2+\sqrt{C}z^2
	\end{cases}.
	\end{equation}
	After making a suitable change of coordinates in $\GL_3(K)$ that sends the conic $\sqrt{A}x^2+\sqrt{B}y^2+\sqrt{C}z^2=0$ to $x_1^2-y_1z_1=0$, and taking $z_1=1$ and $y_1=x_1^2$ in the reduction of the first equation defining $Y$ in Equation \eqref{eq:redhyp}, we get a hyperelliptic equation $t^2=x_1^8+Mx_1^6+Nx_1^4+Mx_1^2+1$ with coefficients: 
	$$M=-4\frac{b\sqrt{B}-c\sqrt{C}}{\sqrt{A}(a-2\sqrt{BC})}$$
	$$N=-2+8\frac{b\sqrt{B}+c\sqrt{C} - 4\sqrt{ABC}}{\sqrt{A}(a-2\sqrt{BC})}.$$ 
	
	Its invariants, as defined in Proposition~\ref{prop:invhyp}, are 
	\begin{equation*}
	(L_1:L_2:L_3)=(2I_3':16I_3I_3'': - 4I_6I_3 ) \in \mathbb{P}_{1,2,3}^2.
	\end{equation*}

	Notice that this equality of projective points is not a coordinate-wise equality, but one in a weighted projective space. 
	
	Let $\pi_2\in\mathcal{O}$ be an element of valuation $\operatorname{min}(\nu(L_1),\nu(L_2)/2, \nu(L_3)/3)$. Then we have
	\begin{equation*}
	(L_1:L_2:L_3)=\left(\frac{2I_3'}{\pi_2}:\frac{16I_3I_3''}{\pi_2^{2}}: \frac{- 4I_6I_3}{\pi_2^{3}} \right) \in \mathbb{P}_{1,2,3}^2.
	\end{equation*}
	This is still not a coordinate-wise equality, but a coordinate-wise valuation equality. 
	It follows that we may assume that $\min(\nu(L_1),\nu(L_2), \nu(L_3))=0$.  Rewriting the normalized invariants  in Theorem \ref{THM:hyperelliptic}  in terms of $I_3, I_3', I_3'', I_6$ and $I$ yields the result.
   \hfill $\qed$ 
   \end{proof}

\begin{lemma}[Theorem~\ref{THM:Main_DegConic}, case (d)]\label{lemma:CaseDd} 
	Let $Y$ be as in Theorem~\ref{THM:Main_DegConic}, in particular $\nu(I_3'')>0$. Assume that
	\(\nu(\Iiii)>0, \nu(\Ivi)=0\) and \(\nu(I)=0\), then the decorated graph has type III.7 and the reduction type of the curve is Cave.
\end{lemma}

\begin{proof}
	From the conditions on the invariants, it follows that $\nu(\Delta_a) = \nu(\Delta_b) = \nu(\Delta_c) = 0$ and exactly one of $\nu(A), \nu(B), \nu(C)$ is positive. We may assume that $\nu(A)>0$. Proposition \ref{prop:degeneratedconicproperties} implies that $\Xb_0$ is reducible and that branch points with the same inertia generator specialize to different irreducible components of $\Xb_0$. Proposition \ref{Prop:norm1} implies that one of the branch points $P_b, P_b'$ and one of the branch points $P_c, P_c'$ specialize to the same irreducible component of $\Xb_0$. Moreover, that proposition implies that none of the branch points specialize to the singular point of $\Xb_0$.  Hence the decorated graph has type III.7 and the reduction type of the curve is Cave.
   \hfill $\qed$ 
   \end{proof}

\begin{lemma}[Theorem~\ref{THM:Main_DegConic}, case (e)]\label{lemma:CaseDe} 
	Let $Y$ be as in Theorem~\ref{THM:Main_DegConic}, in particular $\nu(I_3'')>0$. Assume that \(\nu(\Iiii)>0, \nu(\Iiii')=0\) and \(\nu(I)>0\), then the decorated graph has type IV.4 and the reduction of the curve is Braid.
\end{lemma}

\begin{proof}
	From the conditions on the invariants and Equation \eqref{eq:I}, it follows that $\nu(\Delta_a) = \nu(\Delta_b) = \nu(\Delta_c) = 0$ and  exactly two among $\nu(A), \nu(B), \nu(C)$ are positive. We may assume, without 
	loss of generality, that $\nu(C)=0$.  Now Proposition \ref{prop:degeneratedconicproperties} implies that $\Xb_0$ is reducible and that  branch points with the same inertia generator specialize to different irreducible components  of $\Xb_0$.  Moreover,  Proposition \ref{prop:degeneratedconicproperties}.(iii) implies that $P_c$ specializes to the same point as one of $\{P_b,P_b'\}$ on $\Xb$ and $P_c'$ specializes to the same point as one of $\{P_a,P_a'\}$ up to renaming $P_c$ and $P_c'$. Hence the decorated graph has type IV.4. The two irreducible components of $\Xb_0$ are the central ones. The reduction of the curve is Braid.
   \hfill $\qed$ 
   \end{proof}

\section{Hyperelliptic case} \label{Sec:Hyper}

In this section we give an analogous result to Theorems \ref{THM:Main_NonDegConic} and \ref{THM:Main_DegConic} for the hyperelliptic case, i.e., for curves in $\mathcal{M}_{3,V}^{\text{hyp}}$.
Recall that $(K, \nu)$ is a complete discretely valued field of
characteristic $0$ and residue characteristic $p\geq 0$ different from $2$. Recall
that we replace $K$ by a finite extension, if necessary, without
changing the notation. 

Let $Y/K$ be a genus-$3$ hyperelliptic curve such that
$\Aut_{\Kb}(Y)$ contains a subgroup $V\simeq C_2\times C_2$
such that for every non-trivial element \(\sigma\in V\) the quotient $Y/\langle\sigma\rangle$ has genus~$1$. Then we can write (see \cite[Section 4.3]{IreneThesis} or \cite[Table 3]{LR11}): 
\begin{equation} \label{hyp: Y}
Y:\,y^2=x^8+Mx^6+Nx^4+Mx^2+1,
\end{equation}
and we identify $V$ with the group generated by
$$
\sigma_1(x,y)=(-x,y)\text{ and }\sigma_2(x,y)=(1/x,y/x^4).
$$

We set $\sigma_3 :=\sigma_1\sigma_2$. 

In particular, the genus-3 hyperelliptic curve \(Y\) is a \(V\)-Galois cover of a conic, and we obtain the following diagram. 
\begin{equation*}
\xymatrix{
	& Y\ar@{-}[ld]\ar@{-}[d]\ar@{-}[rd] & \\
	Y/\langle\sigma_1\rangle=E_1\ar@{-}[rd]&Y/\langle\sigma_2\rangle=E_2\ar@{-}[d] & Y/\langle\sigma_3\rangle=E_3\ar@{-}[ld]\\
	&Y/V= X &  
}
\end{equation*}
where
\begin{align*}
Y:\, &y^2 =x^8+Mx^6+Nx^4+Mx^2+1,\\
X:\,&w^2=v^2+(M-4)v+(-2M+N+2)
\end{align*}
with $w=\frac{y}{x^2}$ and $v=(x+\frac{1}{x})^2$. Moreover, we compute the discriminants
\begin{align*}
\Delta(Y)=&2^{4}(-2M+N+2)^2(2M+N+2)^2(M^2-4N+8)^4,\\
\Delta(X)=&M^2-4N+8.
\nonumber\end{align*}

The problem with the hyperelliptic model given in Equation \eqref{hyp: Y} is that it is singular at the infinity point $(0:1:0)$ and it is not easy to keep track of the ramification data.

Instead, we choose to work with the smooth model: 
$$
Y:\,\begin{cases}
t^2=y^4+My^3z+Ny^2z^2+Myz^3+z^4,\\
0=x^2-yz
\end{cases}\subseteq\mathbb{P}^3_{K,(1,1,1,2)},
$$ 
where now the automorphisms are given by 
\begin{equation} \label{eq: hyp_autom}
\sigma_1((x : y : z : t)) = (-x : y : z : t)\text{ and }\sigma_2((x : y : z : t))= (x : z : y : t).
\end{equation}

We get the following equation for the conic $X:=Y/V$:
\begin{equation}\label{conichyp}
\begin{aligned}
Y&\rightarrow X:\,w^2=v^2+(M-4)uv 
+ (-2M + N + 2)u^2\subseteq\mathbb{P}^2_K\\
(x:y:z:t)&\mapsto(u:v:w)=(x^2:(y+z)^2:t),
\end{aligned}
\end{equation}
and the six branch points in $X$ are 
\begin{equation}\label{eq:branchhypV4}
\begin{aligned}
\sigma_1 :&\; P_1=(0:1:1), & P'_1=(0:1:-1),\\
\sigma_2 :&\; P_2 = (1:4:\lambda), & P_2' = (1:4:-\lambda),\\
\sigma_3 :&\; P_3 = (1:0:\mu), & P_3' = (1:0:-\mu),
\end{aligned}
\end{equation}
where $\lambda$ is a root of $T^2 - (2M + N + 2)=0$ and $\mu$ is a root of $T^2 - (-2M + N + 2)=0$.

Note that the discriminant of this conic is still $\Delta(X)=M^2-4N+8$. 

\begin{remark}\label{hypconic}
	If $\nu(\Delta(X))>0$, then the conic always reduces to the product of two different lines $(\frac{M-4}{2}u + v - w)(\frac{M-4}{2}u + v + w)$.
\end{remark}

\subsection{Invariants}\label{SSec:invariants}
Shioda \cite{Shioda67} gives $9$ invariants $J_2, J_3, \ldots, J_{10}$
for genus-$3$ hyperelliptic curves, which we call the Shioda invariants. The Shioda invariants parametrize the $5$-dimensional locus of
hyperelliptic curves $\mathcal{M}_3^{\text{hyp}}$ inside the moduli space of genus-$3$ curves $\mathcal{M}_3$. The stratum $\M_{3, V}^{\text{hyp}}$ is the intersection of $\M_{3,V}$ with $\mathcal{M}_3^{\text{hyp}}$, and by Lemma \ref{lem:Galois}.(3) it is a $2$-dimensional stratum.

In \cite[Lemma 3.14]{LR11}, the authors give (sufficient and necessary) conditions in terms of the Shioda invariants for a curve to belong to the stratum $\mathcal{M}_{3,V}^{\text{hyp}}$. Loc.~cit.~also gives expressions to compute parameters $M,N$ from the Shioda invariants, thus obtaining a model 
$$
Y:\,y^2=x^8+Mx^6+Nx^4+Mx^2+1.
$$

However, using the Shioda invariants restricted to the stratum $\M_{3, V}^{\text{hyp}}$ is not practical to characterize the stable reduction of \(Y\), so we define invariants for the stratum  $\M_{3, V}^{\text{hyp}}$.

\begin{proposition}\label{prop:invhyp} The invariant ring of $\M_{3, V}^{\emph{hyp}}$ is generated by the following invariants of weight $1,2$ and $3$ respectively: 
	$$
	L_1=N+10,\,L_2=M^2-4N+8,\,L_3=(2M+N+2)(2M-N-2).
	$$
\end{proposition}
\begin{proof} First we need to check that they are invariants. Secondly that they generate the ring of invariants.
	
	Isomorphisms between hyperelliptic curves $y^2=f(x,z)$  are given by 
	linear maps $(x,z)\mapsto(a_{11}x+ a_{12}z, a_{21}x+ a_{22}z)$. 
	Since the isomorphisms between curves in the family $y^2=x^8+Mx^6+Nx^4+Mx^2+1$ preserve the automorphism group generated by $(x,z) \mapsto (-x,z)$ and $(x,z)\mapsto (z,x)$, 
	every isomorphism can be written as a composition of 
	$i: (x,z) \mapsto (x, -z)$, $r: (x,z)\mapsto (x+z, x-z)$ and automorphisms of the curve. Notice that $r^2=i^2=(ri)^3=\operatorname{Id}$, and that they generate a finite group $G$ isomorphic to $S_3$.
	
	This implies that all the curves isomorphic to the one with parameters $(M,N)$ are the curves with parameters:
	\begin{equation}\label{eq:MNalternatives}
	\begin{gathered}
	(M,N),\,(-M,N),\\
	((8M-4N+56)/(2M+N+2),(-20M+6N+140)/(2M+N+2)),\\
	(-(8M-4N+56)/(2M+N+2),(-20M+6N+140)/(2M+N+2)),\\
	((-8M-4N+56)/(-2M+N+2),(20M+6N+140)/(-2M+N+2)),\\
	(-(-8M-4N+56)/(-2M+N+2),(20M+6N+140)/(-2M+N+2)).
	\end{gathered}
	\end{equation}
	We proceed as in the proof of Proposition \ref{Prop:Invariants_Generate}. \texttt{Magma} \cite{Magma} produces the invariants $L_1,L_2$ and $L_3$ as generators of the algebra of invariants $K[M,N]^G$ for the fields $\mathbb{F}_3$ and~$\mathbb{Q}$. Again, Molien's Formula \cite[Theorem $3.2.2$]{Kemper} extends the result to any field of characteristic different from $2$.
   \hfill $\qed$ 
   \end{proof}

\begin{remark} The following equalities hold: $\Delta(X)=L_2$ and $\Delta(Y)=2^4L_2^4L_3^2$.
\end{remark}

\begin{remark} \label{rem:normhyp} 
	If the curve \(Y\in\M_{3, V}^{\text{hyp}}\) has a model as in Equation \eqref{hyp: Y} given by the parameters \((M,N)\), then there is always a pair in Equation \eqref{eq:MNalternatives} such that the valuation of both terms is non-negative so again the valuation of the invariants $L_i$ can be assumed to be non-negative. In this situation, the valuation of the three invariants $L_i$ cannot be simultaneously positive. In that case $N\equiv-10\text{ mod }\pi$ because of $\nu(L_1)>0$, $M^2\equiv16\text{ mod }\pi$ because of $\nu(L_3)>0$, but then $L_2\equiv64\text{ mod }\pi$ and $\nu(L_2)$ cannot be positive. 
\end{remark}

\begin{proposition}\label{prop:goodhyp} Let $y^2=x^8+Mx^6+Nx^4+Mx^2+1$ be a hyperelliptic curve in $\mathcal{M}_{3,V}^{\text{hyp}}$ with invariants $L_1,L_2,L_3$ defined as in Proposition \ref{prop:invhyp}. It has potentially good reduction if and only if $\nu(L_1^2/L_2)\geq 0$ and $\nu(L_2^3/L_3^2)=0$.
\end{proposition}
\begin{proof}
	By Remark \ref{rem:normhyp}, we can assume $M,N,L_i\in\mathcal{O}$ and at least one of the $L_i$ having valuation zero. If $\nu(L_1^2/L_2)\geq 0$ and $\nu(L_2^3/L_3^2)=0$ hold then we have $\nu(L_1)\geq\nu(L_2)/2=\nu(L_3)/3\geq0$. This gives us $\nu(L_2)=\nu(L_3)=0$ and $\nu(\Delta(Y))=\nu(2^4L_2^4L_3^2)=0$. Hence the curve has good reduction. 
	
	Conversely, assume the curve has potentially good reduction, then there exists a hyperelliptic curve model of $Y$ having good reduction. Because of Corollary $3.5$ in \cite{LLLR} this model can be taken of the form $y^2=x^8+Mx^6+Nx^4+Mx^2+1$ with $M,N\in\mathcal{O}$ and $\nu(\Delta(Y))=0$. Now, because of Remark \ref{rem:normhyp} we can also assume $\nu(L_1),\,\nu(L_2),\,\nu(L_3)\geq0$. So $\nu(L_1^2/L_2)\geq 0$ and $\nu(L_2^3/L_3^2)=0$.
   \hfill $\qed$ 
   \end{proof}

\subsection{The main theorem and its proof}

We characterize the possible reduction types of a genus-3 hyperelliptic curve $Y$ in $\mathcal{M}_{3,V}^{\text{hyp}}$ in terms of the invariants $L_1$, $L_2$ and $L_3$ defined in Proposition \ref{prop:invhyp}. Because of Proposition \ref{prop:goodhyp}, and once the invariants are normalized as in Remark \ref{rem:normhyp}, $Y$ has potentially good reduction if and only if $\nu(\Delta(Y))=0$. The theorem below describes the different types of bad reduction when $\nu(\Delta(Y))>0$. 

\begin{theorem}\label{THM:hyperelliptic}
	Let $Y$ be a hyperelliptic genus-$3$ curve defined by $Y : \,t^2=y^4+My^3z+Ny^2z^2+Myz^3+z^4, x^2=yz\subseteq\mathbb{P}^3_{1,1,1,2}$ and normalized as in Remark \ref{rem:normhyp}. Let $X$ be the conic $Y/\langle\sigma_1,\sigma_2\rangle$ with $\sigma_1, \sigma_2$ given as in Equation \eqref{eq: hyp_autom}. 
	Then if the valuation of \(\Delta(Y)\) is positive, $Y$ has geometric
	bad reduction and one of the cases in Table~\ref{tab:hyp} occurs.
\end{theorem} 
\renewcommand{\arraystretch}{1.25}
\begin{table}[h]
	\caption{Cases of Theorem~\ref{THM:hyperelliptic}. \label{tab:hyp}}
	\centering
	\begin{adjustbox}{max width = \textwidth}
		\begin{tabular}{ccccccc}
			\hline\hline
			& \(\nu(L_1)\) &      \(\nu(L_2)\)       &      \(\nu(L_3)\)       &     Other conditions     & Decorated graph & Stable curve \\ \hline\hline
			(a)   &         &         \(= 0\)         &         \(>0\)          & \(\nu(L_1^2-4L_2) = 0\)  & II.3             &     Loop     \\ \hline
			(b)   &   \(= 0\)    &         \(= 0\)         &         \(>0\)          & \(\nu(L_1^2-4L_2) > 0\)  & III.1            &     DNA      \\ \hline
			(c)   &              &         \(> 0\)         &         \(=0\)          &                          & II.2             &     DNA      \\ \hline
			(d.i)  &    \textcolor{black}{\multirow{3}{*}{\(=0\)} }          & \multirow{3}{*}{\(>0\)} & \multirow{3}{*}{\(>0\)} & \(\nu(L_2\textcolor{black}{L_1}) < \nu(L_3) \) & IV*.2            &     Cat      \\ \cline{1-1}\cline{5-7}
			(d.ii)  &              &                         &                         & \(\nu(L_2\textcolor{black}{L_1}) > \nu(L_3) \) & IV.5             &    Grl Pwr     \\ \cline{1-1}\cline{5-7}
			(d.iii) &              &                         &                         & \(\nu(L_2\textcolor{black}{L_1}) = \nu(L_3) \) & III.4             &    Candy     \\ \hline\hline
			&              &
		\end{tabular}
	\end{adjustbox}
\end{table}

\begin{proof}
	Recall $\Delta(Y)=2^4L_3^2\Delta(X)^4$ and $\Delta(X) = L_2$, and assume $\nu(\Delta(Y))>0$. 
	
	If the valuation $\nu(L_2)$ is zero, then $\nu(L_3)>0$ holds, i.e., at least one of the valuations $\nu(2M + N + 2)$, $\nu(-2M + N + 2)$ is positive; and the conic \(X\) has good reduction. 
	If exactly one of them is positive, that is, $\nu(N+2)=0$, then the special fiber $\Xb$ of the stably marked model of \(X\) is of type II.3.
	Otherwise, if both $\nu(2M + N + 2)$ and $\nu(-2M + N + 2)$ are positive, and hence \(\nu(N+2)>0\) and \(\nu(L_1) = 0\), then the special fiber $\Xb$ of the stably marked model of \(X\) is of type III.1.
	Statements (a) and (b) then follow from noticing that if $\nu(L_1)>0$, then we write 
	\begin{equation*}\label{N+2}2^3L_3-2^2L_2L_1+L_1^3=-(N+2)(L_3-2^5L_1),\end{equation*}
	hence $\nu(N+2)>0$ if and only if $\nu(L_1^2-4L_2)>0$.
	
	Suppose now $\nu(L_2)>0$. Then by Remark~\ref{hypconic}, the conic $\Xb$ is a product of two lines. If $\nu(L_3)=0$, that is, we are in case (c), then the branch points $P_1$ and $P_1'$ specialize to different lines by  Proposition~\ref{prop:degeneratedconicproperties}.(i), and the same holds for $P_2, P_2'$ and $P_3, P_3'$. So we obtain that $\Xb$ is of type II.2.
	
	Finally we assume $\nu(L_2)>0$, $\nu(L_3)>0$. We write
	\begin{equation} \label{eq:relationsparameters}
	L_2 =(M+4-2\lambda)(M+4+2\lambda)  =(M-4-2\mu)(M-4+2\mu),
	\end{equation}
	where  $\lambda$ is a root of $T^2 - (2M + N + 2)$ with $\nu(\lambda)>0$ and $\mu$ is a root of $T^2 - (-2M + N + 2)$ with $\nu(\mu) = 0$.
	
	Notice that we get $\nu(M+4)>0$ and $\nu(M-4)=0$ by Equation \eqref{eq:relationsparameters}. 
	Then the conic $X$ reduces to a product of two lines and the branch points $P_2$ and $P_2'$ specialize to $(1:4:0)$, the intersection of the lines. 
	
	Consider the coordinate
	$$
	\xi  = \frac{ - (4+\lambda)u + v + w}{(4-\lambda)u- v + w},
	$$
	where $u,v,w$ are the coordinates of $X$ as in Equation \eqref{conichyp}, and which satisfies $\xi(P_1)=\infty$, $\xi(P'_1)=0$, $\xi(P'_2)=1$ and $\xi(P_2) = (M + 4 - 2\lambda)( M + 4 + 2\lambda)^{-1}$. Depending on the value of $\overline{\xi(P_2)}$ we get different possibilities for the decorated graph. We have:
	{\small $$
		\overline{\xi(P_2)} = \begin{cases}
		\infty&\text{ iff }\nu(M+4+2\lambda)>\nu(M+4-2\lambda),\;  \text{(type IV.5)},\\
		0&\text{ iff }\nu(M+4+2\lambda)<\nu(M+4-2\lambda),\;  \text{(type IV.5)},\\
		1&\text{ iff }\nu(M+4+2\lambda)=\nu(M+4-2\lambda) = \nu(M+4)<\nu(\lambda),\;  \text{(type IV*.2)},\\
		\neq 0,\infty, 1&\text{ iff }\nu(M+4+2\lambda)=\nu(M+4-2\lambda) = \nu(\lambda) \leq \nu(M+4),\;  \text{(type III.4)}.
		\end{cases}
		$$}

	The different cases there can be rewritten as:
	\begin{enumerate}
		\item Case IV*.2 if and only if $2\nu(M^2-16)\geq\nu(L_2)$ and $2\nu(M^2-16)<\nu(L_3)$,
		\item Case IV.5  if and only if $2\nu(M^2-16)<\nu(L_2)$, and
		\item Case III.4 if and only if $2\nu(M^2-16)\geq\max\{\nu(L_2),\nu(L_3)\}$;
	\end{enumerate}
	and one can check that these conditions are equivalent to the ones in the statement.
   \hfill $\qed$ 
   \end{proof}

\begin{corollary} \label{cor:hyp_jinvariants}
	\begin{enumerate}[(i)]
		\item The \(j\)-invariant of the genus-1 component of the special fiber in Theorem~\ref{THM:hyperelliptic}.(a) is $j=2^4(12L_2+L_1^2)^3/((4L_2-L_1^2)^{2}L_2)$.
		\item The \(j\)-invariants of the two genus-1 components of the special fiber in Theorem~\ref{THM:hyperelliptic}.(d.iii) are equal to $1728$.
	\end{enumerate}
\end{corollary}

\begin{proof}
	\begin{enumerate}[(i)]
		\item In order to compute the $j$-invariant of the elliptic curve component $E$ of~$\overline{Y}$ in case (a), we assume first that $\nu(-2M+N+2)>0$. Then modulo $\pi$ the equation of $Y$ reduces to	
		$$	
		y^2=(x^2+1)^2(x^4+(M-2)x^2+1),	
		$$	
		so the elliptic curve we are looking for is $(\frac{y}{x^2+1})^2=(x^4+(M-2)x^2+1)$ with 	
		$$	
		j=\dfrac{2^4(12L_2+L_1^2)^3}{(4L_2-L_1^2)^2L_2} \mod \pi.
		$$

		\item In order to compute the $j$-invariants of the two elliptic curves in case (d.iii) we proceed as in Lemma \ref{lem:caseb}.(iii) to get that the two elliptic curves are isomorphic between them and isomorphic to the intermediate elliptic curves $E_1=Y/\langle\sigma_1\rangle$ and $E_2=Y/\langle\sigma_2\rangle$ or $E_3=Y/\langle\sigma_3\rangle$ depending on $\pm 2M+N+2$ having positive valuation. The elliptic curve  $E_1$ is given by the equation:
		$$
		y^2=x^4+Mx^3+Nx^2+Mx+1,
		$$
		where $\nu(N-6)=2\nu(M\pm4)$, 
		and hence with $j$-invariant:
		$$
		j\equiv 1728 \text{ mod }\pi.
		$$

	\end{enumerate}
   \hfill $\qed$ 
   \end{proof}
\section*{Appendix A. Admissible covers}\label{app:adm}
\addcontentsline{toc}{section}{Appendix}

In this appendix, we provide more details on the proofs of Lemma \ref{lem:adm} and Theorem \ref{FromGraphtoRedType}. 

\bigskip
\noindent\textit{\textbf{Proof of Lemma \ref{lem:adm}.}}
    Let $(\Xb, \overline{D})$ be a decorated graph. Recall, from Lemma \ref{lem:graph}, that there is a unique way of labeling the singular points of $(\Xb, \overline{D})$ so that the product of the elements on the labels of each irreducible component is the trivial element. Let $\Xb_1,\Xb_2,\dots, \Xb_r$ be the irreducible components of $\Xb$.   
    To prove the lemma, we need to describe:
    \begin{enumerate}[(i)]
        \item for each irreducible component $\Xb_j$, the restriction $\fb|_{\Xb_j}: \Yb_j \rightarrow \Xb_j$, which is a (not necessarily connected) $V$-Galois cover;
        \item the intersection of the different irreducible components of $\Yb$ over each of the singular points of $\Xb$.
    \end{enumerate}
    Note that the singular points of $\Yb$ map to the singular points of $\Xb$. 

For each irreducible component $\Xb_j$, we write $V_j$ for the subgroup of $V$ generated by the $\sigma_i$'s, where $i$ runs over the labels of the (singular and marked) points of $\Xb_j$. The number of irreducible components of $\Yb_j$ is equal to the index $[V : V_j]$. Each of the irreducible components of $\Yb_j$ is a $V_j$-Galois cover of $\Xb_j$. This finishes the description of (i).

We now address (ii). Let $\tau$ be a singular point of $\Xb$, given by the intersection of the two components $\Xb_{j_1}$ and $\Xb_{j_2}$. The preimage $\fb^{-1}(\tau)$ consists of $4$ points if the label of $\tau$ is $0$ and $2$ points otherwise. In each of these points, one of the irreducible components of $\Yb_{j_1}$ intersects with one of the irreducible components of $\Yb_{j_2}$. The only case when it is not straightforward to determine these intersections is when the cardinality of $\fb^{-1}(\tau)$ is $4$ and neither $V_{j_1}$ nor $V_{j_2}$ equals $V$. In this case, let $\fb^{-1}(\tau)=\{\tau_0,\dots,\tau_3\}$, where $\tau_0$ is chosen arbitrarily and $\tau_i=\sigma_i(\tau_0)$ for $i=1,2,3$. For each $i=0,\dots,3$, $\tau_i$ and $\sigma_\ell(\tau_i)$ lie on the same irreducible component of $\Yb_{j_1}$ (resp. $\Yb_{j_2}$) if and only if $\sigma_\ell \in V_1$ (resp. $V_2$). 
We conclude that every irreducible component of $\Yb_{j_1}$ intersects every irreducible component of $\Yb_{j_2}$ if and only if $V_{j_1}\neq V_{j_2}.$ Example \ref {exa:correct} describes one of the cases where this happens. 
If $V_{j_1}=V_{j_2}$ an irreducible component of $\Yb_{j_1}$ intersects one of the irreducible components of $\Yb_{j_2}$ in two points and does not intersect the other one. Example \ref{exa:wrong} describes one of the cases where this happens.
\hfill $\qed$
\begin{example}\label{exa:wrong}
    Consider Case IV.5.

    	Let  $(\Xb, \overline{D})$ be a decorated graph of type IV.5,
	Figure~\ref{fig:TypeIV.5}.  
	
	\begin{figure}[h!] \sidecaption
		\begin{tikzpicture}[scale=.6]
		\draw (0,0) coordinate (O) -- +(160:3.5);
    \draw (O) -- +(-20:0.5);
    \draw (O) -- +(20:3.5);
    \draw (O) -- +(-160:0.5);
    \draw (O)++(160:3) coordinate (A) -- +(110:3cm);
    \draw (A) -- +(-70:0.5cm);
    \draw (O)++(20:3) coordinate (B) -- +(70:3cm);
    \draw (B) -- +(-110:0.5cm);
		\foreach \x/\y in {1/\two,2/\one}
    \draw (A)++(110:\x)++(20:0.1) -- ++(-160:0.2) node[left=0.5mm] {\y };
    \draw (O)++(160:1.5)++(70:0.1) -- ++(-110:0.2) node[below=0.5mm] {\three };
    \draw (O)++(20:1.5)++(110:0.1) -- ++(-70:0.2) node[below=0.5mm] {\three };
    \foreach \x/\y in {1/\two,2/\one}
    \draw (B)++(70:\x)++(160:0.1) -- ++(-20:0.2) node[right=0.5mm] {\y };
    \draw (-195:2.8) node[left=0.5mm] {\three};
    \draw (-169:-3.5) node[left=0.5mm] {\three};
    \draw (0:0) node[below=0.5mm] {0};
		\end{tikzpicture}
		\caption{$\Xb$ of Type IV.5. \label{fig:TypeIV.5}}
	\end{figure}

 We use the notation from the proof of Lemma \ref{lem:adm}. We number the irreducible components of $\Xb$ as
	$\Xb_1,\Xb_2,\Xb_3,\Xb_4$ (from left to right). In Figure \ref{fig:TypeIV.5}, we have already labelled the singular points, as explained in the proof of Lemma \ref{lem:adm}.

 There are two irreducible components over both $\Xb_2$ and $\Xb_3$, and $V_2=V_3= \langle \sigma_3\rangle$. Let $\tau$ be the intersection point of $\Xb_2$ and $\Xb_3$. The fiber $\fb^{-1}(\tau)$ contains 4 points, which we label $\tau_0,\dots,\tau_3$ as in the proof of Lemma \ref{lem:adm}. Now, since $V_2$ and $V_3$ are both generated by $\sigma_3$, we have that $\tau_0$ and $\tau_3$ (resp. $\tau_1$ and $\tau_2$) lie on the same irreducible component of $\Yb_2$ and $\Yb_3$. This is illustrated in Figure \ref{fig:TypeIV.5_Y}. The irreducible components above $\Xb_1$ and $\Xb_4$ have genus $0$ and intersect the rest of the graph in two points only. Contracting these yields the type Grl Pwr in Figure \ref{figs:admissible-covers}. 
 
 	\begin{figure}[h!]\sidecaption  
  	\begin{tikzpicture}[scale=.6]
 			\node at (0,.5) (O) {};
 			\node at (0,2.5) (P) {};
 			\draw (O)++(160:3) coordinate (A) -- +(110:2cm);
 			\draw (A) -- +(-70:2cm);
 			\draw (O)++(20:3) coordinate (B) -- +(70:2cm);
 			\draw (B) -- +(-110:2cm);
 			\node at ($(A)+(110:2.5)$) (A2) {};
 			\node at ($(B)+(70:2.5)$) (B2) {};
 			\draw ($(A)+(-90:1)$) to[quick curve through={($(A)+(-60:1)$) (O) (0,-.5)}]
 			(-.5,-.7);
 			\draw ($(B)+(-90:1)$) to[quick curve through={($(B)+(240:1)$) (O) (0,-.5)}]
 			(.5,-.7);
 			\draw ($(A2)+(-90:1)$) to[quick curve through={($(A2)+(-60:1)$) (P) (0,1.5)}]
 			(-.5,1.2);
 			\draw ($(B2)+(-90:1)$) to[quick curve through={($(B2)+(240:1)$) (P) (0,1.5)}]
 			(.5,1.2);
 		\end{tikzpicture}
		\caption{$\Yb$ corresponding to Type IV.5. \label{fig:TypeIV.5_Y}}
	\end{figure}

A similar argument shows that Case IV*.2 yields the type Cat.
\end{example}

\begin{example}\label{exa:correct}
    We compute here the stable type in Case IV.4, showing how it differs from Case IV.5.
    
    We use the notation from the proof of Lemma \ref{lem:adm}. We number the irreducible components of $\Xb$ as
	$\Xb_1,\Xb_2,\Xb_3,\Xb_4$ (from left to right). In Figure \ref{fig:TypeIV.4}, we have already labeled the singular points as explained in the proof of Lemma \ref{lem:adm}.

 	\begin{figure}[h!] \sidecaption
		\begin{tikzpicture}[scale=.6]
		\draw (0,0) coordinate (O) -- +(160:3.5);
    \draw (O) -- +(-20:0.5);
    \draw (O) -- +(20:3.5);
    \draw (O) -- +(-160:0.5);
    \draw (O)++(160:3) coordinate (A) -- +(110:3cm);
    \draw (A) -- +(-70:0.5cm);
    \draw (O)++(20:3) coordinate (B) -- +(70:3cm);
    \draw (B) -- +(-110:0.5cm);
		\foreach \x/\y in {1/\two,2/\one}
    \draw (A)++(110:\x)++(20:0.1) -- ++(-160:0.2) node[left=0.5mm] {\y };
    \draw (O)++(160:1.5)++(70:0.1) -- ++(-110:0.2) node[below=0.5mm] {\three };
    \draw (O)++(20:1.5)++(110:0.1) -- ++(-70:0.2) node[below=0.5mm] {\one };
    \foreach \x/\y in {1/\two,2/\three}
    \draw (B)++(70:\x)++(160:0.1) -- ++(-20:0.2) node[right=0.5mm] {\y };
    \draw (-195:2.8) node[left=0.5mm] {\three};
    \draw (-169:-3.5) node[left=0.5mm] {\one};
    \draw (0:0) node[below=0.5mm] {0};
		\end{tikzpicture}
		\caption{$\Xb$ of Type IV.4. \label{fig:TypeIV.4}}
	\end{figure}

 There are two irreducible components over both $\Xb_2$ and $\Xb_3$, and $V_2= \langle \sigma_3\rangle$ and $V_3=\langle \sigma_1\rangle$. Let $\tau$ be the intersection point of $\Xb_2$ and $\Xb_3$. The fiber $\fb^{-1}(\tau)$ contains 4 points, which we label $\tau_0,\dots,\tau_3$ as in the proof of Lemma \ref{lem:adm}. 
 We write $\Yb_{j,0}$ for the irreducible component of $\Yb_j$ containing the point $\tau_0$ (for $j=2,3$). The component $\Yb_{2,0}$ also contains the point $\tau_3$, whereas the component $\Yb_{3,0}$ also contains the point $\tau_1$. Therefore, each irreducible component of $\Yb_2$ intersects each of the irreducible components of $\Yb_3$, Figure \ref{fig:TypeIV.4_Y}.  The irreducible components above $\Xb_1$ and $\Xb_4$ have genus $0$ and intersect the rest of the graph in two points only. Contracting these yields the type Braid in Figure \ref{figs:admissible-covers}. 
 \begin{figure}[h!]
     \sidecaption
     \begin{tikzpicture}[scale=.6]
		\draw (0,0) coordinate (O) -- +(160:3.5);
		\draw (O) -- +(-20:0.5);
		\draw (O) -- +(20:3.5);
		\draw (O) -- +(-160:0.5);
		\draw (0,.75) coordinate (P) -- + (160:4);
		\draw (P) -- +(-20:1.5);
		\draw (P) -- +(20:4);
		\draw (P) -- +(-160:1.5);
		\draw (O)++(160:3) coordinate (A) -- +(110:2.5cm);
		\draw (A) -- +(-70:1.5cm);
		\draw (O)++(20:3) coordinate (B) -- +(70:2.5cm);
		\draw (B) -- +(-110:1.5cm);
	\end{tikzpicture}
     \caption{$\Yb$ corresponding to Type IV.4. \label{fig:TypeIV.4_Y}}
	\end{figure}

A similar argument applies to the cases III.1, III.5, IV.1, IV.2 and IV*.1.
\end{example}

\newpage
\section*{Appendix B. Decorated Graphs}\label{sec:adm}
\addcontentsline{toc}{section}{Appendix}

\begin{table}[htp]
	\caption{Correspondence between the decorated graphs in Figures~(\ref{figs:dec-graphs-I})--(\ref{figs:dec-graphs-IV}) and the stable curves in Figure~\ref{figs:admissible-covers}.}
 \label{tab:correspondence}
	\centering
	\begin{tabular}{cc|cc}
		\hline
		Stable curve & Decorated graph  & Stable curve & Decorated graph \\ \hline
		    Good     &        I         &  Winky Cat   &      III.6      \\
		   Candy     &      II.1        &     Cave     &      III.7      \\
		    DNA      &      II.2       &   Grl Pwr  &      IV.1       \\
		    Loop     &      II.3       &    Garden    &      IV.2       \\
		    Lop      &      II.4     &     Cat      &      IV.3       \\
		    DNA      &      III.1     &    Braid     &      IV.4       \\
		   Looop     &      III.2    &    Grl Pwr     &      IV.5       \\
		    Loop     &      III.3   &    Braid     &      IV*.1      \\
		   Candy     &      III.4     &     Cat      &      IV*.2      \\
		    Tree     &      III.5    &    Looop     &      IV*.3      \\ \hline\\
	\end{tabular}
\end{table}

\renewcommand\thefigure{\Roman{figure}}
\renewcommand{\thesubfigure}{\thefigure.\arabic{subfigure}}
\setcounter{figure}{0}
\begin{figure}[h] \sidecaption 
	\centering
	\begin{tikzpicture}[scale=.7] 
	\draw (0,0) -- (7,0);
	\foreach \x in {1,...,6}
	\draw (\x,-0.1) -- (\x,0.1);
	\end{tikzpicture}
	\caption{Stably marked curve with 6 marked points and one component.\label{figs:dec-graphs-I}}
\end{figure}
\begin{figure}[h] 
	\centering
	
	\subfloat[]{
		\begin{tikzpicture}[cm={cos(45),sin(45),-sin(45),cos(45),(0,0)}, scale=.6]
		\draw (0.5,0) -- (5,0);
		\draw (1,-0.5) -- (1,4);
		\foreach \x/\y in {1/\two,2/\one,3/\one}
		\draw (0.9, \x) -- (1.1, \x) node[left=3mm, below=0.5mm] {\y };
		\foreach \x/\y in {2/\two,3/\three,4/\three}
		\draw (\x,-0.1) -- (\x,0.1) node[below=3mm, right=1mm] {\y };
		\end{tikzpicture}
	}
	\subfloat[]{
		\begin{tikzpicture}[cm={cos(45),sin(45),-sin(45),cos(45),(0,0)}, scale=.6]
		\draw (0.5,0) -- (5,0);
		\draw (1,-0.5) -- (1,4);
		\foreach \x/\y in {1/\three,2/\two,3/\one}
		\draw (0.9, \x) -- (1.1, \x) node[left=3mm, below=0.5mm] {\y };
		\foreach \x/\y in {2/\one,3/\two,4/\three}
		\draw (\x,-0.1) -- (\x,0.1) node[below=3mm, right=1mm] {\y };
		\end{tikzpicture}
	}
	
	\subfloat[]{
		\begin{tikzpicture}[scale=.6]
		\draw (0.5,0) -- (6,0);
		\draw (1,-0.5) -- (1,3);
		\foreach \x/\y in {1/\one,2/\one}
		\draw (0.9, \x) -- (1.1, \x) node[left=3mm] {\y };
		\foreach \x/\y in {2/\two,3/\two,4/\three,5/\three}
		\draw (\x,-0.1) -- (\x,0.1) node[above=-8mm] {\y };
		\end{tikzpicture}
	}
	\subfloat[]{
		\begin{tikzpicture}[scale=.6]
		\draw (0.5,0) -- (6,0);
		\draw (1,-0.5) -- (1,3);
		\foreach \x/\y in {1/\one,2/\two}
		\draw (0.9, \x) -- (1.1, \x) node[left=3mm] {\y };
		\foreach \x/\y in {2/\one,3/\two,4/\three,5/\three}
		\draw (\x,-0.1) -- (\x,0.1) node[above=-8mm] {\y };
		\end{tikzpicture}
	}
	\caption{Stably marked curves with 6 marked points and two components.}
\end{figure}
\begin{figure}[h] 
	\centering
	\subfloat[]{
		\begin{tikzpicture}[scale=.6]
		\draw (-0.5,0) -- (3.5,0);
		\draw (0,0) -- +(110:3cm);
		\draw (0,0) -- +(-70:0.5cm);
		\draw (3,0) -- +(70:3cm);
		\draw (3,0) -- +(-110:0.5cm);
		\foreach \x/\y in {1/\one,2/\one}
		\draw (110:\x)++(20:0.1) -- ++(-160:0.2) node[left=0.5mm] {\y };
		\foreach \x/\y in {1/\two,2/\two}
		\draw (\x,-0.1) -- (\x,0.1) node[below=3mm] {\y };
		\foreach \x/\y in {1/\three,2/\three}
		\draw (3,0)++(70:\x)++(160:0.1) -- ++(-20:0.2) node[right=0.5mm] {\y };
		\end{tikzpicture}
	}
	\subfloat[]{
		\begin{tikzpicture}[scale=.6]
		\draw (-0.5,0) -- (3.5,0);
		\draw (0,0) -- +(110:3cm);
		\draw (0,0) -- +(-70:0.5cm);
		\draw (3,0) -- +(70:3cm);
		\draw (3,0) -- +(-110:0.5cm);
		\foreach \x/\y in {1/\one,2/\one}
		\draw (110:\x)++(20:0.1) -- ++(-160:0.2) node[left=0.5mm] {\y };
		\foreach \x/\y in {1/\two,2/\three}
		\draw (\x,-0.1) -- (\x,0.1) node[below=3mm] {\y };
		\foreach \x/\y in {1/\two,2/\three}
		\draw (3,0)++(70:\x)++(160:0.1) -- ++(-20:0.2) node[right=0.5mm] {\y };
		\end{tikzpicture}
	}
	\subfloat[]{
		\begin{tikzpicture}[scale=.6]
		\draw (-0.5,0) -- (3.5,0);
		\draw (0,0) -- +(110:3cm);
		\draw (0,0) -- +(-70:0.5cm);
		\draw (3,0) -- +(70:3cm);
		\draw (3,0) -- +(-110:0.5cm);
		\foreach \x/\y in {1/\two,2/\one}
		\draw (110:\x)++(20:0.1) -- ++(-160:0.2) node[left=0.5mm] {\y };
		\foreach \x/\y in {1/\two,2/\three}
		\draw (\x,-0.1) -- (\x,0.1) node[below=3mm] {\y };
		\foreach \x/\y in {1/\three,2/\one}
		\draw (3,0)++(70:\x)++(160:0.1) -- ++(-20:0.2) node[right=0.5mm] {\y };
		\end{tikzpicture}
	}
	\subfloat[]{
		\begin{tikzpicture}[scale=.6]
		\draw (-0.5,0) -- (3.5,0);
		\draw (0,0) -- +(110:3cm);
		\draw (0,0) -- +(-70:0.5cm);
		\draw (3,0) -- +(70:3cm);
		\draw (3,0) -- +(-110:0.5cm);
		\foreach \x/\y in {1/\two,2/\one}
		\draw (110:\x)++(20:0.1) -- ++(-160:0.2) node[left=0.5mm] {\y };
		\foreach \x/\y in {1/\three,2/\three}
		\draw (\x,-0.1) -- (\x,0.1) node[below=3mm] {\y };
		\foreach \x/\y in {1/\two,2/\one}
		\draw (3,0)++(70:\x)++(160:0.1) -- ++(-20:0.2) node[right=0.5mm] {\y };
		\end{tikzpicture}
	}
	
	\subfloat[]{
		\begin{tikzpicture}[scale=.6]
		\draw (-0.5,0) -- (3.5,0);
		\draw (0,0) -- +(110:4cm);
		\draw (0,0) -- +(-70:0.5cm);
		\draw (3,0) -- +(70:4cm);
		\draw (3,0) -- +(-110:0.5cm);
		\foreach \x/\y in {1/\two,2/\one,3/\one}
		\draw (110:\x)++(20:0.1) -- ++(-160:0.2) node[left=0.5mm] {\y };
		\draw (1.5,-0.1) -- (1.5,0.1) node[below=3mm] {\two };
		\foreach \x/\y in {1/\three,2/\three}
		\draw (3,0)++(70:\x)++(160:0.1) -- ++(-20:0.2) node[right=0.5mm] {\y };
		\end{tikzpicture}
	}
	\subfloat[]{
		\begin{tikzpicture}[scale=.6]
		\draw (-0.5,0) -- (3.5,0);
		\draw (0,0) -- +(110:4cm);
		\draw (0,0) -- +(-70:0.5cm);
		\draw (3,0) -- +(70:4cm);
		\draw (3,0) -- +(-110:0.5cm);
		\foreach \x/\y in {1/\two,2/\one,3/\one}
		\draw (110:\x)++(20:0.1) -- ++(-160:0.2) node[left=0.5mm] {\y };
		\draw (1.5,-0.1) -- (1.5,0.1) node[below=3mm] {\three };
		\foreach \x/\y in {1/\two,2/\three}
		\draw (3,0)++(70:\x)++(160:0.1) -- ++(-20:0.2) node[right=0.5mm] {\y };
		\end{tikzpicture}
	}
	\subfloat[]{
		\begin{tikzpicture}[scale=.6]
		\draw (-0.5,0) -- (3.5,0);
		\draw (0,0) -- +(110:4cm);
		\draw (0,0) -- +(-70:0.5cm);
		\draw (3,0) -- +(70:4cm);
		\draw (3,0) -- +(-110:0.5cm);
		\foreach \x/\y in {1/\three,2/\two,3/\one}
		\draw (110:\x)++(20:0.1) -- ++(-160:0.2) node[left=0.5mm] {\y };
		\draw (1.5,-0.1) -- (1.5,0.1) node[below=3mm] {\three };
		\foreach \x/\y in {1/\two,2/\one}
		\draw (3,0)++(70:\x)++(160:0.1) -- ++(-20:0.2) node[right=0.5mm] {\y };
		\end{tikzpicture}
	}
	\caption{Stably marked curves with 6 marked points and three components.}
\end{figure}
\begin{figure}[h] 
	\centering
	\subfloat[]{
		\begin{tikzpicture}[scale=.6]
		\draw (0,0) coordinate (O) -- +(160:3.5);
		\draw (O) -- +(-20:0.5);
		\draw (O) -- +(20:3.5);
		\draw (O) -- +(-160:0.5);
		\draw (O)++(160:3) coordinate (A) -- +(110:3cm);
		\draw (A) -- +(-70:0.5cm);
		\draw (O)++(20:3) coordinate (B) -- +(70:3cm);
		\draw (B) -- +(-110:0.5cm);
		\foreach \x/\y in {1/\one,2/\one}
		\draw (A)++(110:\x)++(20:0.1) -- ++(-160:0.2) node[left=0.5mm] {\y };
		\draw (O)++(160:1.5)++(70:0.1) -- ++(-110:0.2) node[below=0.5mm] {\two };
		\draw (O)++(20:1.5)++(110:0.1) -- ++(-70:0.2) node[below=0.5mm] {\two };
		\foreach \x/\y in {1/\three,2/\three}
		\draw (B)++(70:\x)++(160:0.1) -- ++(-20:0.2) node[right=0.5mm] {\y };
		\end{tikzpicture}
	}
	\subfloat[]{
		\begin{tikzpicture}[scale=.6]
		\draw (0,0) coordinate (O) -- +(160:3.5);
		\draw (O) -- +(-20:0.5);
		\draw (O) -- +(20:3.5);
		\draw (O) -- +(-160:0.5);
		\draw (O)++(160:3) coordinate (A) -- +(110:3cm);
		\draw (A) -- +(-70:0.5cm);
		\draw (O)++(20:3) coordinate (B) -- +(70:3cm);
		\draw (B) -- +(-110:0.5cm);
		\foreach \x/\y in {1/\one,2/\one}
		\draw (A)++(110:\x)++(20:0.1) -- ++(-160:0.2) node[left=0.5mm] {\y };
		\draw (O)++(160:1.5)++(70:0.1) -- ++(-110:0.2) node[below=0.5mm] {\two };
		\draw (O)++(20:1.5)++(110:0.1) -- ++(-70:0.2) node[below=0.5mm] {\three };
		\foreach \x/\y in {1/\two,2/\three}
		\draw (B)++(70:\x)++(160:0.1) -- ++(-20:0.2) node[right=0.5mm] {\y };
		\end{tikzpicture}
	}
	
	\subfloat[]{
		\begin{tikzpicture}[scale=.6]
		\draw (0,0) coordinate (O) -- +(160:3.5);
		\draw (O) -- +(-20:0.5);
		\draw (O) -- +(20:3.5);
		\draw (O) -- +(-160:0.5);
		\draw (O)++(160:3) coordinate (A) -- +(110:3cm);
		\draw (A) -- +(-70:0.5cm);
		\draw (O)++(20:3) coordinate (B) -- +(70:3cm);
		\draw (B) -- +(-110:0.5cm);
		\foreach \x/\y in {1/\two,2/\one}
		\draw (A)++(110:\x)++(20:0.1) -- ++(-160:0.2) node[left=0.5mm] {\y };
		\draw (O)++(160:1.5)++(70:0.1) -- ++(-110:0.2) node[below=0.5mm] {\one };
		\draw (O)++(20:1.5)++(110:0.1) -- ++(-70:0.2) node[below=0.5mm] {\three };
		\foreach \x/\y in {1/\two,2/\three}
		\draw (B)++(70:\x)++(160:0.1) -- ++(-20:0.2) node[right=0.5mm] {\y };
		\end{tikzpicture}
	}
	\subfloat[]{
		\begin{tikzpicture}[scale=.6]
		\draw (0,0) coordinate (O) -- +(160:3.5);
		\draw (O) -- +(-20:0.5);
		\draw (O) -- +(20:3.5);
		\draw (O) -- +(-160:0.5);
		\draw (O)++(160:3) coordinate (A) -- +(110:3cm);
		\draw (A) -- +(-70:0.5cm);
		\draw (O)++(20:3) coordinate (B) -- +(70:3cm);
		\draw (B) -- +(-110:0.5cm);
		\foreach \x/\y in {1/\two,2/\one}
		\draw (A)++(110:\x)++(20:0.1) -- ++(-160:0.2) node[left=0.5mm] {\y };
		\draw (O)++(160:1.5)++(70:0.1) -- ++(-110:0.2) node[below=0.5mm] {\three };
		\draw (O)++(20:1.5)++(110:0.1) -- ++(-70:0.2) node[below=0.5mm] {\one };
		\foreach \x/\y in {1/\two,2/\three}
		\draw (B)++(70:\x)++(160:0.1) -- ++(-20:0.2) node[right=0.5mm] {\y };
		\end{tikzpicture}
	}
	\subfloat[]{
		\begin{tikzpicture}[scale=.6]
		\draw (0,0) coordinate (O) -- +(160:3.5);
		\draw (O) -- +(-20:0.5);
		\draw (O) -- +(20:3.5);
		\draw (O) -- +(-160:0.5);
		\draw (O)++(160:3) coordinate (A) -- +(110:3cm);
		\draw (A) -- +(-70:0.5cm);
		\draw (O)++(20:3) coordinate (B) -- +(70:3cm);
		\draw (B) -- +(-110:0.5cm);
		\foreach \x/\y in {1/\two,2/\one}
		\draw (A)++(110:\x)++(20:0.1) -- ++(-160:0.2) node[left=0.5mm] {\y };
		\draw (O)++(160:1.5)++(70:0.1) -- ++(-110:0.2) node[below=0.5mm] {\three };
		\draw (O)++(20:1.5)++(110:0.1) -- ++(-70:0.2) node[below=0.5mm] {\three };
		\foreach \x/\y in {1/\two,2/\one}
		\draw (B)++(70:\x)++(160:0.1) -- ++(-20:0.2) node[right=0.5mm] {\y };
		\end{tikzpicture}
	}
	\caption{Stably marked curves with 6 marked points and four components, all containing at least one marked point.}
\end{figure}
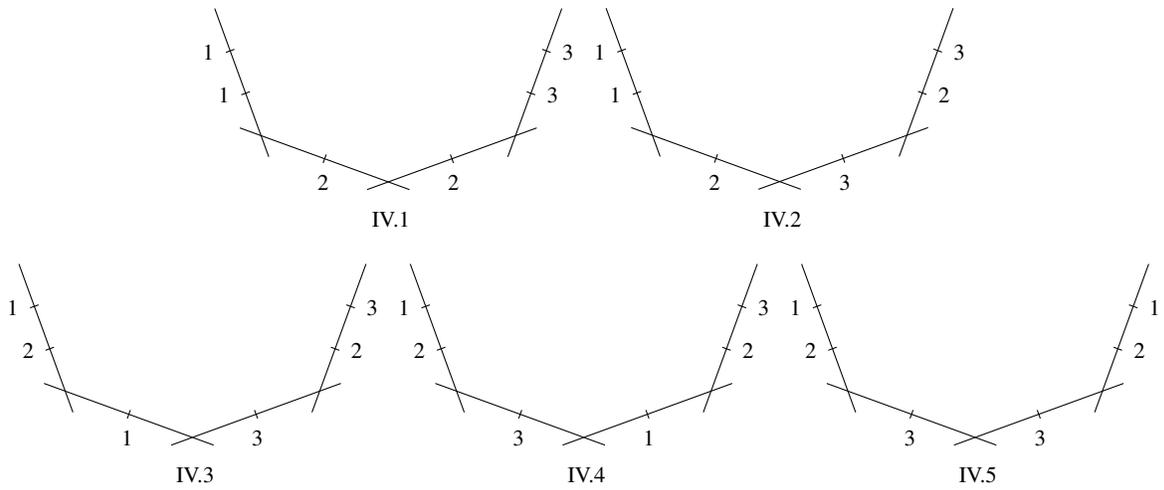

\renewcommand{\thefigure}{IV*}
\begin{figure}[h] 
	\centering
	\subfloat[]{
		\begin{tikzpicture}[scale=.6]
		\draw (0.5,0) -- (5.5,0);
		\draw (1,-0.5) -- (1,3);
		\draw (3,-0.5) -- (3,3);
		\draw (5,-0.5) -- (5,3);
		\foreach \x/\y in {1/\one,2/\one}
		\draw (0.9, \x) -- (1.1, \x) node[left=3mm] {\y };
		\foreach \x/\y in {1/\two,2/\two}
		\draw (2.9, \x) -- (3.1, \x) node[left=3mm] {\y };
		\foreach \x/\y in {1/\three,2/\three}
		\draw (4.9, \x) -- (5.1, \x) node[left=3mm] {\y };
		\end{tikzpicture}
	}
	\subfloat[]{
		\begin{tikzpicture}[scale=.6]
		\draw (0.5,0) -- (5.5,0);
		\draw (1,-0.5) -- (1,3);
		\draw (3,-0.5) -- (3,3);
		\draw (5,-0.5) -- (5,3);
		\foreach \x/\y in {1/\one,2/\one}
		\draw (0.9, \x) -- (1.1, \x) node[left=3mm] {\y };
		\foreach \x/\y in {1/\three,2/\two}
		\draw (2.9, \x) -- (3.1, \x) node[left=3mm] {\y };
		\foreach \x/\y in {1/\three,2/\two}
		\draw (4.9, \x) -- (5.1, \x) node[left=3mm] {\y };
		\end{tikzpicture}
	}
	\subfloat[]{
		\begin{tikzpicture}[scale=.6]
		\draw (0.5,0) -- (5.5,0);
		\draw (1,-0.5) -- (1,3);
		\draw (3,-0.5) -- (3,3);
		\draw (5,-0.5) -- (5,3);
		\foreach \x/\y in {1/\two,2/\one}
		\draw (0.9, \x) -- (1.1, \x) node[left=3mm] {\y };
		\foreach \x/\y in {1/\three,2/\two}
		\draw (2.9, \x) -- (3.1, \x) node[left=3mm] {\y };
		\foreach \x/\y in {1/\one,2/\three}
		\draw (4.9, \x) -- (5.1, \x) node[left=3mm] {\y };
		\end{tikzpicture}
	}
	\caption{Stably marked curves with 6 marked points and four components, one of which doesn't contain any marked point.\label{figs:dec-graphs-IV}}
\end{figure}
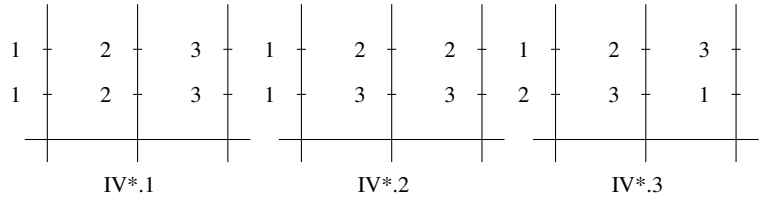

\renewcommand{\thefigure}{B}
\renewcommand{\thesubfigure}{}
\setcounter{figure}{2}
\begin{figure}[!h]
	\centering
	\subfloat[Lop]{
		\begin{tikzpicture}[scale=.6]
		\draw[very thick, dashed] (0,0) to[quick curve through={(2,0.5) (3,1.5) (2,3) (1,1.5) (2,0.5)}]
		(4,0);
		\end{tikzpicture}
	}
	\subfloat[Loop]{
		\begin{tikzpicture}[scale=.6]
		\draw[very thick] (0,0) to[quick curve through={(2,0.5) (3,1.5) (2,3) (1,1.5) (2,0.5) (4,0) (6,0.5) (7,1.5) (6,3) (5,1.5) (6,0.5)}]
		(8,0);
		\end{tikzpicture}
	}
	\subfloat[Looop]{
		\begin{tikzpicture}[scale=.6]
		\draw (0,0) to[quick curve through={(2,0.5) (3,1.5) (2,3) (1,1.5) (2,0.5) (4,0) (6,0.5) (7,1.5) (6,3) (5,1.5) (6,0.5) (8,0) (10,0.5) (11,1.5) (10,3) (9,1.5) (10,0.5)}]
		(12,0);
		\end{tikzpicture}
	}
	
	\subfloat[DNA]{
		\begin{tikzpicture}[scale=.4]
		\draw (-0.5,0.5) to[quick curve through={(1,1) (2,2) (1,3) (0,4) (1,5) (2,6) (1,7)}]
		(-0.5,7.5);
		\draw (2.5,0.5) to[quick curve through={(1,1) (0,2) (1,3) (2,4) (1,5) (0,6) (1,7)}]
		(2.5,7.5);
		\clip (-2.5,0) rectangle (4.5,8);
		\end{tikzpicture}
	}
	\subfloat[Candy]{
		\begin{tikzpicture}[scale=.5]
		\draw[very thick] (-0.5,0.5) to[quick curve through={(2,3)}]
		(-0.5,6.5);
		\draw[very thick] (2.5,0.5) to[quick curve through={(0,3)}]
		(2.5,6.5);
		\clip (-2.5,0) rectangle (4.5,8);
		\end{tikzpicture}
	}
	\subfloat[Cave]{
		\begin{tikzpicture}[scale = .5]
		\draw (0,0) to[quick curve through={(2,3)}]
		(6,4);
		\draw (-1,1) to[quick curve through={(2,0.5)}]
		(6,0);
		\draw (3,-1) to[out angle = 90, in angle = 90, curve through={(4,5)}]
		(5,-1);
		\end{tikzpicture}
	}
	\subfloat[Winky cat]{
		\begin{tikzpicture}[scale = .5]
		\draw (-0.5,0.5) to[quick curve through={(2,3) (-1.5,6.5) (-1.5, 5.5) (-1.5,6.5)}]
		(-3.5,6.5);
		\draw[very thick] (2.5,0.5) to[curve through={(0,3)}]
		(3.5,6.5);
		\end{tikzpicture}
	}
	
	\subfloat[Tree]{
		\begin{tikzpicture}[scale=.5]
		\draw (-0.5,-0.5) to[quick curve through={(2,3)}]
		(-0.5,6.5);
		\draw (2.5,-0.5) to[quick curve through={(0,3)}]
		(2.5,6.5);
		\draw[very thick] (-2,0) to (4,0);
		\end{tikzpicture}
	}
	\subfloat[Grl pwr]{
		\begin{tikzpicture}[scale=.5]
		\draw (0,2) to[curve through={(3,1)}]
		(6,3);
		\draw (4,3) to[curve through={(7,1)}] (10,2);
		\draw (0,1) to[in angle = 120, curve through={(3,2)}]
		(6,0);
		\draw (4,0) to[out angle = 60, curve through={(7,2)}] (10,1);
		\end{tikzpicture}
	}
	\subfloat[Garden]{
		\begin{tikzpicture}[scale=.4]
		\draw (-0.5,-0.5) to[quick curve through={(2,3)}]
		(-0.5,6.5);
		\draw (2.5,-0.5) to[quick curve through={(0,3)}]
		(2.5,6.5);
		\draw (-2,0) to[curve through={(2,0)(4,0)(5,0.5) (5,2) (5,0.5)}] (8,0);
		\end{tikzpicture}
	}
	
	\subfloat[Braid]{
		\begin{tikzpicture}[scale=.5]
		\draw (0,5) to (5,0);
		\draw (0,4.5) to (7,1.5);
		\draw (1,2) to (9,5);
		\draw (3,-1) to (8,5);
		\end{tikzpicture}
	}
	\subfloat[Cat]{
		\begin{tikzpicture}[scale = .5]
		\draw (-0.5,0.5) to[quick curve through={(2,3) (-1.5,6.5) (-1.5, 5.5) (-1.5,6.5)}]
		(-3.5,6.5);
		\draw (2.5,0.5) to[curve through={(0,3) (3.5,6.5) (3.5,5.5) (3.5,6.5)}]
		(5.5,6.5);
		\end{tikzpicture}
	}
	
	\caption{Admissible covers. The genus-2 components correspond to the thick dashed lines, and the genus-1 components correspond to the thick solid lines. The remaining components have genus 0.\label{figs:admissible-covers}}
\end{figure}
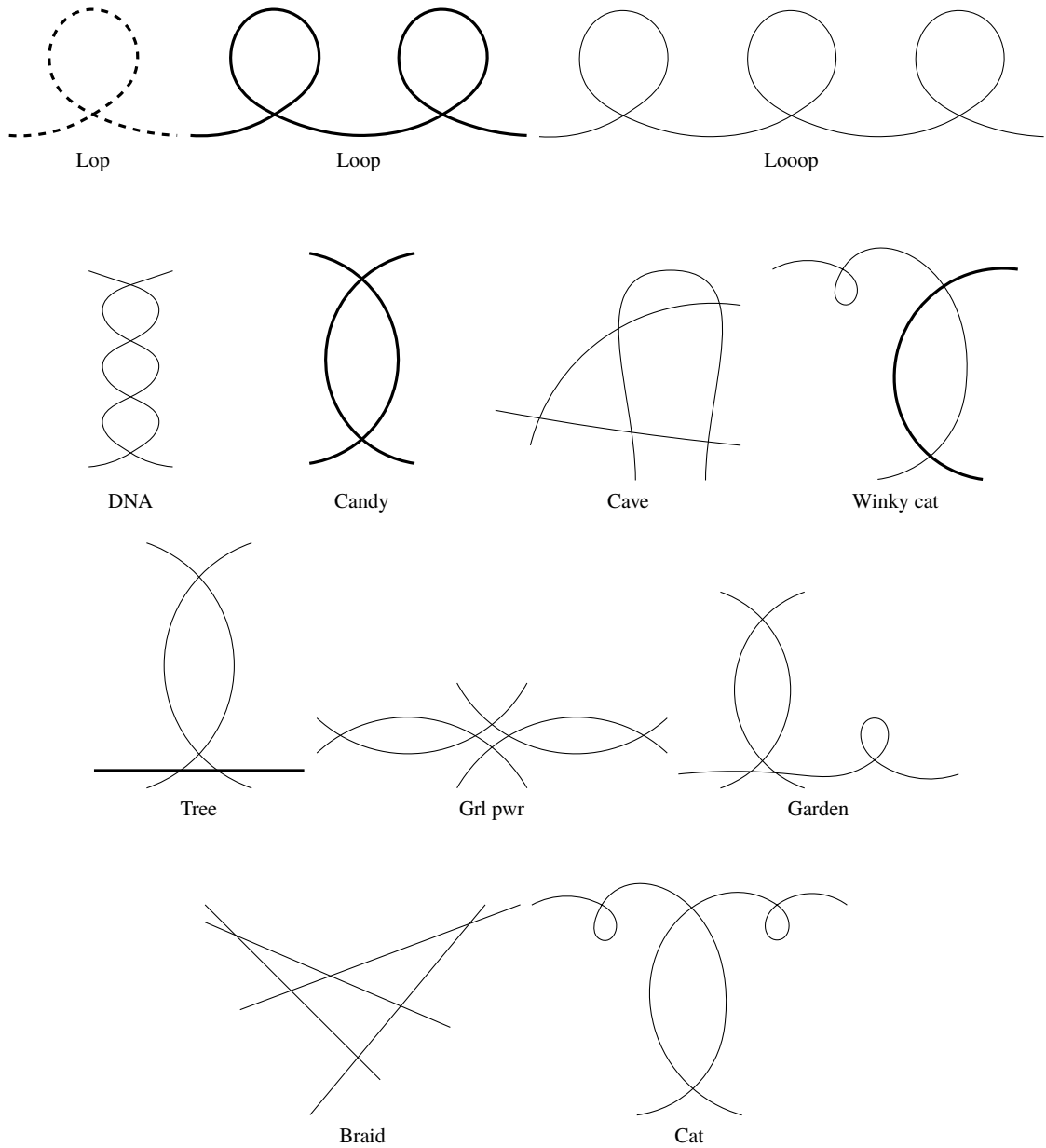

\end{document}